\documentclass[final,1p]{elsarticle}

\usepackage{amssymb}
\usepackage{amsmath}
\usepackage{amsthm}
 \usepackage{lineno}
 \usepackage{empheq} 

\usepackage{latexsym}
\usepackage[ruled,lined]{algorithm2e}
\usepackage{algorithmic}
\usepackage{array}
\usepackage[colorlinks,linkcolor=blue, anchorcolor=blue, citecolor=blue]{hyperref}
\usepackage{color,tabularx}
\usepackage{booktabs}
\usepackage{subfigure}
\usepackage{multirow}
\usepackage{stmaryrd}
\usepackage[T1]{fontenc}
\biboptions{numbers,sort&compress}
\usepackage{graphicx} 
\usepackage{adjustbox}

\newcommand{\f}{\frac}
\newcommand{\p}{\partial}

\newcommand{\mG}{\mathcal{G}}
\newcommand{\mQ}{\mathcal{Q}}
\newcommand{\mM}{\mathcal{M}}
\newcommand{\mbD}{\mathbb{ D }}
\numberwithin{equation}{section}
\allowdisplaybreaks
\newtheorem{theorem}{Theorem}[section]
\newtheorem{lemma}[theorem]{Lemma}
\newtheorem{corollary}[theorem]{Corollary}

\theoremstyle{definition}

\theoremstyle{remark}
\newtheorem{remark}{Remark}

\journal{XXXX}
\begin{document}

\begin{frontmatter}


\title{A variable time-step, second-order, and MBP-preserving linear stabilized scheme for the time-fractional Allen--Cahn equation}


\author[ouc]{Bingyin Zhang} \ead{byzhang2022@163.com}
\author[ouc]{Ao Zhang} \ead{zhangao6290@stu.ouc.edu.cn}
\author[ouc,lab]{Hongfei Fu\corref{cor1}} \ead{fhf@ouc.edu.cn}
\address[ouc]{School of Mathematical Sciences, Ocean University of China, Qingdao, Shandong 266100, China}
\address[lab]{Laboratory of Marine Mathematics, Ocean University of China, Qingdao, Shandong 266100, China}
\cortext[cor1]{Corresponding author.}

\begin{abstract}
In this paper, we present a second-order linear scheme based on the variable-step Alikhanov formula and central difference discretization for the time-fractional Allen--Cahn equation. The nonlinear potential is treated explicitly via a second-order extrapolation with preprocessing, which enables the discrete maximum-bound principle (MBP) to be preserved through an appropriate stabilization technique. Moreover, by developing a discrete fractional Gr\"onwall inequality together with the uniform boundedness of numerical solutions guaranteed by the MBP, we establish an $\alpha$-robust and optimal second-order maximum-norm error estimate under initial weak singularity assumption. In addition, energy stability is proved in the sense that the discrete original energy is uniformly bounded by the initial energy plus a high-order spatiotemporal correction term. Finally, extensive numerical experiments are presented to demonstrate the effectiveness of the proposed scheme.
\end{abstract}


\begin{keyword}
Time-fractional Allen--Cahn equation \sep Variable-step Alikhanov method \sep MBP 
\sep Discrete energy stability \sep Maximum-norm error estimate
\end{keyword}

\end{frontmatter}


\section{Introduction}

To model the motion of the anti-phase boundaries in crystalline solids, Allen and Cahn \cite{Acta_Allen_1979} proposed the classical Allen--Cahn equation
\begin{equation}\label{Model:AC}
\p_t \phi = \varepsilon^2 \Delta \phi + f( \phi ), \quad t >0, \  \mathbf{x} \in \Omega,
\end{equation}
subject to the initial condition $ \phi( \mathbf{x}, 0 ) = \phi_{\text{init}} ( \mathbf{x} ) $, where $\Omega \subset \mathbb{R}^{d} $ ($ d \le 3 $) is a smooth domain with boundary $ \partial \Omega $. The order parameter $ \phi $ represents the concentration difference in a binary system, $ 0 < \varepsilon \ll 1 $ is a small parameter characterizing the interfacial width. The nonlinear potential function $ F(\phi) $, satisfying $  f(\phi) = -F'(\phi) $, is assumed to be bistable; a typical example is the double-well potential $ F(\phi) = \frac{1}{4} ( 1 - \phi^{2} )^{2} $, see, e.g., \cite{JSC_Du_2020,JCP_Liao_2020}. Equipped with periodic or homogeneous Neumann boundary condition, the classical Allen--Cahn equation \eqref{Model:AC} can be viewed as an $ L^{2} $ gradient flow of a free energy $E[\phi] $, i.e.,
$$
\p_t \phi := - \frac{ \delta E }{ \delta \phi } \qquad \text{with} \  E[\phi] := \int_{\Omega} \Big( \frac{\varepsilon^2}{2} \vert \nabla \phi(\mathbf{x}) \vert^2 + F( \phi(\mathbf{x}) ) \Big) d \mathbf{x}.
$$
It is well known that for the classical Allen--Cahn equation \eqref{Model:AC}, the energy dissipation law $ E[\phi] (t) \leq E[\phi](s) $ for $ t > s $, and the maximum bound principle (MBP)
\begin{equation}\label{MBP}
	\text{if}~~ \max_{\mathbf{x} \in \bar{\Omega}} |\phi_{\text {init }}(\mathbf{x})| \leq 1 \quad \Longrightarrow 
	\quad \max_{\mathbf{x} \in \bar{\Omega}}|\phi(\mathbf{x},t)| \leq 1, \quad \forall t>0,
\end{equation}
hold. Since its inception, the model has been widely employed to simulate a broad range of phenomena, including interface dynamics in multiphase fluids \cite{CiCP_2012_Kim,SISC_2010_Shen,JCP_2006_Yang}, mean curvature flow \cite{NM_2003_Feng,JDG_1993_Ilmanen}, image segmentation \cite{ANM_2004_Benes,LiuQiaoZhang}, as well as numerous applications in materials science.

In heterogeneous porous media, diffusive transport is strongly influenced by interactions between the rock matrix and solute molecules. In such an environment, a significant fraction of solute molecules may become adsorbed within the micropores of the rock matrix \cite{MS_Sharma_2015}. The residence times of these adsorbed molecules can differ substantially from those of molecules in the mobile (bulk) phase \cite{JCP_Zhokn_2017,WRR_Schumer_2003}, leading to anomalous subdiffusive transport behavior. This regime is characterized by a sublinear growth of the particle mean square displacement $ < r^2 > $ with respect to $t$ \cite{PRL_Chepizhko_2013}. Such phenomena can be naturally described by subdiffusive (i.e., time-fractional) phase-field models; see, for example, \cite{CPC_Wang_2019,JCP_Wang_2017,SISC_Tang_2019}. The incorporation of nonlocal time-fractional derivative operators into phase-field equations has been shown to substantially alter the underlying diffusive dynamics, and thus has gained great research interests \cite{JCP_Wang_2017,SISC_Tang_2019,JCP_Zhang_2025,JSC_Huang_2023}. In this paper, we focus on the following time-fractional Allen--Cahn (tFAC) equation
\begin{equation}\label{Model:tAC}
	\begin{aligned}
		\p_t^{\alpha} \phi = \varepsilon^2 \Delta \phi + f( \phi ), \quad t >0, \  \mathbf{x} \in \Omega,
	\end{aligned}
\end{equation}
where the Caputo fractional derivative of order $ \alpha \in (0,1) $ is defined by
\begin{equation*}
	\p_t^{\alpha} v = \int^{t}_{0} \omega_{1-\alpha} (t-s)\partial_s v (s) ds,~~
	\omega_{\mu}(t) := \frac{t^{\mu-1}}{\Gamma(\mu)}.
\end{equation*}
The tFAC equation can also be interpreted as an $ L^{2} $  gradient flow in the fractional setting, namely, $ \p_t^{\alpha} \phi := - \frac{ \delta E }{ \delta \phi } $. It was shown in \cite{SISC_Tang_2019} that the tFAC equation obeys a ‘weak version’ of the classical energy dissipation law, i.e., $E[\phi](t) \leq E[\phi](0) $ for all $ t >0 $. Moreover, it also preserves the MBP \eqref{MBP}; see \cite{JSC_Du_2020,JCP_Zhang_2025}. Therefore, to ensure long-term stable numerical simulations and avoid nonphysical solutions for the modeling  of the tFAC model, it is highly desirable for numerical methods to preserve both of these inherent physical properties.

In \cite{SISC_Tang_2019}, the authors proposed a  first-order linear stabilized $L1$ scheme for the tFAC equation \eqref{Model:tAC}, which is both energy-stable and MBP-preserving. By discretizing the fractional derivative via the backward Euler convolution quadrature, Du et al. \cite{JSC_Du_2020} developed several structure-preserving time-discrete schemes, including weighted convex splitting scheme and a linear weighted stabilized scheme, and rigorously proved that these two schemes are MBP-preserving and fractional energy-stable, i.e., $E[\phi^n] \leq E[\phi^{n,\alpha}] $, where $\phi^{n,\alpha}$ is a convex combination of $\phi^0, \phi^1, \ldots, \phi^{n-1}$, and  meanwhile,  these time-stepping schemes are proved to achieve a convergence rate of order $O(\tau^{\alpha})$. In recent years, the widely used scalar auxiliary variable (SAV) approach \cite{JCP_Shen_2018} has also been extensively employed to construct high-order linear schemes that are unconditionally (modified) energy stable for time-fractional phase-field models, including but not limited to the tFAC equation; see, for example, \cite{SISC_Ji_2020,SISC_Hou_2021,SISC_Zhao_2024,JCP_Quan_2022}.

With the aim of developing much higher-order ($>$1) time-stepping schemes that preserve the discrete MBP, several approaches have been proposed and analyzed for the tFAC model \eqref{Model:tAC}, including the $L1$ scheme \cite{ACM_Ji_2020}, the $L1_{R}$ scheme \cite{SISC_Liao_2021}, the Alikhanov scheme  \cite{JCP_Liao_2020,JSC_Liao_2024}, and the SFTR (shifted fractional trapezoidal rule) scheme \cite{JSC_Huang_2023}. In particular, the authors in \cite{JCP_Liao_2020} proposed a nonlinear and nonuniform time-stepping Alikhanov scheme, and rigorously proved that it is second-order accurate in time and preserves the discrete MBP. More recently, by developing a novel discrete gradient structure for the Alikhanov formula via a local-nonlocal splitting technique \cite{JSC_Liao_2024}, this nonlinear scheme was further shown to satisfy a discrete energy dissipation law with respect to an asymptotically compatible energy. However, it is worth noting that all these mentioned high-order MBP-preserving schemes  are nonlinear. As a consequence, nonlinear iterative solvers must be employed at each time step, which inevitably increases the computational cost. In \cite{JCAM_Hu_2026}, a slightly different nonlinear Alikhanov scheme was proposed, together with a stabilized linear iterative method. The authors proved that the iterative scheme also preserves the MBP and converges to the solution of the underlying nonlinear scheme. Leveraging the exponential SAV approach, we proposed a linear, MBP-preserving second-order variable-step BDF2 scheme for the classical Allen--Cahn equation in \cite{JSC_ZFLX_2025}, in which a novel auxiliary functional and stabilization approach is developed; and further, in Ref. \cite{JCP_Zhang_2025} we presented some linear, nonuniform time-stepping $ L1 $ and Alikhanov schemes that are energy-stable and MBP-preserving. Nevertheless, this approach still suffers from two limitations: (i) the introduced SAV inherits the initial weak singularity of $\phi$, making the rigorous error analysis more challenging; and (ii) 
only 'modified' energy stable  can be guaranteed, which may cause a large discrepancy with the original energy for long-term simulations. These considerations motivate us to develop some fully linear, higher-order, and structure-preserving numerical schemes for the tFAC equation, together with rigorous theoretical analysis.

To this end, in this work we present a variable time-step second-order Alikhanov scheme by introducing a simple pre-processing procedure, which enjoys the following three main advantages:
\begin{itemize}
\item the proposed scheme is fully linear and preserves the discrete MBP, and is therefore computationally efficient;

\item under the graded temporal mesh with grading parameter $\gamma \ge 1$, the optimal temporal convergence order $\min\{2,\gamma\alpha\}$ can be rigorously established under the weak singularity regularity assumptions:
\begin{equation}\label{Assum:regu}
    \| \phi(t) \|_{W^{4,\infty}} \le C_{\phi}, \quad \| \partial_t^{\ell} \phi(t) \|_{W^{2,\infty}} \le C_{\phi} ( 1 + t^{\alpha-\ell} ) \quad \text{for} \   \ell= 1,2,3.
\end{equation}

\item the scheme is energy stable with respect to the original energy, in the sense that the discrete original energy can be uniformly bounded by the initial energy up to a high-order perturbation term.
\end{itemize}

The remainder of the paper is organized as follows. In the next section, we present some preliminaries. A variable time-step second-order linear scheme is developed in Section \ref{Sec:sch_mbp}, where the discrete MBP is established. Based on the proved discrete MBP, Section \ref{Sec:error_energy} presents maximum-norm error estimates and analyzes discrete energy stability. Extensive numerical experiments are presented in Section \ref{Sec:Numer}, and concluding remarks are offered in the final section. 
Throughout the paper, we denote by $C$ (with or without subscript) a generic $\alpha$-robust positive constant, which may depend on the given data, but is independent of the mesh sizes and can be different under different circumstances.
\section{Preliminaries}
In this section, we shall present some preliminaries.

\subsection{The nonuniform time-stepping Alikhanov formula}
Throughout the paper, we consider (generally nonuniform) time levels $ 0 = t_{0} < t_{1} < \cdots < t_{N} = T $ with time stepsize $ \tau_{k} := t_{k} - t_{k-1} $ for $ 1 \leq k \leq N $, and define the maximum stepsize $ \tau := \max_{1\leq k \leq N} \tau_{k} $. Denote the off-set time level $ t_{k-\theta} := ( 1 - \theta ) t_{k} + \theta t_{k-1} $ for $ k \geq 1 $. Let $ r_{k} := \tau_{k}/\tau_{k-1} (k\geq 2)$ be the adjacent time-step ratio such that $r_{*}\le r_{k} \le r^{*}$, where $ r_{*} $ and $ r^{*} $ are the minimum and maximum time-step ratio, respectively. Moreover, we assume the following condition:
\begin{description}
	 \item[(M1)] The minimum time-step ratio $ r_{*} = 4/7 $, and the maximum time-step ratio $ r^{*} $ is bounded.
\end{description}
This assumption means that one may employ a series of decreasing time-steps with the reduction factor down to $ r_{*} $ (which ensures some useful properties of the  Alikhanov formula; see Lemmas \ref{lem:L21_DC}--\ref{lem:L21_DC_positive}), and increasing time-steps with the amplification factor up to $ r^{*} $ (which affects the time-step restriction in error analysis and energy stability; see Theorems \ref{thm:Convergence} and \ref{thm:energy}). The use of nonuniform meshes allows us to adopt an efficient adaptive time-stepping strategy \cite{SISC_Qiao_2011,JCP_Liao_2020} to capture the dynamics process, admiting multiple time scales, governed by the tFAC model. Moreover, it is well known that solutions to subdiffusion problems usually exhibit weak singularities near the initial time, and are smooth away from $ t = 0 $, see \cite{IMA_Jin_2016,SINUM_Stynes_2017}, which thus may lead to a loss of accuracy in numerical simulations based upon  uniform temporal meshes. However, the special designed nonuniform temporal meshes, e.g., the graded mesh $ t_{k} = T( k/N )^{\gamma} $, can alleviate the defection.

For the sake of convergence analysis, we also need the following condition \cite{JCP_Liao_2020,SINUM_Mustapha_2014}:
\begin{description}
	\item[(M2)] For parameter $ \gamma \geq 1 $, there exist mesh-independent constants $ C_{1\gamma}, C_{2\gamma}, C_{3\gamma}, C_{4\gamma} > 0 $ such that $ \tau_{k} \leq \tau \min\{ 1, C_{1\gamma} t_{k}^{1-1/\gamma} \} $ for $ 1 \leq k \leq N $, $ t_{k} \leq C_{2\gamma} t_{k-1} $ for $ 2 \leq k \leq N $, with $ \tau_{1} \leq C_{3\gamma} N^{-\gamma} $ and $ \tau \leq C_{4\gamma} N^{-1} $.
\end{description}
Here, the parameter $ \gamma \geq 1 $ controls the extent to which the time levels are concentrated near the initial time $ t = 0 $. A simple example of a family of meshes satisfying {\bf M2} is the graded mesh.

Let $w^k = w(t_k)$, $w^{k-\theta} = w(t_{k-\theta})$ and $\bar{w}^{k,\theta} = ( 1 - \theta ) w^k + \theta w^{k-1} $. Define the temporal difference operator $\nabla_\tau v^k=v^k-v^{k-1}$ and $\mbD_{\tau} v^{k} = \nabla_\tau v^k/\tau_{k}$ for $k \geq 1$. The nonuniform time-stepping Alikhanov formula of the Caputo derivative at $ t = t_{n- \theta}  $ with $ \theta = \alpha/2 $ is given by \cite{JCP_Alikhanov_2015,JSC_Chen_2019,JCP_Liao_2020,JSC_Fu_2023}:
\begin{equation*}\label{Formula:L21}
	\begin{aligned}
		\p_t^{\alpha} w ( t_{n- \theta} ) \approx\mbD_\tau^\alpha w^n 
		:=  \sum_{k=1}^n A_{n-k}^{(n)} \nabla_\tau w^k,
	\end{aligned}
\end{equation*}
where the discrete convolution kernels $ A_{n-k}^{(n)} $ are defined as follows: $ A^{(1)}_{0} = a_{0}^{(1)} $ if $ n = 1 $, and for $ n \geq 2 $
\begin{equation*}\label{L21:kernel}
	A^{(n)}_{n-k} := \left\{
	\begin{aligned}
		& a^{(n)}_{0} + b^{(n)}_{1}/r_{n}, \qquad \qquad \qquad \qquad \qquad \    k = n,\\
		& a^{(n)}_{n-k} + b^{(n)}_{n-k+1}/r_{k} - b^{(n)}_{n-k}, \qquad \quad   2 \leq k \leq n-1, \\
		& a^{(n)}_{n-1} - b^{(n)}_{n-1}, \qquad  \qquad \qquad \qquad \qquad  \  \   \  \ k = 1,
	\end{aligned}
	\right.
\end{equation*}
with the discrete coefficients
\begin{equation*}
	\begin{aligned}
		a^{(n)}_{n-k} := \frac{1}{\tau_{k}} \int^{\min\{t_{k},t_{n- \theta}\}}_{t_{k-1}} \omega_{1-\alpha}(t_{n- \theta}-s) ds, \qquad 1 \leq k \leq n,
	\end{aligned}
\end{equation*}
\begin{equation*}
	\begin{aligned}
		b^{(n)}_{n-k} := \frac{2}{\tau_{k}\left(\tau_{k}+\tau_{k+1}\right)} \int^{t_{k}}_{t_{k-1}} \big(s-t_{k-\frac{1}{2}}\big)\omega_{1-\alpha}(t_{n- \theta}-s) ds, \qquad 1 \leq k \leq n-1.
	\end{aligned}
\end{equation*}
Furthermore, it is shown in \cite{CiCP_Liao_2021,JCP_Liao_2020} that the discrete convolution kernels satisfy the following properties, i.e.,
\begin{lemma}[\cite{JCP_Liao_2020}]\label{lem:L21_DC}
Under condition {\bf M1}, there holds
\begin{description}
	\item[(i)] the discrete kernels $A_{n-k}^{(n)}$ fulfill $A_0^{(n)} \leq \frac{24}{11 \tau_n} \int_{t_{n-1}}^{t_n} \omega_{1-\alpha}\left(t_n-s\right)ds $ and
	$$
	A_{n-k}^{(n)} \geq \frac{4}{11 \tau_n} \int_{t_{n-1}}^{t_n} \omega_{1-\alpha}\left(t_n-s\right) \mathrm{d} s, \quad 1 \leq k \leq n;
	$$
	
	\item[(ii)] the discrete kernels $A_{n-k}^{(n)}$ are monotone for $1 \leq k \leq n-1$, i.e.,
	$$
	A_{n-k-1}^{(n)}-A_{n-k}^{(n)} \geq\left(1+1/r_{k}\right) b_{n-k}^{(n)}-\frac{1}{5 \tau_k} \int_{t_{k-1}}^{t_k}\left(t_k-s\right) \omega_{-\alpha}\left(t_{n-\theta}-s\right) \mathrm{d} s>0;
	$$
	
	\item[(iii)] the first kernel $A_0^{(n)}$ is appropriately larger than the second one, i.e.,
	$$
	\frac{1-2 \theta}{1-\theta} A_0^{(n)} - A_1^{(n)}>0, \quad   n \geq 2.
	$$	
\end{description}
\end{lemma}

The following lemma is presented in \cite[Lemmas 2.1--2.2 \& Theorem 2.1]{JSC_Liao_2024}, which will be used to prove the discrete energy stability of the scheme to be proposed in next section.
\begin{lemma}[\cite{JSC_Liao_2024}]\label{lem:L21_DC_positive}	
	Under condition {\bf M1}, it holds
		$$	
		2(\nabla_\tau w^n)\,\mbD_\tau^\alpha w^n 	    
		\ge \mG[\nabla_\tau w^n ] - \mG[\nabla_\tau w^{n-1} ] + \frac{ 2 \alpha a_0^{(n)}}{2-\alpha}\left(\nabla_\tau w^n\right)^2,     
		\quad    n \geq 1,	
		$$
		where the nonnegative functionals $\mG$ is defined by	
		$$	
		\mG[v^n ] := \sum_{j=1}^{n-1}\bigl(B_{n-j-1}^{(n)}-B_{n-j}^{(n)}\bigr) \Bigl(\sum_{\ell=j+1}^n v^{\ell}\Bigr)^2	
		+  B_{n-1}^{(n)} \Bigl(\sum_{\ell=1}^n v^{\ell}\Bigr)^2, ~ n \geq 1;\quad \mG [v^0 ] = 0,	
		$$
		with $ B_{0}^{(1)} = 4(1-\alpha)a_{0}^{(1)}/(2-\alpha)$, $ B_{0}^{(n)} = 4(1-\alpha)a_{0}^{(n)}/(2-\alpha)  + 2 b_{1}^{(n)}/r_{n} $ for $n \ge 2$, and $ B_{n-k}^{(n)} = A_{n-k}^{(n)} $ for $ 1 \leq k \leq n-1 $. 
\end{lemma}

\begin{remark}
Let $ R_{*} = R_{*} ( \alpha ) $ be the unique positive root of the nonlinear equation	
\begin{equation*}\label{lemC:root}		
	2\sqrt{ \frac{ 2 ( 1 - \alpha/2 ) R_{*} }{ 1 + \alpha + ( 1 - \alpha/2 ) R_{*} } + \frac{ R_{*} }{ 1 + R_{*} } } + 3 - \frac{1}{ R_{*}^{2} ( 1 + R_{*} ) } = 0 \quad \text{for} \  \alpha \in (0,1).	
\end{equation*}	
Then, it holds that $ 0.3865 \approx R_{*}(0) < R_{*}( \alpha ) < R_{*}(1) \approx 0.4037 $ for all $ \alpha \in (0,1) $; see \cite[Lemma 2.1]{JSC_Liao_2024}. In fact, for a fixed $ \alpha \in (0,1) $, if the adjacent time-step ratios satisfy a weaker condition, namely $ r_{k} \geq R_{*}(\alpha) $ for $ 2 \leq k \leq n $, the conclusion of Lemma \ref{lem:L21_DC_positive} still holds; see \cite{JSC_Liao_2024} for more details. However, since the subsequent energy stability analysis relies on the MBP-preserving property (and hence invokes condition {\bf M1}), we adopt condition {\bf M1} throughout the paper for the sake of consistency.
\end{remark}

\subsection{ Central finite difference method }
For simplicity of presentation, the discussion is restricted to the two-dimensional square domain $ \Omega = (0,L)^{2} $ with periodic boundary conditions. The associated analysis can be readily extended to the three-dimensional case and/or homogeneous Neumann boundary condition.
Given a positive integer $ M $, let $ h = L/M $ be the spatial grid length and set $ \Omega_{h} := \{ \mathbf{x}_{h} = ( ih, jh ) \mid 0 \leq i,j \leq M \} $. Let $ \mathbb{V}_h $ be the set of all $M$-periodic real-valued grid functions on $ \Omega_{h} $, i.e.,
$$
\mathbb{V}_h:=\left\{v \mid v= \{v_{i,j}\}_{i,j=1}^{M}~ \text{and}~  v ~ \text{is periodic}\right\} .
$$ 
Define the following discrete inner product, and discrete $ L^2 $ and $ L^{\infty} $ norms
$$
\langle v, w\rangle=h^2 \sum_{i, j=1}^M v_{i j} w_{i j}, \quad\|v\|=\sqrt{\langle v, v\rangle}, \quad\|v\|_{\infty}=\max _{1 \leq i, j \leq M} |v_{i j}|
$$
for any $v, w \in \mathbb{V}_h$, and
$$
\langle \mathbf{v}, \mathbf{w}\rangle = \langle v^1, w^1\rangle + \langle v^2, w^2\rangle,  \quad \|\mathbf{v}\|=\sqrt{\langle\mathbf{v}, \mathbf{v}\rangle}
$$
for any $\mathbf{v} = (v^1, v^2)^\top, \mathbf{w}= (w^1, w^2)^\top \in \mathbb{V}_h \times \mathbb{V}_h$. Moreover, for the underlying MBP property \eqref{MBP}, we introduce the following maximum-norm function space, i.e.,
$$
\mathbb{V}_{\text{mbp}} := \left\{v \mid v \in \mathbb{V}_{h} ~ \text{with} ~ \| v \|_{\infty} \leq 1 \right\} .
$$ 

In this work, we consider a central finite difference discretization for the spatial differential operators. For any $ v \in \mathbb{V}_h $, the discrete Laplace operator $\Delta_h$ is defined by
$$
\Delta_h v_{i j}=\frac{1}{h^2}\left(v_{i+1, j}+v_{i-1, j}+v_{i, j+1}+v_{i, j-1}-4 v_{i j}\right), \quad 1 \leq i, j \leq M,
$$
and the discrete gradient operator $\nabla_h$ is defined by
$$
\nabla_h v_{i j}=\left(\frac{v_{i+1, j}-v_{i j}}{h}, \frac{v_{i, j+1}-v_{i j}}{h}\right)^\top, \quad 1 \leq i, j \leq M .
$$
Notice that $ \mathbb{V}_h $ is a finite-dimensional linear space, thus any grid function in $ \mathbb{V}_h $ and any linear operator $ P: \mathbb{V}_h \rightarrow \mathbb{V}_h $ can be treated as a vector in $ \mathbb{R}^{M^2} $ and a matrix
in $ \mathbb{R}^{M^2 \times M^2 } $, respectively.

\begin{lemma}[\cite{JCM_Tang_2016}]\label{lem:MBP_left} 
	For any $ a > 0 $, we have $ \| ( a I - \Delta_{h} )^{-1} \|_{\infty} \leq a ^{-1} $, where $ I $ represents the identity operator.
\end{lemma}

\begin{lemma}[\cite{SIREV_Du_2021}]\label{lem:MBP_right} 
	If $\kappa \geq\left\|f^{\prime}\right\|_{C[-1, 1]}$ holds for some positive constant $\kappa$, then we have $|f( \xi ) + \kappa \xi | \leq \kappa $ for any $\xi \in[-1, 1]$.
\end{lemma}

\subsection{Discrete fractional Gr\"onwall inequality}
Define a sequence of discrete complementary convolution kernels $ \{ P^{(n)}_{n-j} \}_{j=1}^{n} $ by
$$
P^{(n)}_{0} := \frac{1}{ A^{(n)}_{0} }, \quad P^{(n)}_{n-j} := \frac{1}{ P^{(j)}_{0} } \sum^{n}_{k=j+1} ( A^{(k)}_{k-j-1} - A^{(k)}_{k-j} ) P^{(n)}_{n-k}, \quad 1 \leq j \leq n-1,
$$
such that the discrete kernels $ P^{(n)}_{n-j} \geq 0 $ fulfill
\begin{equation}\label{kernel:comple}
	\sum^{n}_{j=k} P^{(n)}_{n-j} A^{(j)}_{j-k} = 1 \quad \text{for} \   1 \leq k \leq n \leq N,
\end{equation}
and
\begin{equation}\label{kernel:bound}
	P^{(n)}_{0} \leq \frac{11}{4} \Gamma(2-\alpha) \tau_{n}^{\alpha}, \qquad  \sum^{n}_{j=1} P^{(n)}_{n-j} \omega_{1+m\alpha-\alpha}(t_{j}) \leq \frac{11}{4}\omega_{1+m\alpha}(t_{n}), \quad m = 0,1,
\end{equation}
for $ 1 \leq n \leq N $, see \cite{JCP_Liao_2020,SINUM_Liao_2019}. The following lemma also holds.
\begin{lemma}[\cite{SINUM_Liao_2019}]\label{lem:ML}
	Assume condition {\bf M1} holds. For any real parameter $ \mu > 0 $, we have
	$$
	\sum^{n-1}_{j=1} P^{(n)}_{n-j} E_{\alpha} ( \mu t_{j}^{\alpha} ) \leq \frac{77}{16} \frac{ E_{\alpha} ( \mu t_{n}^{\alpha} ) - 1 }{ \mu }, \quad 1 \leq n \leq N,
	$$
	  where $ E_{\alpha} (z) := \sum^{\infty}_{k=0} \frac{ z^{k} }{ \Gamma( 1 + k \alpha ) } $ is the single-parameter Mittag--Leffler function.
\end{lemma}

Next, we establish an $\alpha$-robust discrete fractional Gr\"onwall inequality that plays an important role in the subsequent maximum-norm error analysis.
\begin{lemma}\label{thm:Gron} Assume condition {\bf M1} holds. Let $ \{ \lambda_{k} \}_{k=0}^{N} $ be nonnegative constants with $ 0 < \sum^{N}_{k=0} \lambda_{k} \leq \Lambda $, where $ \Lambda > 0 $ is some constant (independent of the time stepsize). Suppose that the nonnegative sequences $ \{ \xi^k \}_{k=1}^{N} $ and $ \{ \eta^k \}_{k=1}^{N} $ are bounded and the nonnegative sequence $ \{ w^{k} \}_{k=0}^{N} $ satisfies
	\begin{equation}\label{Condi:Gron}
		\sum^{n}_{k=1} A^{(n)}_{n-k} \nabla_{\tau} w^{k} \leq \sum^{n}_{k=0} \lambda_{n-k} w^{k} + \xi^{n} + \eta^{n} \quad \text{for} \  1 \le n \le N.
	\end{equation}
	Then, if $ \tau_{n} \leq \sqrt[\alpha]{4/(33\Gamma(2-\alpha)\Lambda)} $, it holds that
	\begin{equation}\label{Conclusion1:Gron}
		w^{n}  \leq C_{n,\alpha} \big( w^{0} + \max_{1\leq k\leq n} \sum^{k}_{j=1} P^{(k)}_{k-j} ( \xi^{j} + \eta^{j} ) \big),
	\end{equation}
	where $C_{n,\alpha}:=2 E_{\alpha} ( 77/8~ \Lambda t_{n}^{\alpha} )$ is $\alpha$-dependent, but is robust when $\alpha\to 1$, and further, we have
	\begin{equation}\label{Conclusion2:Gron}
		w^{n} \leq  C_{n,\alpha} \big( w^{0} + \max_{1\leq k\leq n} \sum^{k}_{j=1} P^{(k)}_{k-j} \xi^{j} + \frac{11}{4} \omega_{1+\alpha}( t_{n} ) \max_{1\leq k \leq n } \eta^{k} \big).
	\end{equation}
\end{lemma}
\begin{proof} The conclusion \eqref{Conclusion2:Gron} is a consequent result of \eqref{Conclusion1:Gron} by applying \eqref{kernel:bound} with $ m =1 $. Therefore, it remains to verify the conclusion \eqref{Conclusion1:Gron}. We replace the index $ n $ with $ j $ in \eqref{Condi:Gron}, then multiply by $ P^{(n)}_{n-j} $ and sum over $ j $ to obtain
	\begin{equation}\label{Gron:1}
		\sum_{j=1}^{n} P^{(n)}_{n-j} \sum^{j}_{k=1} A^{(j)}_{j-k} \nabla_{\tau} w^{k} \leq \sum_{j=1}^{n} P^{(n)}_{n-j} \sum^{j}_{k=0} \lambda_{j-k} w^{k} + \sum_{j=1}^{n} P^{(n)}_{n-j} ( \xi^{j} + \eta^{j} ) .
	\end{equation}
	Exchanging  the order of summation and using the identity relation \eqref{kernel:comple}, we get
	\begin{equation*}
		\sum_{j=1}^{n} P^{(n)}_{n-j} \sum^{j}_{k=1} A^{(j)}_{j-k} \nabla_{\tau} w^{k} = \sum_{k=1}^{n} \nabla_{\tau} w^{k} \sum^{n}_{j=k} P^{(n)}_{n-j} A^{(j)}_{j-k} = \sum_{k=1}^{n} \nabla_{\tau} w^{k} = w^{n} - w^{0} .
	\end{equation*}
	Thus, it follows from \eqref{Gron:1} that
	\begin{equation}\label{Gron:2}
		w^{n}  \leq w^{0} + \sum_{j=1}^{n} P^{(n)}_{n-j} \sum^{j}_{k=0} \lambda_{j-k} w^{k} + \sum_{j=1}^{n} P^{(n)}_{n-j} ( \xi^{j} + \eta^{j} ) .
	\end{equation}
	Denote
	$
    G_{n} := w^{0} + \max_{1\leq k \leq n} \sum^{k}_{j=1} P^{(k)}_{k-j} ( \xi^{j} + \eta^{j} ),
	$
	then the claimed estimate \eqref{Conclusion1:Gron} can be rewritten as 
	\begin{equation}\label{Gron:3}
		w^{n} \leq C_{n,\alpha} G_{n} \quad \text{for} \  1 \leq n \leq N. 
	\end{equation}
 It is clear that $ G_{n} $ is monotonous for $ n \geq 1 $, i.e., $ G_{n} \geq G_{n-1} $. Moreover, since the Mittag--Leffler function satisfies $ E_{\alpha}(0) = 1 $ and $ E_{\alpha}'(z) > 0 $ for all real $ z > 0 $, we have $ C_{n,\alpha} \geq C_{n-1,\alpha} \geq 2 $ for $ n \geq 1 $. 
    
    Now, we use the mathematical induction to complete the proof of \eqref{Conclusion1:Gron}. First, taking $ n = 1 $ in \eqref{Gron:2}, one has
	\begin{equation*}
		w^{1}  \leq w^{0} + P^{(1)}_{0} ( \lambda_{1} w^{0} + \lambda_{0} w^{1} ) + P^{(1)}_{0} ( \xi^{1} + \eta^{1} ) .
	\end{equation*}
	From \eqref{kernel:bound} and the given restriction on the time stepsize, it holds that
	$$
	P^{(1)}_{0} \lambda_{i} \leq \frac{11}{4} \Gamma(2-\alpha) \tau_{1}^{\alpha} \Lambda \leq  \frac{1}{3}, \quad i=0,1,
	$$
	which, together with the fact $ P^{(1)}_{0} \geq 0 $, implies that
	$$
	w^{1} \leq \frac{3}{2} \left( \frac{4}{3} w^{0} + P^{(1)}_{0} ( \xi^{1} + \eta^{1} ) \right) \leq 2 G_{1} \leq C_{1,\alpha} G_{1}.
	$$
	
	Next, assume \eqref{Gron:3} holds for all $1 \leq k \leq n-1$, we shall prove it also holds for $k=n$. Now, choose some $ k(n) $ such that $ w^{ k(n) } = \max_{0\leq j \leq n-1} w^{j} $. If $ w^{n} \leq w^{ k(n) } $, then due to the monotonicity of $ C_{k,\alpha} $ and $ G_{k} $, we have
	$$
	w^{n} \leq w^{ k(n) } \leq C_{ k(n),\alpha } G_{ k(n) } \leq C_{ n,\alpha } G_{ n },
	$$
	as required. Otherwise, if $ w^{n} > w^{ k(n) } $, we deduce from \eqref{Gron:2}, assumption \eqref{Gron:3} and  the monotonicity of $ C_{k,\alpha} $ and $ G_{k} $ that
	\begin{equation*}
		\begin{aligned}
			w^{n} & \leq G_{n} + \sum_{j=1}^{n-1} P^{(n)}_{n-j} \sum^{j}_{k=0} \lambda_{j-k} w^{k} + P^{(n)}_{0} \sum^{n}_{k=0} \lambda_{n-k} w^{k} \\
			&\leq G_{n} + \sum_{j=1}^{n-1} P^{(n)}_{n-j} \sum^{j}_{k=0} \lambda_{j-k} C_{k,\alpha} G_{k} + P^{(n)}_{0} \sum^{n}_{k=0} \lambda_{n-k} w^{k} \\
			&\leq G_{n} + \Lambda \sum_{j=1}^{n-1} P^{(n)}_{n-j} C_{j,\alpha} G_{j} + P^{(n)}_{0} \Lambda w^{n}.
		\end{aligned}
	\end{equation*}
	Then, the time stepsize restriction $ \tau_{n} \leq \sqrt[\alpha]{4/(33\Gamma(2-\alpha)\Lambda)} $ implies that
	\[P^{(n)}_{0} \Lambda \le \frac{11}{4} \Gamma(2-\alpha)\Lambda \tau_{n}^{\alpha}  \le \frac{1}{3},\]
	and thus
	\begin{equation*}
		\begin{aligned}
			w^{n} \leq \frac{3}{2} G_{n} + \frac{3}{2} \Lambda \sum_{j=1}^{n-1} P^{(n)}_{n-j} C_{j,\alpha} G_{j} \leq \frac{3}{2} G_{n} + 3 \Lambda G_{n} \sum_{j=1}^{n-1} P^{(n)}_{n-j} E_{\alpha} ( 77/8~ \Lambda t_{j}^{\alpha} ).
		\end{aligned}
	\end{equation*}
	Finally, by Lemma \ref{lem:ML} with $ \mu = 77/8 \Lambda $, we have
	\begin{equation*}
		\begin{aligned}
			w^{n} \leq \frac{3}{2} G_{n} + 3 \times \frac{77}{16} \Lambda G_{n} \frac{ E_{\alpha} ( 77/8 \Lambda t_{n}^{\alpha} ) - 1 }{ 77/8 \Lambda } = \frac{3}{2} E_{\alpha} ( 77/8 \Lambda t_{n}^{\alpha} ) G_{n} \leq C_{n,\alpha} G_{n},
		\end{aligned}
	\end{equation*}
	which completes the induction and thus the proof.
\end{proof}

\section{MBP-preserving stabilized IMEX-Alikhanov scheme }\label{Sec:sch_mbp}

In this section, we construct a second-order, MBP-preserving linear stabilized IMEX-Alikhanov scheme on nonuniform temporal grids for model \eqref{Model:tAC} as follows.

Let $\widehat{\phi}^{n,\theta}=\phi^{0} $ for $ n = 1 $, and for $ n \geq 2 $, $ \widehat{\phi}^{n,\theta} $ is generated by the following second-order extrapolation formula equipped with a simple cut-off, namely,
\begin{equation}\label{sch:L21_n:e1}
	\widehat{\phi}^{n,\theta} = \min\left\{ \max\left\{ ( 1 + (1-\theta) r_{n} ) \phi^{n-1} - (1-\theta) r_{n} \phi^{n-2}  , - 1 \right\}, 1 \right\}.
\end{equation}
Then the stabilized IMEX scheme reads as finding $ \phi^{n} \in \mathbb{V}_h $ such that
\begin{equation}\label{sch:L21_n}
	\mathbb{D}^{\alpha}_{\tau} \phi^{n}  = \varepsilon^2 \Delta_h \bar{\phi}^{n,\theta} + f(\widehat{\phi}^{n,\theta}) - \kappa ( \bar{\phi}^{n,\theta} - \widehat{\phi}^{n,\theta} ), \quad 1 \le n \le N,
\end{equation}
where $ \kappa \geq 0 $ is a stabilization constant to be determined. 
Equivalently, the scheme \eqref{sch:L21_n} can be rewritten as follows:
\begin{equation}\label{sch:L21_1_1}
	\begin{aligned}
		& \mQ^{1}\phi^{1}  
		=   \mM^{1} \phi^{0}  + f(\phi^{0}) + \kappa \phi^{0}, \quad n=1;
	\end{aligned}
\end{equation}
\begin{equation}\label{sch:L21_n_1}
	\begin{aligned}
		& \mQ^{n} \phi^{n} 
		 =  \mM^{n} \phi^{n-1} + \Xi^{n-2} (\phi) + f(\widehat{\phi}^{n,\theta})  + \kappa \widehat{\phi}^{n,\theta}, \quad n \ge 2,
		\end{aligned}
	\end{equation}
where the operators 
$$ 
\mQ^{n} := ( A^{(n)}_{0} + \kappa ( 1 - \theta ) ) I - ( 1 - \theta ) \varepsilon^2  \Delta_h,~ n \ge 1; 
$$ 
$$ 
\mM^{1} := ( A^{(1)}_{0} - \kappa \theta ) I + \theta \varepsilon^2  \Delta_h,\quad 
\mM^{n} := ( A_{0}^{(n)} - A_{1}^{(n)} - \kappa \theta ) I + \theta \varepsilon^2 \Delta_h,~ n \ge 2;
$$ 
$$
		\Xi^{n-2}[\phi] := \sum_{k=1}^{n-2} \big( A_{n-k-1}^{(n)} - A_{n-k}^{(n)} \big) \phi^k + A_{n-1}^{(n)} \phi^0, ~ n \ge 2.
$$

Since $ A^{(n)}_{0} > 0 $ for all $ n \geq 1 $, the operator $ \mQ^{n} $ is self-adjoint and positive definite. Consequently, the scheme \eqref{sch:L21_1_1}--\eqref{sch:L21_n_1} or originally \eqref{sch:L21_n} is uniquely solvable. Next, we are ready to establish the discrete MBP for the proposed scheme.

\begin{theorem}\label{thm:MBP_L21_n} Assume condition {\bf M1} and the stabilization parameter $
	\kappa \geq \left\|f^{\prime}\right\|_{C[-1, 1]}
	$ hold. If the time stepsize satisfies
	\begin{equation}\label{MBP:L21_tau}
		\tau \leq \sqrt[\alpha]{ 4/( 11 ( 1 - \theta ) ( 4 \varepsilon^{2}/h^2 + \kappa ) \Gamma( 2 - \alpha ) ) },
	\end{equation}
	the stabilized IMEX-Alikhanov scheme \eqref{sch:L21_n} preserves the MBP for $\left\{\phi^n\right\}$, i.e., the discrete version of MBP is valid:
	\begin{equation*}\label{MBP:L21}
		\left\|\phi_{\text {init }}\right\|_{\infty} \leq 1 \quad  \Longrightarrow \quad \left\|\phi^n\right\|_{\infty} \leq 1, \quad \forall n \geq 1.
	\end{equation*}
\end{theorem}
\begin{proof} This claim will be verified by the complete mathematical induction argument. For $ n = 1 $, 
the applications of Lemmas \ref{lem:MBP_left}--\ref{lem:MBP_right}  to \eqref{sch:L21_1_1} imply
	\begin{equation}\label{Formu:L21SAV_MBP_1}	
		( A^{(1)}_{0} + \kappa ( 1 - \theta ) ) \| \phi^{1} \|_{\infty}      
		\le  \| \mM^{1} \phi^{0}  \|_{\infty} + \kappa.
	\end{equation}
	Moreover, under the condition $\tau_1  \leq 1/\sqrt[\alpha]{ ( 4 \varepsilon^{2}/h^2 + \kappa )\theta \Gamma( 2 - \alpha ) }$, the diagonal elements
    $A^{(1)}_{0} -( 4 \varepsilon^{2}/h^2 + \kappa )\theta$  of $\mM^{1}$ are nonnegative,     
    and thus
	$
	\|\mM^{1} \|_{\infty} = A_0^{(1)} - \kappa \theta,
	$
	which yields $ 
	\|\mM^{1} \phi^{0} \|_{\infty}    
	\le \|\mM^{1} \|_{\infty} \|\phi^{0} \|_{\infty}    
	\le  A_0^{(1)} - \kappa \theta
	$.
	Inserting this estimate into \eqref{Formu:L21SAV_MBP_1}, we obtain
	$ \| \phi^{1} \|_{\infty} \leq 1 $.

	Now, suppose that $ \|\phi^k\|_{\infty} \leq 1 $ holds for all $ 0 \leq k \leq n-1 $ and some $n \ge 2$. First, by \eqref{sch:L21_n:e1} we see  $ \| \widehat{\phi}^{n,\theta} \|_{\infty} \leq 1 $. Next, we shall show that $ \|\phi^k\|_{\infty} \leq 1 $ holds for $k= n$. 
From Lemma \ref{lem:L21_DC} (i) and (iii), we know that
	$$
	A_0^{(n)} - A_1^{(n)}  > \frac{ \theta }{ 1- \theta } A_0^{(n)} \geq \frac{ 4 \theta \tau_{n}^{-\alpha} }{ 11(1- \theta) \Gamma(2-\alpha) } \geq \kappa \theta + \frac{4 \theta \varepsilon^2}{h^2},
	$$
  where the time stepsize condition \eqref{MBP:L21_tau} has been used in the last step. Thus, all elements of $\mM^{n}$ are positive and
	$ \|\mM^{n}\|_{\infty} = A_0^{(n)} - A_1^{(n)} - \kappa \theta $, which together with the hypothesis $ \|\phi^{n-1}\|_{\infty} \leq 1 $ yields
	\begin{equation}\label{Formu:L21_n_MBP_3}
		\|\mM^{n} \phi^{n-1}\|_{\infty} \le \|\mM^{n} \|_{\infty} \|\phi^{n-1} \|_{\infty} \leq A_0^{(n)} - A_1^{(n)} - \kappa \theta.
	\end{equation}
	Moreover, by the positivity and monotonicity of the discrete kernels $ \{ A^{(n)}_{n-k} \} $ and the assumption that $\|\phi^{k}\|_{\infty} \leq 1, 0 \leq k \leq n-1$, the second right-hand side term of \eqref{sch:L21_n_1} can be estimated by
	\begin{equation}\label{Formu:L21_n_MBP_4}
		\| \Xi^{n-2} [\phi] \|_{\infty} \leq \sum^{n-2}_{k=1} \big( A^{(n)}_{n-k-1} - A^{(n)}_{n-k} \big) \| \phi^{k} \|_{ \infty } + A^{(n)}_{n-1} \| \phi^{0} \|_{ \infty } \leq A_{1}^{(n)}.
	\end{equation}
	Collecting the estimates \eqref{Formu:L21_n_MBP_3}--\eqref{Formu:L21_n_MBP_4} and  applying Lemmas \ref{lem:MBP_left}--\ref{lem:MBP_right}, we deduce from \eqref{sch:L21_n_1} that
	\begin{equation}\label{Formu:L21_n_MBP_5}
		\begin{aligned}
			( A^{(n)}_{0} + \kappa ( 1 - \theta ) ) \| \phi^{n} \|_{\infty} 
             \leq \|\mQ^{n} \phi^{n} \|_{\infty} 
			  & \le \|\mM^{n} \phi^{n-1}\|_{\infty} + \| \Xi^{n-2} [\phi]  \|_{\infty} 
            +  \|  f(\widehat{\phi}^{n,\theta}) + \kappa \widehat{\phi}^{n,\theta}  \|_{\infty} \\ 
            & \leq A^{(n)}_{0} + \kappa ( 1 - \theta ),
		\end{aligned}
	\end{equation}
	which immediately yields $ \| \phi^{n} \|_{\infty} \leq 1$, thus completing the proof.
\end{proof}

\section{ Maximum-norm error estimate and discrete energy stability }\label{Sec:error_energy}

In this section, we establish an $\alpha$-robust maximum-norm error estimate and the discrete energy stability for the stabilized IMEX-Alikhanov scheme. To this end, we employ a specific graded mesh defined by $ t_{k} = T( k/N )^{\gamma} $ that satisfies conditions \textbf{M1} and \textbf{M2}.

We begin by recalling the bounds for the truncation errors $ R_{1}^{n} [w] := \p_t^{\alpha} w(t_{n-\theta}) -\mbD^{\alpha}_{\tau} w^{n}$ and $ R_{2}^{n} [w] := w(t_{n-\theta}) - \bar{w}^{n,\theta}$, which are summarized in the following lemmas.
\begin{lemma}[\cite{IMA_Chen_2021}]\label{lem:L21_robust} If $w(t)$ satisfies $\vert \partial_{t}^{\ell} w \vert \leq C_w (1+t^{\alpha - \ell}) \text{ for } 1\leq \ell \leq 3 $, the following estimate holds for $N \geq 3$, $\zeta_{N} = 1/(\ln N)$ and some positive constant $C_0$
	\begin{equation*}
		\begin{aligned}
			\sum^{n}_{j=1} P^{(n)}_{n-j} \vert R_{1}^{j} [w] \vert \leq C_0 \frac{11 e^{\gamma} \Gamma(1+\zeta_{N}-\alpha)}{4\Gamma(1+\zeta_{N})} T^{\alpha} \left(\frac{t_{n}}{T}\right)^{\zeta_{N}} N^{-\min\{\gamma\alpha, 3-\alpha\}}, \quad   1 \leq n \leq N.
		\end{aligned}
	\end{equation*}
\end{lemma}

\begin{lemma}[\cite{JSC_Fu_2023,CiCP_Liao_2021}]\label{lem:weight_err} If $w(t)$ satisfies $\vert \partial_{t}^{\ell} w \vert \leq C_w (1+t^{\alpha - \ell} ) \text{ for } 1\leq \ell \leq 2 $, then it holds
$$
	\left \vert R_{2}^{n} [w] \right \vert \leq C_1 N^{-\min\{\gamma\alpha,2\}}, \quad   1 \leq n \leq N,
$$ 
where $C_1$ is an $\alpha$-robust positive constant.
\end{lemma}

Moreover, denote the linear extrapolation error as
$
R_{3}^{n} [w] := w( t_{n-\theta} ) - \widetilde{w}( t_{n-\theta} ),
$
where $ \widetilde{w}( t_{n-\theta} ) :=w( t_{0} )$ for $n=1$, and $ \widetilde{w}( t_{n-\theta} ) := ( 1 + (1-\theta) r_{n} ) w( t_{n-1} ) - (1-\theta) r_{n} w( t_{n-2} )$ for $n \geq 2$.

\begin{lemma}\label{lem:extra_err} 
	If $w(t)$ satisfies $\vert \partial_{t}^{\ell} w \vert \leq C_w (1+t^{\alpha - \ell} ) \text{ for } 1\leq \ell \leq 2 $, then it holds
	\begin{equation*}\label{Conclu:extra}
	\left \vert R_{3}^{n} [w] \right \vert \leq C_2  N^{-\min\{\gamma\alpha,2\}}, \quad   1 \leq n \leq N,
    \end{equation*}
 where $C_2$ is an $\alpha$-robust positive constant.
\end{lemma}
\begin{proof} For $ n = 1 $, by the Newton-Leibniz formula and the regularity assumption on $w$, we have
\begin{equation*}
\vert R^{1}_{3} [w] \vert =  \left\vert \int^{ t_{1-\theta} }_{ t_{0} } w^{\prime}(s) ds \right\vert \leq \int^{ t_{1} }_{ t_{0} } \vert w^{\prime}(s) \vert ds \leq C_w \int^{ t_{1} }_{ t_{0} } (1 + s^{\alpha -1}) d s \leq C_w ( \tau_{1} + \tau_{1}^{\alpha} ),
\end{equation*}
which further implies
\begin{equation*}
\vert R^{1}_{3} [w] \vert \leq C_w C_{3\gamma} ( N^{-\gamma} + N^{-\gamma\alpha} ).
\end{equation*}

Next, for $ n \geq 2 $, the triangle inequality gives
\begin{equation}\label{Trun:1}
\begin{aligned}
\vert R^{n}_{3} [w] \vert
& \leq  \vert R_{2}^{n} [w] \vert + \vert \bar{w}^{n,\theta} - \widetilde{w}( t_{n-\theta} ) \vert \\
& = \vert R_{2}^{n} [w] \vert + ( 1 - \theta ) \vert w( t_{n} ) - ( ( 1+r_{n} ) w ( t_{n-1} ) - r_{n} w ( t_{n-2} ) ) \vert.
\end{aligned}
\end{equation}
According to \cite[Lemma 2.3]{JCP_Zhang_2025}, the second term on the right-hand side of \eqref{Trun:1} can be estimated as
$$
 \vert w( t_{n} ) - ( ( 1+r_{n} ) w ( t_{n-1} ) - r_{n} w ( t_{n-2} ) ) \vert \leq C_{w,\gamma} N^{-\min\{\gamma\alpha,2\}}.
$$
Applying this estimate and Lemma~\ref{lem:weight_err} to \eqref{Trun:1} yields the desired result for $ n \geq 2 $. Let $C_2:=\max\{2C_w C_{3\gamma}, C_1+ ( 1 - \theta )C_{w,\gamma}\}$. This completes the proof.
\end{proof}

\subsection{Maximum-norm error estimate}
Taking the advantage of the discrete MBP established in Theorem \ref{thm:MBP_L21_n} and applying the discrete fractional Gr\"onwall inequality presented in Lemma \ref{thm:Gron}, we can derive the maximum-norm error estimate for $ e^{n} := \phi(t_{n}) - \phi^{n} $ with $ 0 \leq n \leq N $, without assuming the global Lipschitz continuity of the nonlinear potential $f$.

\begin{theorem}\label{thm:Convergence} 
  Suppose that the solution of \eqref{Model:tAC} satisfies the regularity assumption \eqref{Assum:regu}. Let $\kappa \geq \left\|f^{\prime}\right\|_{C[-1, 1]}$. Moreover, assume that the time stepsize condition \eqref{MBP:L21_tau} and the maximum stepsize condition
	\begin{equation}\label{Condi:Cover}
		\tau \leq \sqrt[\alpha]{2/(33 \kappa \Gamma(2-\alpha) ( 1 + 2 (1-\theta) r^* ) )}
	\end{equation}
    hold. Then the solution of the stabilized IMEX-Alikhanov scheme \eqref{sch:L21_n}  is convergent in the maximum norm; that is,
	\begin{equation*}\label{Conclu:Cover}
		\| e^{n} \|_{\infty} \leq C_3 ( N^{-\min\{ \gamma \alpha, 2 \}} + h^2 ),
	\end{equation*}
  where $C_3$ is an $\alpha$-robust positive constant that depends on $C_0 \sim C_2$, $\alpha$ and $\kappa$.
\end{theorem}
\begin{proof}
For $ n \geq 2 $, subtracting \eqref{sch:L21_n} from \eqref{Model:tAC} gives the following error equation
\begin{equation}\label{Cover:2}
	\mathbb{D}^{\alpha}_{\tau} e^{n} = \varepsilon^2 \Delta_h \bar{e}^{n,\theta} + f( \phi( t_{n-\theta} ) ) - f( \widehat{\phi}^{n,\theta} ) + \kappa \bigl( \bar{\phi}^{n,\theta} - \widehat{\phi}^{n,\theta} \bigr) + R_{1}^{n} [\phi] + R_{2}^{n} [\Delta \phi] + R_{s}^{n}[\phi],
\end{equation}
where $ R_{s}^{n}[\phi]:= \varepsilon^{2} ( \Delta \bar{\phi}( t_{n-\theta} ) - \Delta_{h} \bar{\phi}( t_{n-\theta} ))$ represents the spatial truncation error, which can be estimated by
\begin{equation}\label{Cover:2_2}
\| R_{s}^{n}[\phi] \|_{\infty} \leq C h^{2} \| \bar{\phi} ( t_{n-\theta} ) \|_{W^{4,\infty} (\Omega)}.
\end{equation}

Note that the stabilization term can be equivalently rewritten as
\begin{equation}\label{Cover:22}
	\bar{\phi}^{n,\theta} - \widehat{\phi}^{n,\theta} = - \bar{e}^{n,\theta} - R_{2}^{n} [\phi] + \phi( t_{n-\theta} ) - \widehat{\phi}^{n,\theta}.
\end{equation}
Then, similar to \eqref{sch:L21_n_1}, the error equation \eqref{Cover:2} also reads as
\begin{equation}\label{Cover:3}
	\begin{aligned}
		  \mQ^{n} e^{n} 
		& = \mM^{n} e^{n-1} + \Xi^{n-2} (e) + f( \phi(t_{n-\theta}) ) - f( \widehat{\phi}^{n,\theta} )
        + \kappa \bigl( \phi(t_{n-\theta}) - \widehat{\phi}^{n,\theta} \bigr)  \\
		& \qquad  + R_{1}^{n}[\phi]- \kappa R_{2}^{n} [\phi] + R_{2}^{n} [\Delta\phi] + R_{s}^{n}[\phi].
	\end{aligned}
\end{equation}
It follows from $ \| \widehat{\phi}^{n,\theta} \|_{\infty} \leq 1 $  and the condition of Lemma \ref{lem:MBP_right} that
$$
\begin{aligned}
	\| f( \phi(t_{n-\theta}) ) - f( \widehat{\phi}^{n,\theta} ) \|_{\infty} \leq \kappa \| \phi(t_{n-\theta}) - \widehat{\phi}^{n,\theta} \|_{\infty}.
\end{aligned}
$$
Following the process of \eqref{Formu:L21_n_MBP_3}--\eqref{Formu:L21_n_MBP_5} and using the above estimate, we obtain from \eqref{Cover:3} that
\begin{equation*}
	\begin{aligned}
		 & ( A^{(n)}_{0} + \kappa ( 1 - \theta ) ) \| e^{n} \|_{\infty} 
	     \leq \|  \mQ^{n} e^{n} \|_{\infty} \\
		& \quad \leq ( A_0^{(n)} - A_1^{(n)} - \kappa \theta ) \| e^{n-1} \|_{\infty} + \sum^{n-2}_{k=1} \big( A^{(n)}_{n-k-1} - A^{(n)}_{n-k} \big) \| e^{k} \|_{ \infty } + A^{(n)}_{n-1} \| e^{0} \|_{ \infty } \\
		& \qquad + 2 \kappa \| \phi(t_{n-\theta}) - \widehat{\phi}^{n,\theta} \|_{\infty} +  \| R_{1}^{n} [\phi] \|_{\infty} + \| \kappa R_{2}^{n} [\phi] \|_{\infty} + \| R_{2}^{n} [\Delta\phi] \|_{\infty} + \| R_{s}^{n}[\phi] \|_{\infty},
	\end{aligned}
\end{equation*}
which together with the fact $ 0 < \theta < 1 $ implies that
\begin{equation}\label{Cover:4}
	\begin{aligned}
		&\sum^{n}_{k=1} A^{(n)}_{n-k} \nabla_{\tau} \| e^{k} \|_{\infty} \\
        & \quad \leq 2 \kappa \| \phi(t_{n-\theta}) - \widehat{\phi}^{n,\theta} \|_{\infty}
        +  \| R_{1}^{n} [\phi] \|_{\infty} + \kappa\|  R_{2}^{n} [\phi] \|_{\infty}
		 + \| R_{2}^{n} [\Delta\phi] \|_{\infty} + \| R_{s}^{n}[\phi] \|_{\infty}.
	\end{aligned}
\end{equation}

Moreover, since the cut-off operation is contractive (due to the MBP-preserving property of the exact solution), the first term on the right-hand side of \eqref{Cover:4} can be bounded as follows:
\begin{equation}\label{Cover:5}
	\begin{aligned} 
		\| \phi(t_{n-\theta}) - \widehat{\phi}^{n,\theta} \|_{\infty} & \le \| \phi(t_{n-\theta}) - \widetilde{\phi}( t_{n-\theta} ) \|_{\infty}+\| \widetilde{\phi}( t_{n-\theta} )  - \widehat{\phi}^{n,\theta} \|_{\infty}\\
		& \leq \| R_{3}^{n}[\phi] \|_{\infty} + ( 1 + (1-\theta) r_{n} ) \| e^{n-1}  \|_{\infty} + (1-\theta) r_{n} \| e^{n-2}  \|_{\infty}.
	\end{aligned}
\end{equation}
Thus, inserting \eqref{Cover:5} into \eqref{Cover:4} yields
\begin{equation}\label{Cover:6}
	\begin{aligned}
		\sum^{n}_{k=1} A^{(n)}_{n-k} \nabla_{\tau} \| e^{k} \|_{\infty} 
		& \leq 2 \kappa \left( 1 + (1-\theta) r_{n} \right) \| e^{n-1}  \|_{\infty} + 2 \kappa (1-\theta) r_{n} \| e^{n-2}  \|_{\infty}+  \| R_{1}^{n} [\phi] \|_{\infty} \\
        & \qquad  + \kappa\|  R_{2}^{n} [\phi] \|_{\infty} + \| R_{2}^{n} [\Delta\phi] \|_{\infty} + 2 \kappa \| R_{3}^{n} [\phi] \|_{\infty} + \| R_{s}^{n} [\phi] \|_{\infty}
        \\
        & \leq 2 \kappa \left( 1 + (1-\theta) r^{*} \right) \| e^{n-1}  \|_{\infty} + 2 \kappa (1-\theta) r^{*} \| e^{n-2}  \|_{\infty}+  \| R_{1}^{n} [\phi] \|_{\infty} \\
		   & \qquad  + \kappa\|  R_{2}^{n} [\phi] \|_{\infty} + \| R_{2}^{n} [\Delta\phi] \|_{\infty} + 2 \kappa \| R_{3}^{n} [\phi] \|_{\infty} + \| R_{s}^{n} [\phi] \|_{\infty},
	\end{aligned}
\end{equation}
for $ n \geq 2 $, which takes the form of \eqref{Condi:Gron} with the substitutions $ w^{k} := \| e^{k} \|_{\infty} $ and 
$$ 
\xi^{n} =  \| R_{1}^{n} [\phi] \|_{\infty}, \quad \eta^{n} =  \kappa \| R_{2}^{n} [\phi] \|_{\infty} + \| R_{2}^{n} [\Delta\phi] \|_{\infty} + 2 \kappa\| R_{3}^{n} [\phi] \|_{\infty} + \| R_{s}^{n}[\phi] \|_{\infty}, 
$$
$$
\lambda_{1} = 2 \kappa \left( 1 + (1-\theta) r^{*} \right), \quad \lambda_{2} = 2 \kappa \left(1-\theta\right) r^{*}, \quad \lambda_{k} = 0 \  \text{for} \  k = 0, \  k \geq 3;\quad \Lambda = 2 \kappa \left( 1 + 2(1-\theta) r^{*} \right).
$$

For the first time level $ n = 1 $, the corresponding error equation is given by
\begin{equation*}
	\mQ^{1} e^{1} = f( t_{1-\theta} ) - f( \phi^{0} ) + R_{1}^{1} [\phi] - \kappa R_{2}^{1}[\phi]  + R_{2}^{1} [\Delta\phi] + \kappa R_{3}^{1}[\phi] + R_{s}^{1}[\phi],
\end{equation*}
where $e^{0} = 0$ and $ \bar{\phi}^{1,\theta} - \phi^{0} = - \bar{e}^{1,\theta} - R_{2}^{1}[\phi] + R_{3}^{1}[\phi] $ have been used. Using the mean value theorem, we have
$$
\| f( \phi( t_{1-\theta} ) ) - f( \phi^{0} ) \|_{\infty} \leq \kappa \| R_{3}^{1}[\phi] \|_{\infty}.
$$
Further, with the help of the triangle inequality and Lemma \ref{lem:MBP_left}, one obtains
\begin{equation*}
	\begin{aligned}
		& A^{(1)}_{0} \| e^{1} \|_{\infty} \leq \|\mQ^{1} e^{1}  \|_{\infty} 
		 \leq \|R_{1}^{1}[\phi]\|_{\infty} + \kappa \| R_{2}^{1}[\phi] \|_{\infty} + \| R_{2}^{1} [\Delta\phi] \|_{\infty} + 2 \kappa\|R_{3}^{1}[\phi]\|_{\infty} + \| R_{s}^{1}[\phi] \|_{\infty},
	\end{aligned}
\end{equation*}
which also takes the form of \eqref{Condi:Gron}, i.e.,
\begin{equation}\label{Cover:7}
	\begin{aligned}
		A^{(1)}_{0} \nabla_{\tau} \| e^{1} \|_{\infty} \leq 
         \|R_{1}^{1}[\phi]\|_{\infty} + \kappa\| R_{2}^{1}[\phi] \|_{\infty} + \| R_{2}^{1} [\Delta\phi] \|_{\infty} 
        + 2\kappa \|R_{3}^{1}[\phi]\|_{\infty} + \| R_{s}^{1}[\phi] \|_{\infty}.
	\end{aligned}
\end{equation}

Finally, under the time stepsize condition \eqref{Condi:Cover}, the application of discrete fractional Gr\"onwall inequality in Lemma \ref{thm:Gron} to \eqref{Cover:6}--\eqref{Cover:7} yields   
\begin{equation*}\label{Cover:8}
	\begin{aligned}
		\| e^{n} \|_{\infty} 
        & \leq C_{n,\alpha} \Big( \max_{1\leq k\leq n} \sum^{k}_{j=1} P^{(k)}_{k-j} \| R_{1}^{j} [\phi] \|_{\infty} \\
        &\qquad + \frac{11}{4} \omega_{1+\alpha}( t_{n} ) \max_{1\leq k \leq n } \left\{\kappa\| R_{2}^{k}[\phi] \|_{\infty}+ \| R_{2}^{k} [\Delta\phi] \|_{\infty} + 2\kappa\| R_{3}^{k} [\phi] \|_{\infty} + \| R_{s}^{k}[\phi] \|_{\infty} \right\} \Big) \\
		&\leq C_3 ( N^{-\min\{ \gamma \alpha, 2 \}} + h^2 ),
	\end{aligned}
\end{equation*}
where Lemmas \ref{lem:L21_robust}--\ref{lem:extra_err} and \eqref{Cover:2_2} have been used in the last inequality. The proof is completed. 
\end{proof}

Based on Lemma \ref{lem:extra_err} and Theorem \ref{thm:Convergence}, the following corollary can be derived from \eqref{Cover:5} immediately.
\begin{corollary}\label{cor:err}
	Assume that the conditions in Theorem \ref{thm:Convergence} hold, we have
	\begin{equation*}\label{Conclu:Cover_n_theta}
		\| \phi( t_{n-\theta} ) - \widehat{\phi}^{n,\theta} \|_{\infty} \leq C_3 ( N^{-\min\{ \gamma \alpha, 2 \}} + h^2 ).
	\end{equation*}
\end{corollary}

\subsection{Discrete energy stability}
Define the discrete version of the original energy
\begin{equation}\label{def:energy_1}
		E_h[ \phi^n ] := \frac{\varepsilon^2}{2}\left\|\nabla_h \phi^n\right\|^2 + \langle F( \phi^{n} ), 1 \rangle, \quad n \geq 0.
\end{equation}
Next, we demonstrate that the proposed stabilized IMEX-Alikhanov scheme \eqref{sch:L21_n} is energy-stable with respect to the discrete energy \eqref{def:energy_1}. More precisely, we prove that for any time $t=t_n$, the discrete energy is bounded by the initial energy plus a high-order spatiotemporal correction term.

\begin{theorem}\label{thm:energy}
	Assume that the conditions in Theorem \ref{thm:Convergence} hold. If the maximum time stepsize also satisfies
	\begin{equation}\label{Condi:energy}
		\tau \leq \sqrt[\alpha]{ \alpha(1-\theta)^{1-\alpha} / ( 4 \Gamma( 3 - \alpha ) )},
	\end{equation} 
    the stabilized IMEX-Alikhanov scheme \eqref{sch:L21_n}  is
	energy stable in the sense that
	\begin{equation}\label{Conclu:energy}
		E_h[ \phi^n ] \leq E_h[ \phi^0 ] + C_4 N^{1-\alpha} ( N^{-\min\{ \gamma \alpha, 2 \}} + h^2 )^2, \qquad n \geq 1,
	\end{equation}
    where $C_4$ is an $\alpha$-robust positive constant that depends on $C_0 \sim C_3$, $\alpha$ and $\kappa$.
\end{theorem}
\begin{proof}
Letting $ n = k $ in \eqref{sch:L21_n} and taking the inner product with $ \nabla_{\tau} \phi^{k}$, we get
\begin{equation}\label{energy:1}
	\begin{aligned}
		& \langle\mbD^{\alpha}_{\tau} \phi^{k}, \nabla_{\tau} \phi^{k} \rangle + \varepsilon^2 \langle \nabla_h \bar{\phi}^{k,\theta}, \nabla_{\tau} \nabla_h \phi^{k} \rangle - \langle f(\bar{\phi}^{k,\theta}), \nabla_{\tau} \phi^{k} \rangle \\
		& \qquad = \langle  f(\widehat{\phi}^{k,\theta}) - f(\bar{\phi}^{k,\theta}), \nabla_{\tau} \phi^{k} \rangle  - \kappa \langle \bar{\phi}^{k,\theta} - \widehat{\phi}^{k,\theta}, \nabla_{\tau} \phi^{k} \rangle.
	\end{aligned}
\end{equation}

Next, we estimate each term in \eqref{energy:1} for the case $k \geq 2$ separately. First, applying Lemma \ref{lem:L21_DC_positive}, the first term on the left-hand side can be bounded below by
\begin{equation}\label{energy:l1}
	\begin{aligned}
		\langle\mbD^{\alpha}_{\tau} \phi^{k}, \nabla_{\tau} \phi^{k} \rangle \geq \frac{1}{2} \langle \mathcal{G}[\nabla_\tau \phi^k], 1 \rangle - \frac{1}{2} \langle \mathcal{G} [\nabla_\tau \phi^{k-1}], 1 \rangle + \frac{ \alpha a_0^{(k)}}{2-\alpha} \| \nabla_\tau \phi^k \|^2,
	\end{aligned}
\end{equation}
where a simple calculation shows that $a_{0}^{(k)} = \frac{1}{\tau_{k}} \int^{t_{k- \theta}}_{t_{k-1}} \omega_{1-\alpha}(t_{k- \theta}-s) ds = \frac{ (1-\theta)^{1-\alpha} }{ \Gamma(2-\alpha) \tau_{k}^{\alpha} } = O(\tau_{k}^{-\alpha})$.

Second, it is straightforward to verify that
$$
\left( (1-\theta) a + \theta b \right) ( a - b ) 
= \frac{1}{2} a^2 - \frac{1}{2} b^2 + \frac{1 - 2\theta}{2} ( a - b )^2 \ge \frac{1}{2} a^2 - \frac{1}{2} b^2,
$$
for $ \theta = \alpha/2 $ with $ \alpha \in (0,1) $. Then, the second term on the left-hand side can be bounded below by
\begin{equation}\label{energy:l2}
		\varepsilon^2 \langle \nabla_h \bar{\phi}^{k,\theta}, \nabla_{\tau} \nabla_{h} \phi^{k} \rangle \geq \frac{\varepsilon^2}{2} \| \nabla_{h} \phi^{k} \|^2 - \frac{\varepsilon^2}{2} \| \nabla_{h} \phi^{k-1} \|^2.
\end{equation}

Third, for any $ a, b \in [-1,1]$ and any $ \xi$ lying between $a$ and  $b$, it holds that
\begin{equation*}\label{energy:l3_1}
	\begin{aligned}
		\frac{1}{4} \big[ ( a^2 - 1 )^2 - ( b^2 - 1 )^2 \big] \leq ( \xi^3 - \xi ) ( a - b ) + 2 ( a - b )^2.
	\end{aligned}
\end{equation*}
As a consequence, the third term on the left-hand side can be bounded below by
\begin{equation}\label{energy:l3_2}
	\begin{aligned}
		- \langle f(\bar{\phi}^{k,\theta}), \nabla_{\tau} \phi^{k} \rangle \geq \langle F( \phi^{k} ) - F( \phi^{k-1} ), 1 \rangle - 2 \| \nabla_{\tau} \phi^{k} \|^2.
	\end{aligned}
\end{equation}

Fourth, utilizing Cauchy-Schwarz inequality and Young's inequality, the right-hand side can be estimated by
\begin{equation*}
	\begin{aligned}
		& \langle  f(\widehat{\phi}^{k,\theta}) - f(\bar{\phi}^{k,\theta}), \nabla_{\tau} \phi^{k} \rangle  - \kappa \langle \bar{\phi}^{k,\theta} - \widehat{\phi}^{k,\theta}, \nabla_{\tau} \phi^{k} \rangle \\
		& \quad \leq 2 \kappa \| \bar{\phi}^{k,\theta} - \widehat{\phi}^{k,\theta} \| \| \nabla_{\tau} \phi^{k} \| \leq C \tau_{k}^{\alpha} \| \bar{\phi}^{k,\theta} - \widehat{\phi}^{k,\theta} \|^{2} + \frac{ \alpha a_{0}^{(k)} }{ 4-2\alpha } \| \nabla_{\tau} \phi^{k} \|^{2},
	\end{aligned}
\end{equation*}
where the facts $ \| \bar{\phi}^{k,\theta} \|_{\infty} \leq 1 $ and $ \| \widehat{\phi}^{k,\theta} \|_{\infty} \leq 1 $ have been used. Moreover, by applying Lemma \ref{lem:weight_err}, Theorem \ref{thm:Convergence} and Corollary \ref{cor:err}, we conclude from \eqref{Cover:22} that $ \| \bar{\phi}^{k,\theta} - \widehat{\phi}^{k,\theta} \| \leq C \| \bar{\phi}^{k,\theta} - \widehat{\phi}^{k,\theta} \|_{\infty} \leq C ( N^{-\min\{ \gamma \alpha, 2 \}} + h^2 ) $. Thus, we further have
\begin{equation}\label{energy:r_1}
	\begin{aligned}
		&\langle  f(\widehat{\phi}^{k,\theta}) - f(\bar{\phi}^{k,\theta}), \nabla_{\tau} \phi^{k} \rangle  - \kappa \langle \bar{\phi}^{k,\theta} - \widehat{\phi}^{k,\theta}, \nabla_{\tau} \phi^{k} \rangle \\
		&\quad \leq C \tau_{k}^{\alpha} ( N^{-\min\{ \gamma \alpha, 2 \}} + h^2 )^2 + \frac{ \alpha a_{0}^{(k)} }{ 4-2\alpha } \| \nabla_{\tau} \phi^{k} \|^{2}.
	\end{aligned}
\end{equation}

Now, we insert \eqref{energy:l1}--\eqref{energy:r_1} into \eqref{energy:1} to obtain
\begin{equation*}
	\begin{aligned}
		& \frac{1}{2} \langle \mathcal{G}[\nabla_\tau \phi^k], 1 \rangle + E_{h} [ \phi^{k} ]  - \frac{1}{2} \langle \mathcal{G} [\nabla_\tau \phi^{k-1}], 1 \rangle - E_{h} [ \phi^{k-1} ]  + \frac{ \alpha a_0^{(k)}}{2-\alpha} \| \nabla_\tau \phi^k \|^2 \\
		& \leq  C \tau_{k}^{\alpha} ( N^{-\min\{ \gamma \alpha, 2 \}} + h^2 )^2 + \left( 2 + \frac{ \alpha a_{0}^{(k)} }{ 4-2\alpha } \right) \| \nabla_{\tau} \phi^{k} \|^{2}.
	\end{aligned}
\end{equation*}
Then, we conclude that if $ \tau_{k} \leq \sqrt[\alpha]{ \alpha(1-\theta)^{1-\alpha} / ( 4 \Gamma( 3 - \alpha ) )} $, i.e., $2 + \frac{ \alpha a_{0}^{(k)} }{ 4-2\alpha } \le \frac{ \alpha a_0^{(k)}}{2-\alpha}$, it follows that
\begin{equation*}
	\begin{aligned}
		\frac{1}{2} \langle \mathcal{G}[\nabla_\tau \phi^k], 1 \rangle + E_{h} [ \phi^{k} ]  - \frac{1}{2} \langle \mathcal{G} [\nabla_\tau \phi^{k-1}], 1 \rangle - E_{h} [ \phi^{k-1} ]  \leq  C \tau_{k}^{\alpha} ( N^{-\min\{ \gamma \alpha, 2 \}} + h^2 )^2, \quad k \geq 2,
	\end{aligned}
\end{equation*}
which further means
\begin{equation}\label{energy:2}
	\begin{aligned}
		\frac{1}{2} \langle \mathcal{G}[\nabla_\tau \phi^n], 1 \rangle + E_{h} [ \phi^{n} ]  - \frac{1}{2} \langle \mathcal{G} [\nabla_\tau \phi^{1}], 1 \rangle - E_{h} [ \phi^{1} ]  \leq  C \sum^{n}_{k=2} \tau_{k}^{\alpha} ( N^{-\min\{ \gamma \alpha, 2 \}} + h^2 )^2.
	\end{aligned}
\end{equation}

In particular, for $ k = 1 $, 
equation \eqref{energy:1} reduces to
\begin{equation*}\label{energy:3}
	\begin{aligned}
		\langle\mbD^{\alpha}_{\tau} \phi^{1}, \nabla_{\tau} \phi^{1} \rangle + \varepsilon^2 \langle \nabla_{h} \bar{\phi}^{1,\theta}, \nabla_{\tau} \nabla_{h} \phi^{1} \rangle - \langle f(\phi^{0}), \nabla_{\tau} \phi^{1} \rangle = \kappa \langle \bar{\phi}^{1,\theta} - \phi^{0}, \nabla_{\tau} \phi^{1} \rangle  ,
	\end{aligned}
\end{equation*}
which can be seen as an easy case compared to the situation with $ k \geq 2 $. Therefore, under condition \eqref{Condi:energy}, we have
\begin{equation}\label{energy:4}
	\begin{aligned}
		\frac{1}{2} \langle \mathcal{G}[\nabla_\tau \phi^1], 1 \rangle + E_{h} [ \phi^{1} ]  - \frac{1}{2} \langle \mathcal{G} [\nabla_\tau \phi^{0}], 1 \rangle - E_{h} [ \phi^{0} ]  \leq  C \tau_{1}^{\alpha} ( N^{-\min\{ \gamma \alpha, 2 \}} + h^2 )^2,
	\end{aligned}
\end{equation}
where the estimate $ \| \bar{\phi}^{1,\theta} - \phi^{0} \| \leq C ( \| \bar{e}^{1,\theta} \|_{\infty} + \| R_{3}^{1}[\phi] \|_{\infty} ) \leq C ( N^{-\min\{ \gamma \alpha, 2 \}} + h^2 ) $ has been used.  Thus, \eqref{energy:4} together with \eqref{energy:2} and the definition of $ \mathcal{G}[w^{n}] $ in Lemma \ref{lem:L21_DC_positive} gives us
\begin{equation}\label{energy:5}
		E_h[ \phi^n ]  \leq E_h[ \phi^0 ] + C \sum^{n}_{k=1} \tau_{k}^{\alpha} ( N^{-\min\{ \gamma \alpha, 2 \}} + h^2 )^2.
\end{equation}
Note that the graded mesh satisfies 
$ \tau_{k}= t_{k}-t_{k-1}\leq T\gamma N^{-\gamma} k^{\gamma-1}$; consequently
$$
\sum^{n}_{k=1} \tau_{k}^{\alpha} \leq C N^{-\gamma \alpha} \sum^{n}_{k=1} k^{ \alpha ( \gamma - 1 ) } \leq C N^{-\gamma \alpha} n^{ 1 + \alpha ( \gamma - 1 ) } \leq C N^{ 1 - \alpha }.
$$
Finally, inserting this estimate into \eqref{energy:5}, the conclusion \eqref{Conclu:energy} can be deduced immediately.
\end{proof}
\begin{remark}
If the mesh grading parameter satisfies $ \gamma \geq 2/\alpha $ and $ N = O( M ) $ (which is a natural choice for second-order schemes), then the higher-order perturbation term appearing in the energy stability estimate is of order $ O( N^{-3-\alpha} ) $.
\end{remark}

\section{Numerical examples}\label{Sec:Numer}
In this section, we present several numerical experiments to evaluate the accuracy and structure-preserving properties of the proposed stabilized IMEX-Alikhanov scheme for the tFAC model \eqref{Model:tAC}. To adequately address the initial weak singularity and maintain accuracy, two types of mixed nonuniform temporal meshes are employed, both featuring a graded mesh near the initial time. Specifically, the time interval $ [0,T] $ is divided into two subintervals, $ [ 0, \widehat{T} ] $ and $ [\widehat{T}, T] $, with a total of $N$ subintervals. The first subinterval $ [ 0, \widehat{T} ] $ is discretized using a graded mesh given by
\begin{equation}\label{mesh:graded}
	t_{k} = \widehat{T} ( k/\widehat{N} )^{\gamma} \  \  \text{and} \  \  \tau_{k} = t_{k} - t_{k-1}, \quad 1 \leq k \leq \widehat{N} < N,
\end{equation}
where $ \gamma \geq 1 $ denotes the mesh grading parameter. For the subsequent interval $ [\widehat{T}, T] $, two distinct temporal discretization strategies are considered:
\begin{itemize}
	\item[(i)] uniform mesh:
	\begin{equation}\label{mesh:uni}
		\tau_{k} = \frac{ T - \widehat{T} }{ N - \widehat{N} }, \quad \widehat{N}+1 \leq k \leq N;
	\end{equation}
	
	\item[(ii)] adaptive mesh:
	\begin{equation}\label{Alg1:adaptive}
		\begin{aligned}
			\tau_{k} = \max \bigg\{ \max \bigg\{ \tau_{\min} , \f{\tau_{\max}}{ \sqrt{ 1 + \eta \vert \p_{\tau} E^{k-1} \vert^{2} } } \bigg\}, r_{\min} \tau_{k-1} \bigg\} 
		\end{aligned}
	\end{equation}
where $ \tau_{\max}$ and $\tau_{\min} $ correspond to the predetermined maximum and minimum time stepsizes, respectively, $ \eta $ is a tunable parameter, and $r_{\min}$ is a constant satisfying $ r_{\min} \geq 4/7 $.
\end{itemize}

In the following tests, the temporal error is measured in the maximum norm, i.e., $ e(N) := \max_{1\leq n \leq N} \| \phi(t_{n}) - \phi^{n} \|_{\infty} $, and the temporal convergence order is computed as
$
	\text{Order} := \frac{\log(e(N)/e(2N))}{\log 2}.
$
Moreover, the stabilization parameter is fixed as $ \kappa = \| f' \|_{C[-1,1]} = 2 $.

\subsection{Temporal accuracy}\label{Ex:1}
In this example, the temporal accuracy of the stabilized IMEX-Alikhanov scheme is verified using the externally forced tFAC model
\begin{equation*}
	\begin{aligned}
		\p_t^{\alpha} \phi =   \varepsilon^2 \Delta \phi + f( \phi )  + g( \mathbf{x}, t ), \quad t >0, \  \mathbf{x} \in \Omega= (0,2\pi)^2, 
	\end{aligned}
\end{equation*}
with $\varepsilon = 0.1$. Here, the external source term  $g( \mathbf{x}, t )$ is chosen such that the exact solution is given by $ \phi(\mathbf{x},t) = \omega_{1+\alpha}(t)\sin(x)\sin(y) $. It is straightforward to verify that this solution exhibits a weak initial singularity, as characterized in \eqref{Assum:regu}. 

\begin{table}[!ht]
	\caption{{Time accuracy of the stabilized scheme \eqref{sch:L21_n} with different fractional orders $\alpha = 0.3, 0.5, 0.8$ and $ \gamma = 2/\alpha $} \label{Ex5:tab1}}%
	{\footnotesize\begin{tabular*}{\columnwidth}{@{\extracolsep\fill}cccccccc@{\extracolsep\fill}}
			\toprule
			\multicolumn{1}{c}{\multirow{2}{*}{$ N $}} & \multicolumn{2}{c}{$\alpha=0.3$} & \multicolumn{2}{c}{$\alpha=0.5$} & \multicolumn{2}{c}{$\alpha=0.8$} \\
			\cmidrule{2-3} 
			\cmidrule{4-5}
			\cmidrule{6-7}
			& $ e(N) $ & Order & $ e(N) $ & Order & $ e(N) $ & Order \\
			\midrule
		    40     & $4.39 \times 10^{-2}$   &  ---   &  $1.13 \times 10^{-2}$ & ---   & $1.37 \times 10^{-3}$   &  ---   \\
			80     & $1.03 \times 10^{-2}$   &  2.10   &  $2.71 \times 10^{-3}$ & 2.06   & $3.22 \times 10^{-4}$   &  2.09   \\
			160    & $2.51 \times 10^{-3}$   &  2.03   &  $6.65 \times 10^{-4}$ & 2.02   & $8.04 \times 10^{-5}$   &  2.00   \\
			320    & $5.96 \times 10^{-4}$   &  2.07   &  $1.65 \times 10^{-4}$ & 2.01   & $2.01 \times 10^{-5}$   &  2.00   \\
			\midrule 
			\multicolumn{1}{c}{$ \min\{ 2, \gamma \alpha \} $ }   
			&       & 2.00       &    & 2.00 &   & 2.00 \\
			\bottomrule
	\end{tabular*}}
\end{table}

\begin{table}[!ht]
	\caption{{Time accuracy of the stabilized scheme \eqref{sch:L21_n} with different fractional orders $\alpha = 0.99, 0.999, 0.9999$ and $ \gamma = 2/\alpha $} \label{Ex5:tab2}}%
	{\footnotesize\begin{tabular*}{\columnwidth}{@{\extracolsep\fill}cccccccc@{\extracolsep\fill}}
			\toprule
			\multicolumn{1}{c}{\multirow{2}{*}{$ N $}} & \multicolumn{2}{c}{$\alpha=0.99$} & \multicolumn{2}{c}{$\alpha=0.999$} & \multicolumn{2}{c}{$\alpha=0.9999$} \\
			\cmidrule{2-3} 
			\cmidrule{4-5}
			\cmidrule{6-7}
			& $ e(N) $ & Order & $ e(N) $ & Order & $ e(N) $ & Order \\
			\midrule
		    40     & $3.08 \times 10^{-5}$   &  ---   &  $1.88 \times 10^{-5}$ & ---   & $1.77\times 10^{-5}$   &  ---   \\
			80     & $6.95 \times 10^{-6}$   &  1.96   &  $4.74 \times 10^{-6}$ & 1.99   & $4.55 \times 10^{-6}$   &  2.09   \\
			160    & $1.54 \times 10^{-6}$   &  1.99   &  $1.17 \times 10^{-6}$ & 2.02   & $1.14 \times 10^{-6}$   &  2.00   \\
			320    & $3.50 \times 10^{-7}$   &  2.00   &  $2.87 \times 10^{-7}$ & 2.02   & $2.86 \times 10^{-7}$   &  2.00   \\
			\midrule 
			\multicolumn{1}{c}{$ \min\{ 2, \gamma \alpha \} $ }   
			&       & 2.00       &    & 2.00 &   & 2.00 \\
			\bottomrule
	\end{tabular*}}
\end{table}

For $ \gamma \geq 1 $,  the graded mesh \eqref{mesh:graded} is employed on $[0, \widehat{T}]$ with $ \widehat{T} = \min\{ 1/\gamma, T \} $ and $ \widehat{N} = \lceil \frac{ N }{ T + 1 - \gamma^{-1} }  \rceil $. On the remaining interval $ [\widehat{T}, T] $, the uniform temporal mesh \eqref{mesh:uni} is applied, with the final time set to $T=0.5$. The errors and convergence orders of the proposed stabilized IMEX-Alikhanov scheme are tested by setting $ N = M $ and choosing different fractional orders $\alpha = 0.3, 0.5,0.8$. The numerical results, listed in Table \ref{Ex5:tab1}, demonstrate that the proposed method achieves second-order temporal accuracy, which is consistent with the theoretical analysis. Moreover, we test the $ \alpha $-robustness by computing the errors for $ \alpha = 0.99, 0.999 $, and $ 0.9999 $ in Table \ref{Ex5:tab2}. The results indicate that the proposed scheme still performs very well as $\alpha \rightarrow 1^{-} $, confirming that it is indeed $\alpha$-robust.

\subsection{MBP preservation and energy stability}
In this example, we verify the MBP and energy stability of the proposed scheme \eqref{sch:L21_n}. The computational domain is set to $\Omega = (0,1)^2$, with fractional order $\alpha = 0.9$ and the interfacial coefficient $\varepsilon = 0.01$. In all subsequent numerical tests, a uniform spatial mesh of $M = 128$ grid points in each spatial direction is employed. The initial phase-field variable is prescribed as $\phi_{\text{init}}(x,y) = 0.9\sin(5\pi x)\sin(5\pi y)$. 

\begin{figure}[!htbp]
	\vspace{-12pt}
	\centering
	\subfigure[maximum-norm of $\phi$]
	{
		\includegraphics[width=0.45\textwidth]{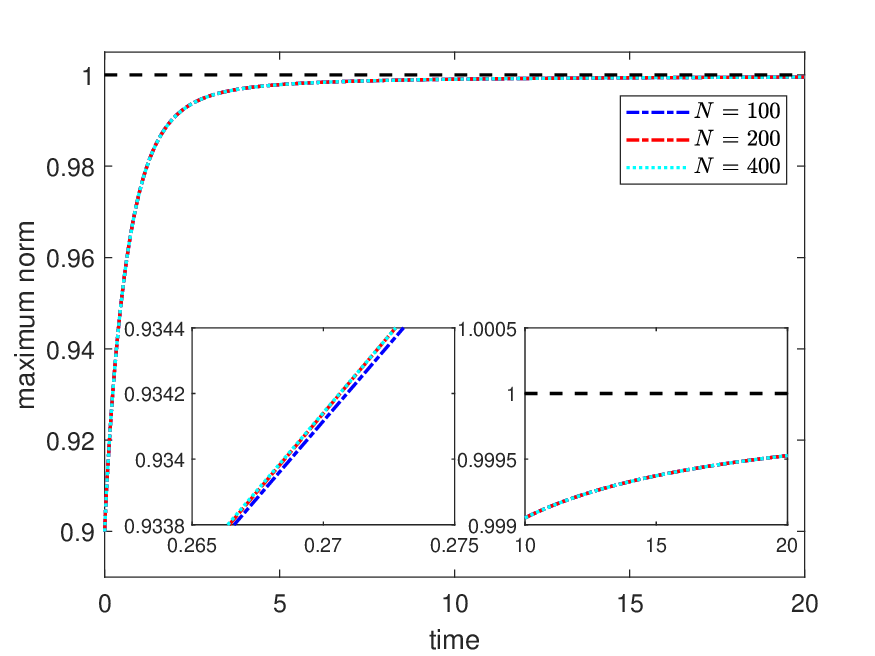}		\label{figEx2_3a}
	}%
	\subfigure[energy]
	{
		\includegraphics[width=0.45\textwidth]{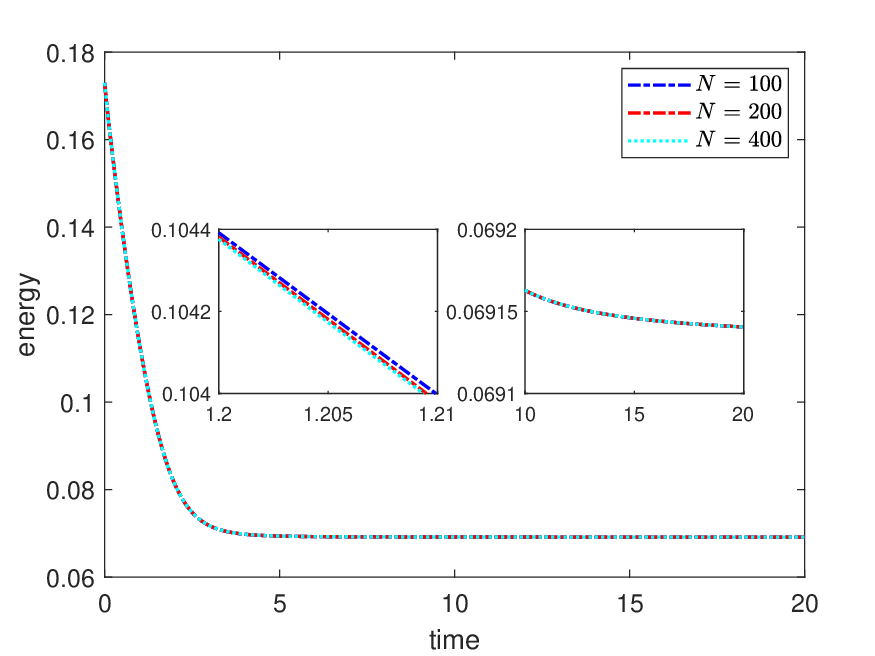}	\label{figEx2_3b}
	}%
	\setlength{\abovecaptionskip}{0.0cm} 
	\setlength{\belowcaptionskip}{0.0cm}
	\caption{Time evolution of the maximum-norm of $\phi$ and original energy with different time steps}	
	\label{figEx2_3}
\end{figure}
Fig.~\ref{figEx2_3} presents the time evolution of the maximum norm and the original energy computed by the stabilized IMEX-Alikhanov scheme up to $T=20$ using different time steps ($N=100$, $200$, and $400$). In this test, we employ mixed nonuniform temporal meshes as described in Example~\ref{Ex:1}. The results indicate that the proposed method preserves both the MBP and energy stability for all cases, demonstrating its reliable structure-preserving performance.

\subsection{Application of adaptive time-stepping strategy}
The dynamic evolution governed by the tFAC model \eqref{Model:tAC} typically requires a long time to reach a steady state and involves multiple time scales, which motivates the use of adaptive time-stepping methods to improve computational efficiency while maintaining numerical accuracy. In this example, a graded mesh with $\widehat{T} = 0.5$, $ \widehat{N} = 30 $, and $ \gamma = 2/\alpha $ is employed on the initial interval $ [0,\widehat{T}] $, and the remainder $ [\widehat{T}, T] $ is discretized using three different temporal strategies: a uniform large time stepsize $\tau=1$, a uniform small time stepsize $\tau=0.01$, and an adaptive time-stepping strategy based on \eqref{Alg1:adaptive}, with parameters $ \tau_{\max} = 1, \tau_{\min} = 0.01$, and $\eta = 10^{7}$. The phase-field variable is initialized with uniformly distributed random values in the interval
$[-0.9,0.9]$. The simulation parameters are set as $\alpha = 0.9$ and $\varepsilon = 0.01$, and the spatial discretization is performed on a uniform grid with $M = 128$ points in each direction over the domain $\Omega = (0,1)^2$.

Fig. \ref{figEx3_1} illustrates snapshots of the numerical solution $\phi$ at four distinct time instants. It can be seen that employing a large time stepsize $\tau=1$ compromises the accuracy of the solution $\phi$, whereas the adaptive time-stepping strategy yields results comparable to those obtained with the small uniform time stepsize $\tau=0.01$. In addition, the evolution of the maximum-norm of $\phi$ and the discrete energy is presented in Fig. \ref{figEx3_2a}--\ref{figEx3_2b}. These results confirm that all three temporal discretization strategies can preserve the MBP and energy stability. Moreover, Fig. \ref{figEx3_2c} depicts the adaptive time stepsizes with respect to time, showing that small time stepsizes are employed only during periods of rapid energy decay. Table \ref{tab1_Ex33} reports the CPU times and the total number of time steps required by the stabilized IMEX-Alikhanov scheme under different temporal discretization strategies. The results demonstrate the high computational efficiency of the adaptive time-stepping strategy. For example, the simulation using a uniform small time stepsize takes approximately one hour to complete, whereas the adaptive strategy finishes in only $37$ seconds. In fact, the adaptive method requires only $695$ time steps, which represents a substantial reduction compared with the uniform fine time-stepping approach (see Table \ref{tab1_Ex33}).
\begin{figure}[!htbp]
	\vspace{-12pt}
	\centering
	\subfigure[$t=5$]
	{
		\begin{minipage}[t]{0.24\linewidth}
			\centering
			\includegraphics[width=1.43in]{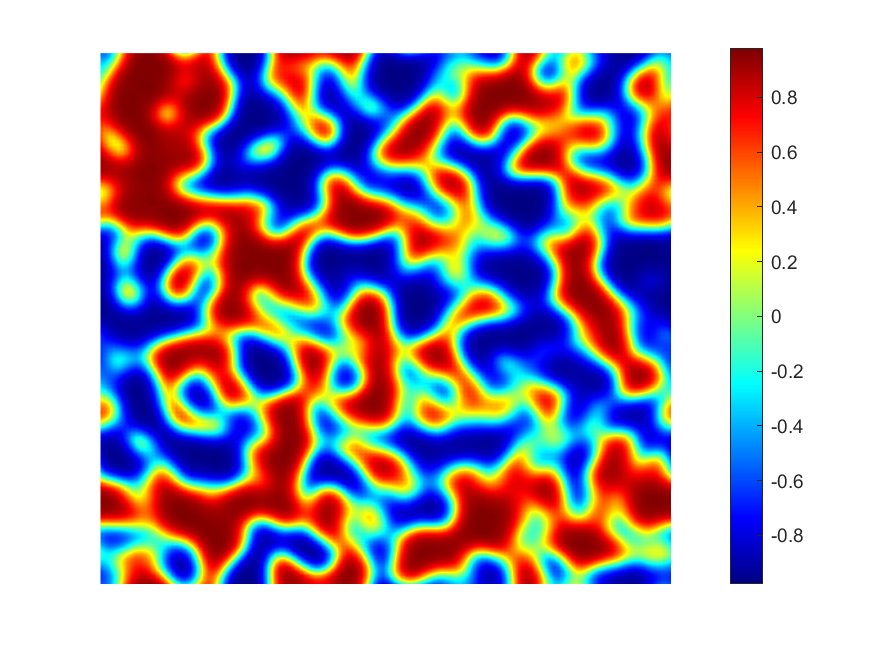}
			\includegraphics[width=1.43in]{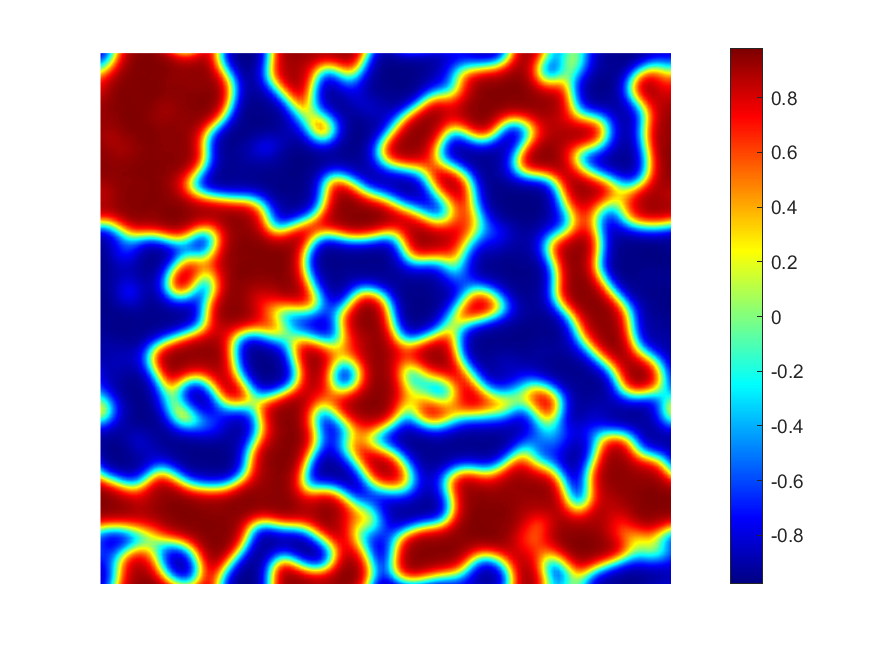}
			\includegraphics[width=1.43in]{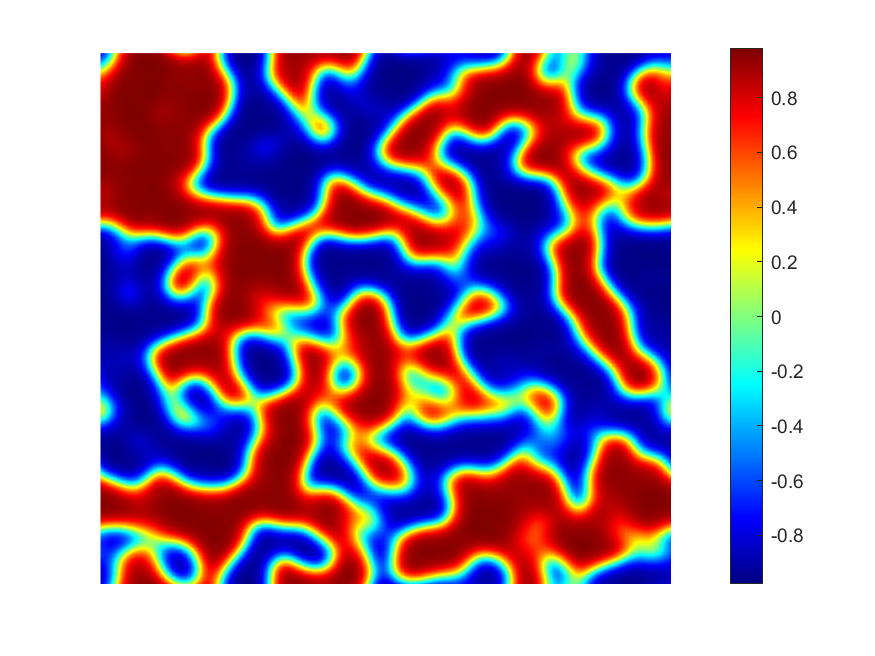}
		\end{minipage}%
	}%
	\subfigure[$t=10$]
	{
		\begin{minipage}[t]{0.24\linewidth}
			\centering
			\includegraphics[width=1.43in]{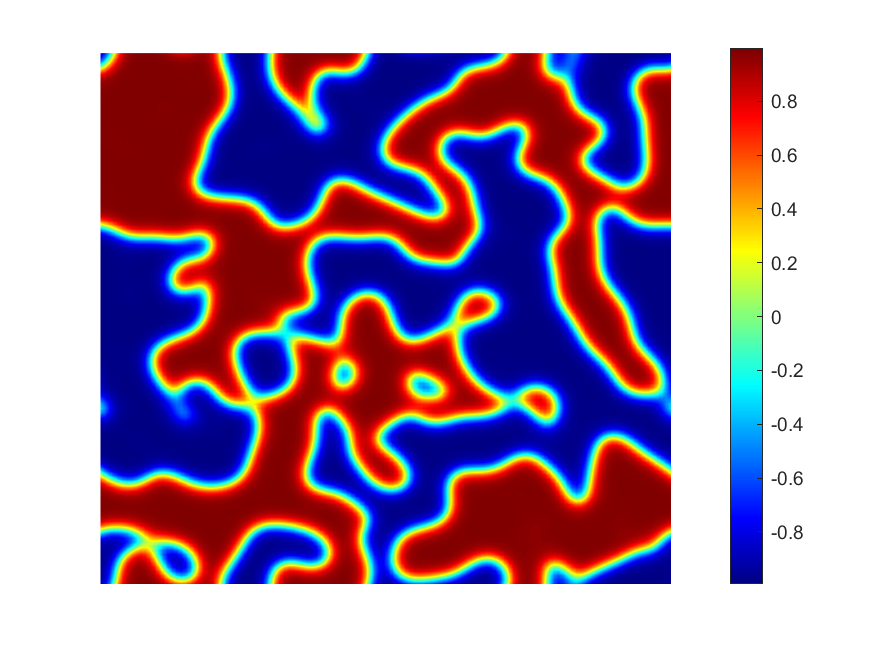}
			\includegraphics[width=1.43in]{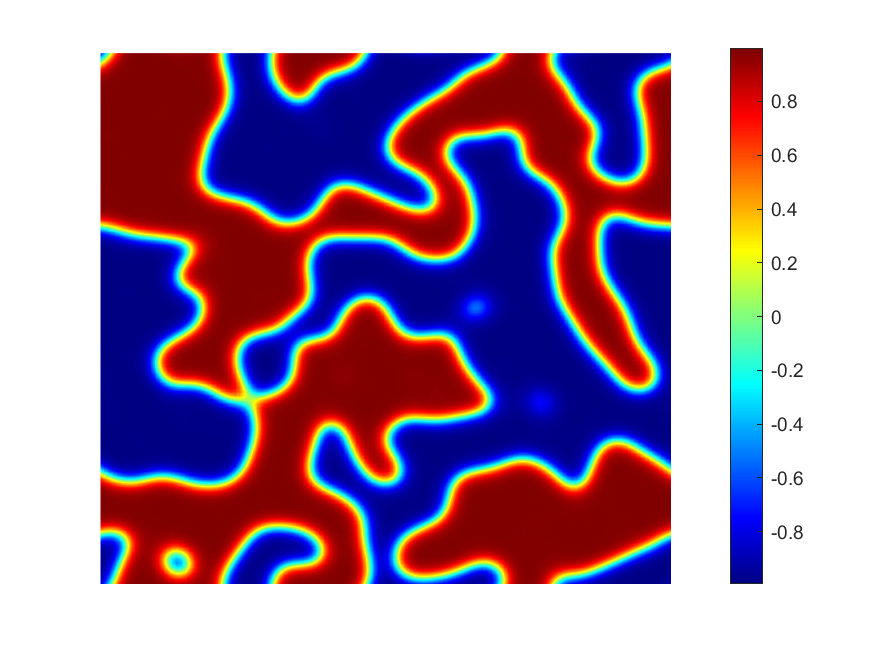}
			\includegraphics[width=1.43in]{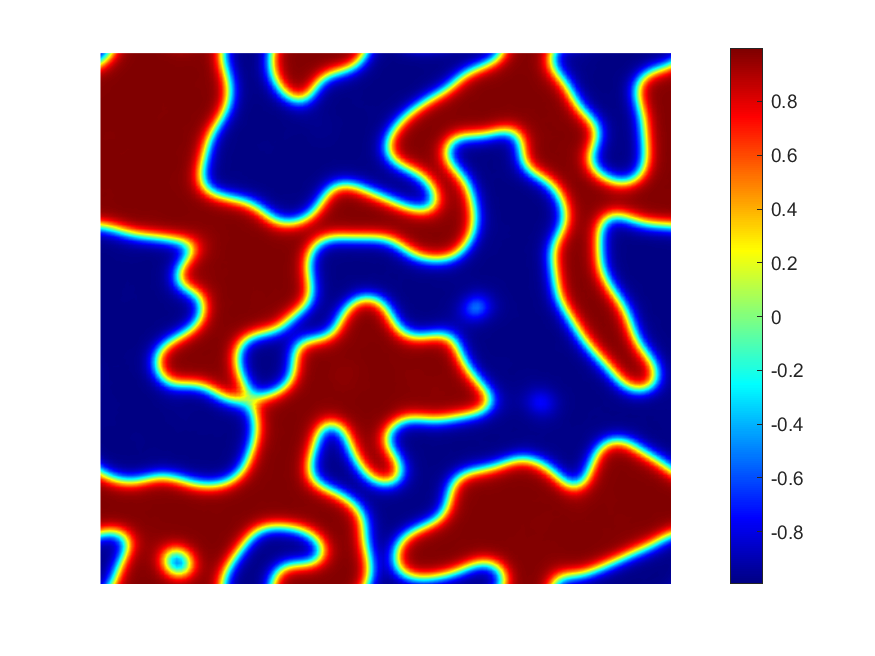}
		\end{minipage}%
	}%
	\subfigure[$t=20$]{
		\begin{minipage}[t]{0.24\linewidth}
			\centering
			\includegraphics[width=1.43in]{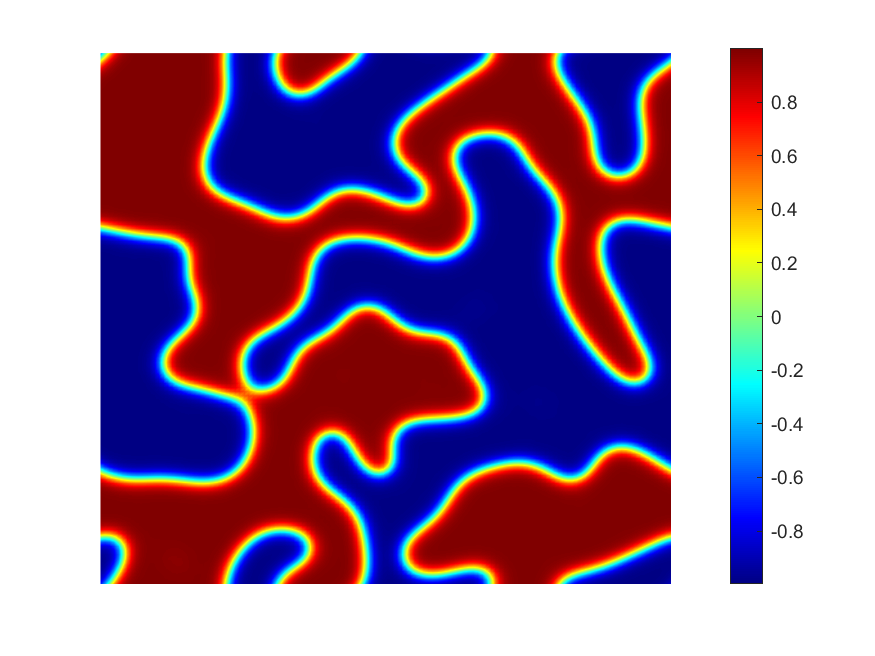}
			\includegraphics[width=1.43in]{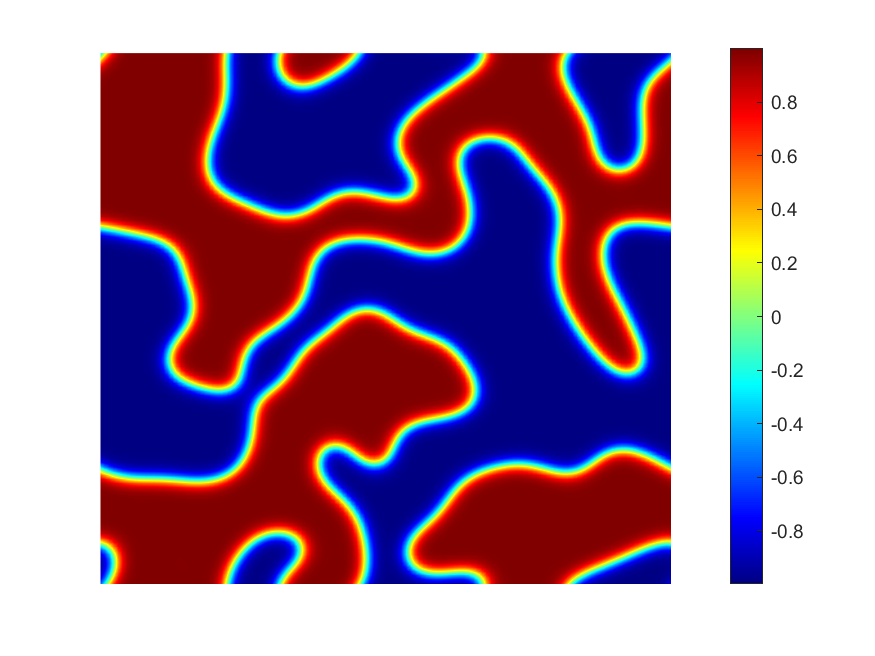}
			\includegraphics[width=1.43in]{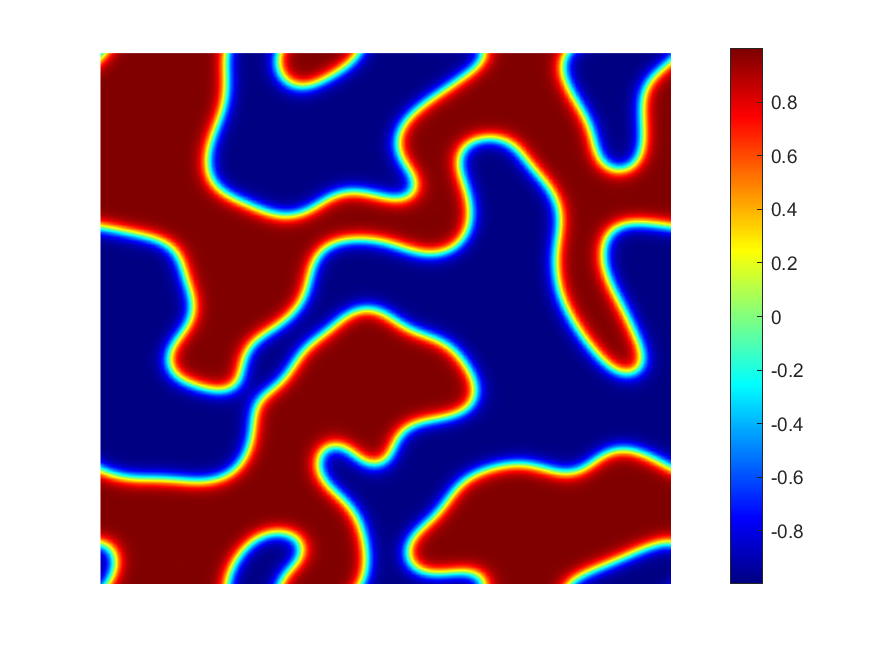}
		\end{minipage}%
	}%
	\subfigure[$t=100$]
	{
		\begin{minipage}[t]{0.24\linewidth}
			\centering
			\includegraphics[width=1.43in]{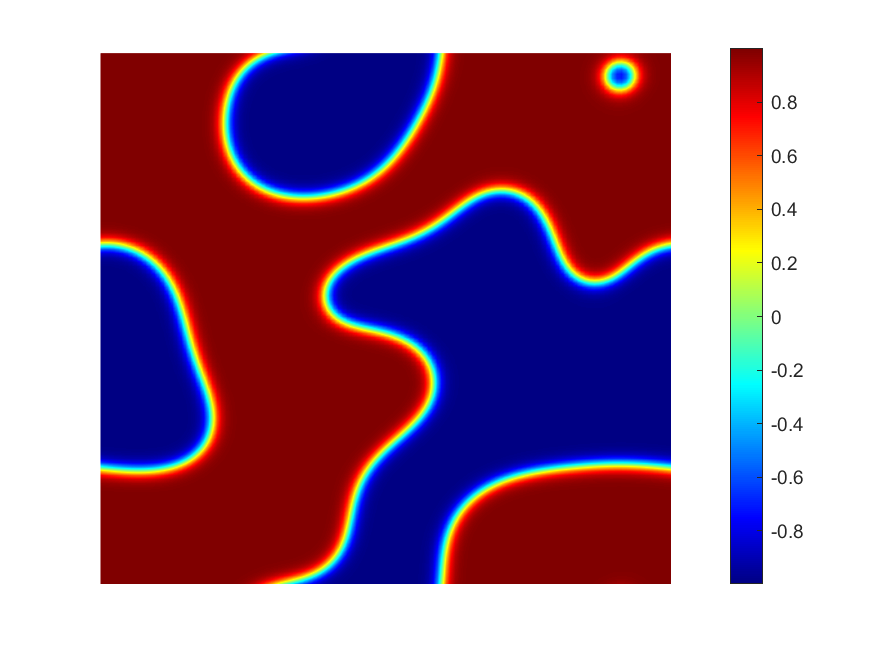}
			\includegraphics[width=1.43in]{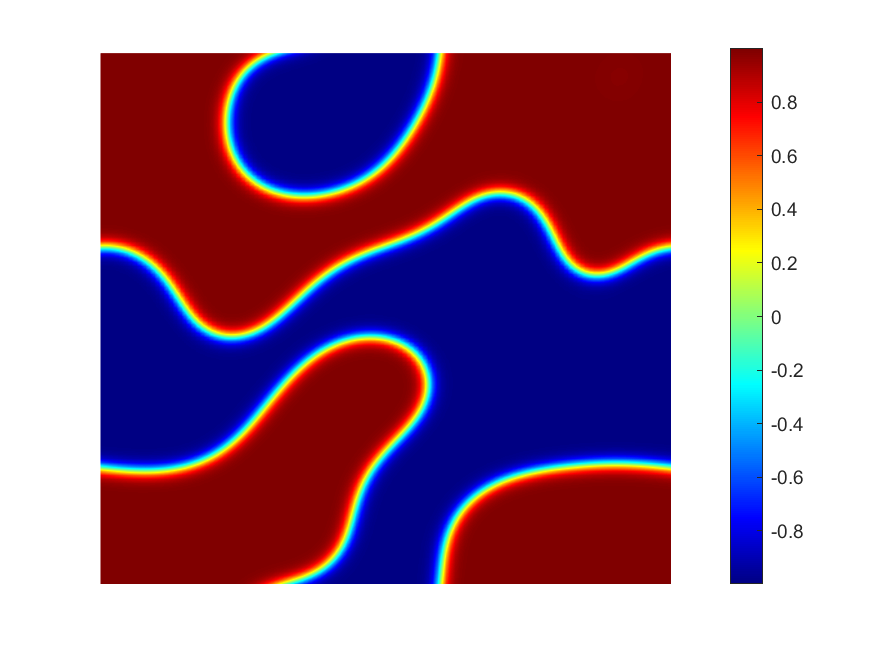}
			\includegraphics[width=1.43in]{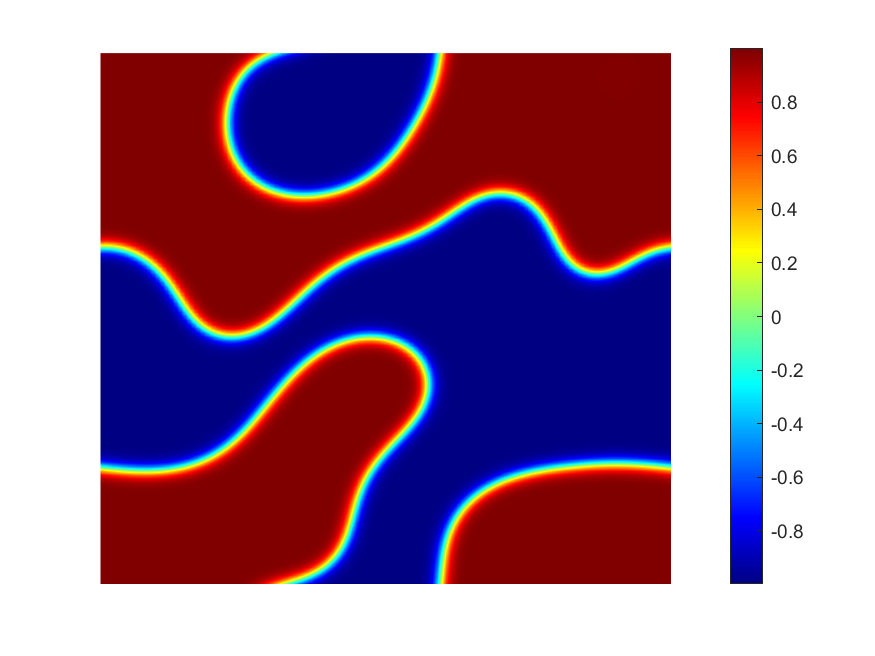}
		\end{minipage}%
	}%
	\setlength{\abovecaptionskip}{0.0cm} 
	\setlength{\belowcaptionskip}{0.0cm}
	\caption{The dynamic snapshots of the numerical solution $\phi$ computed using the uniform  (top, $\tau = 1$; bottom, $\tau = 0.01$) and adaptive (middle) time stepsizes}	\label{figEx3_1}
\end{figure}

\begin{figure}[!htbp]
	\vspace{-10pt}
	\centering
	\subfigure[maximum-norm of $\phi$]
	{
		\includegraphics[width=0.325\textwidth]{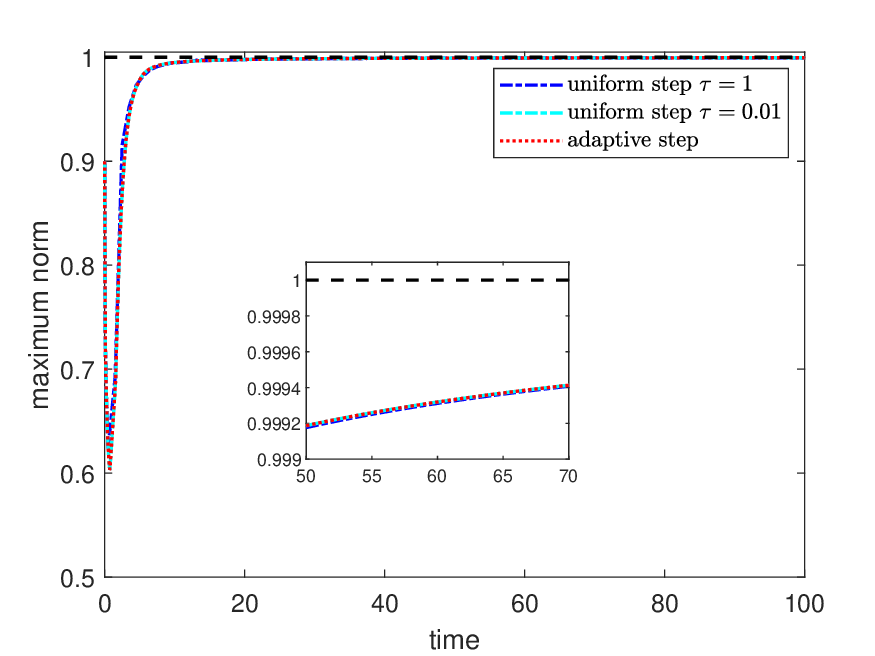}
		\label{figEx3_2a}
	}%
	\subfigure[energy]
	{
		\includegraphics[width=0.325\textwidth]{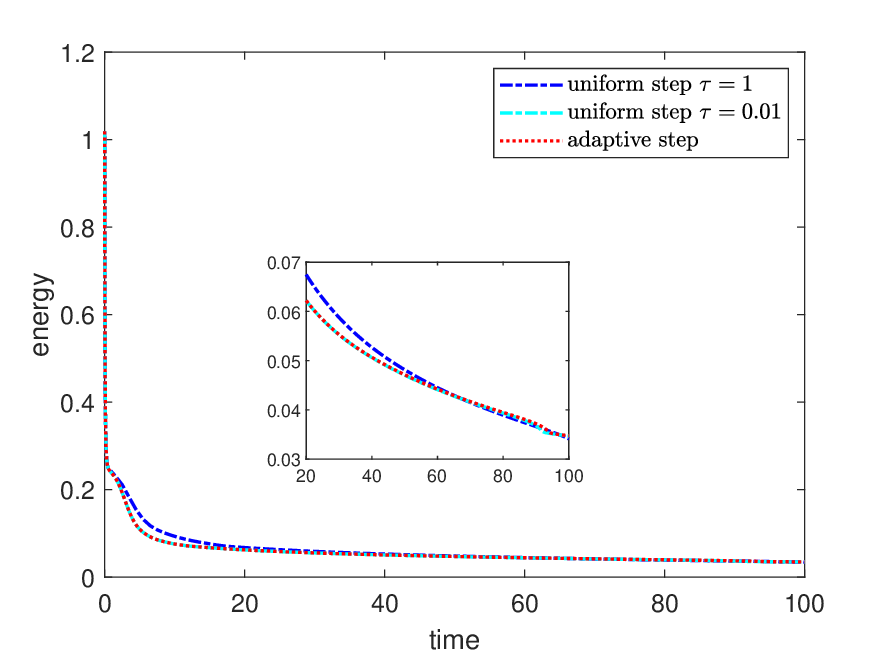}
		\label{figEx3_2b}
	}%
	\subfigure[time stepsizes]
	{
		\includegraphics[width=0.325\textwidth]{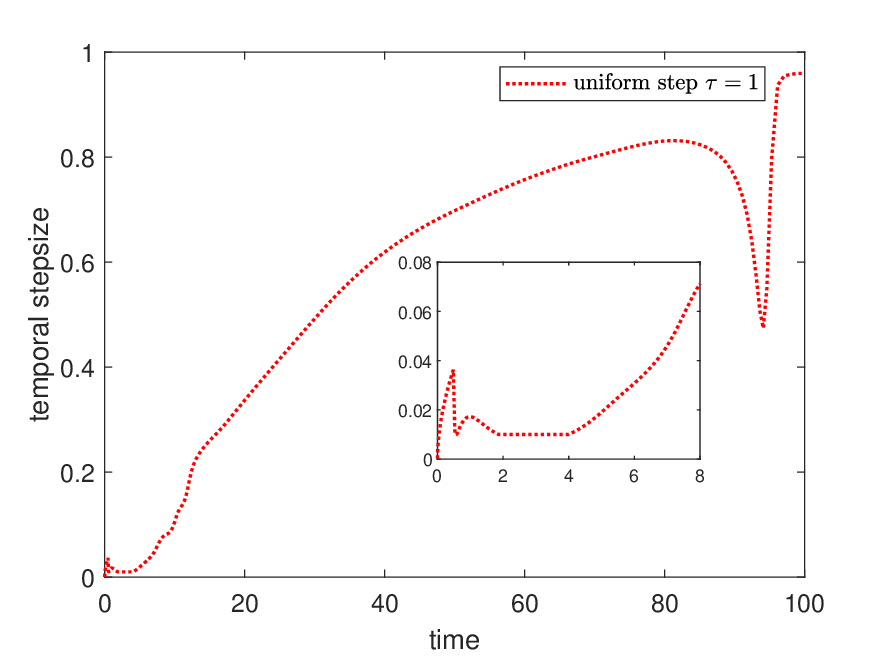}
		\label{figEx3_2c}
	}%
	\setlength{\abovecaptionskip}{0.0cm} 
	\setlength{\belowcaptionskip}{0.0cm}
	\caption{Time evolution of the maximum-norm  (left), energy (middle), and time stepsizes (right)}
\end{figure}

\begin{table}[!htbp]
	\vspace{-12pt}
	\caption{CPU times and the total number of time steps yielded by the scheme}%
	\label{tab1_Ex33}
	{\footnotesize\begin{tabular*}{\columnwidth}{@{\extracolsep\fill}cccc@{\extracolsep\fill}}
			\toprule
			time-stepping strategy  & uniform step  $ \tau = 1$ & adaptive step & uniform step  $ \tau = 0.01 $ \\
			\midrule 
			 $N$        &   131   &  695   &  9981 \\ 
			CPU times   &  5 s  &  37 s  &  57 m  \\ 
			\bottomrule
	\end{tabular*}}
	\vspace{-8pt}
\end{table}

\subsection{Dynamical behavior governed by the tFAC model}
It is well known that the fraction order $\alpha$ plays a crucial role in governing the evolution dynamics of the system. Next, we consider an example involving the merging of four-drops to investigate the influence of $\alpha$ on the evolution process. To this end, we consider the tFAC model \eqref{Model:tAC} with $\varepsilon = 0.02$, and the following initial condition:
\begin{align*}
	\phi_{\text{init}} = -0.9\tanh \frac{ (x-0.3)^2+y^2-0.2^2}{\varepsilon} \tanh \frac{(x+0.3)^2+y^2-0.2^2}{\varepsilon}\\
	\times\tanh \frac{x^2+(y-0.3)^2-0.2^2} {\varepsilon}\tanh \frac{x^2+(y+0.3)^2-0.2^2}{\varepsilon}.
\end{align*}

\begin{figure}[!htbp]
	\vspace{-12pt}
	\centering
	\subfigure[$t=1$]
	{
		\begin{minipage}[t]{0.24\linewidth}
			\centering
			\includegraphics[width=1.43in]{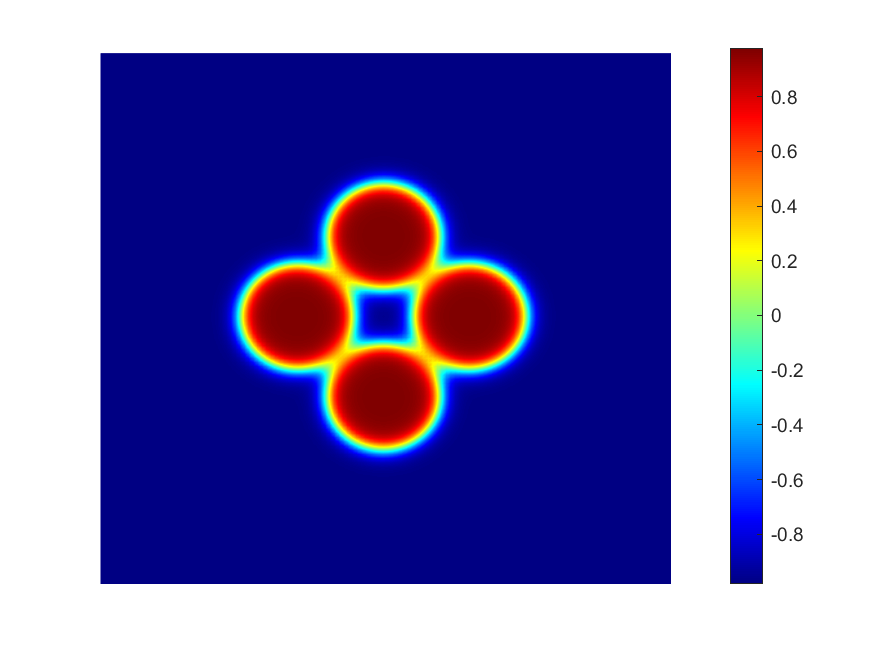}
			\includegraphics[width=1.43in]{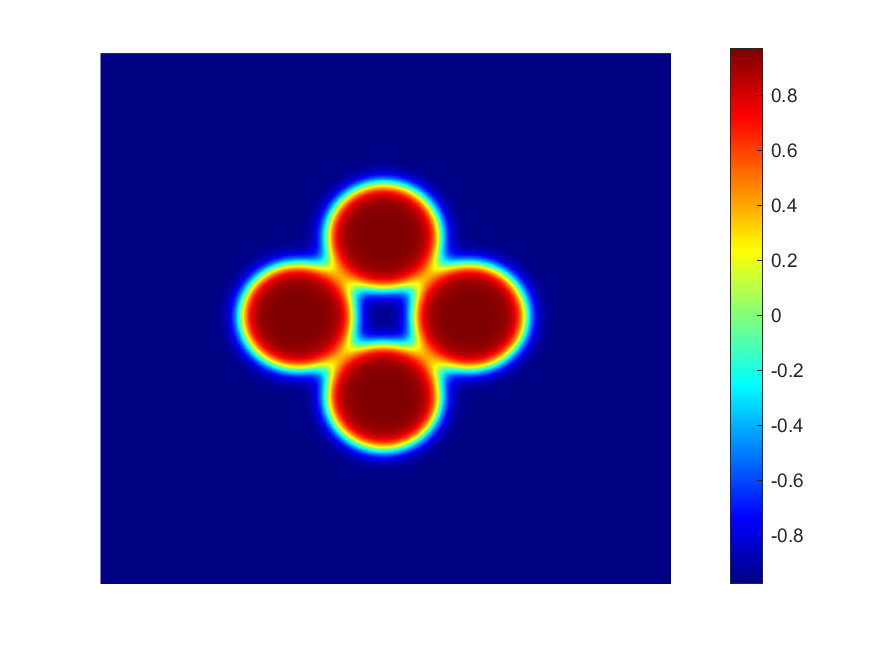}
			\includegraphics[width=1.43in]{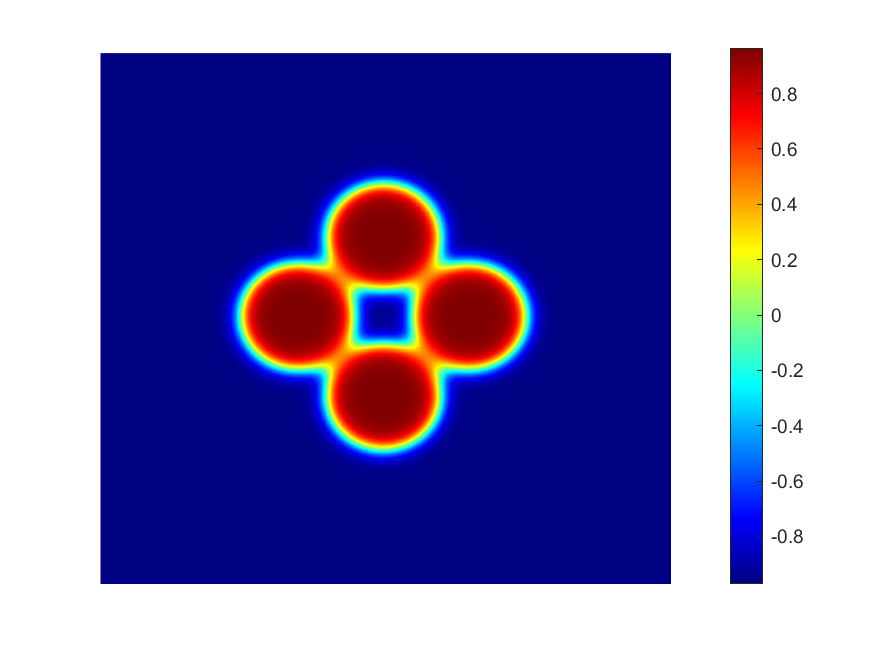}
		\end{minipage}%
	}%
	\subfigure[$t=10$]
	{
		\begin{minipage}[t]{0.24\linewidth}
			\centering
			\includegraphics[width=1.43in]{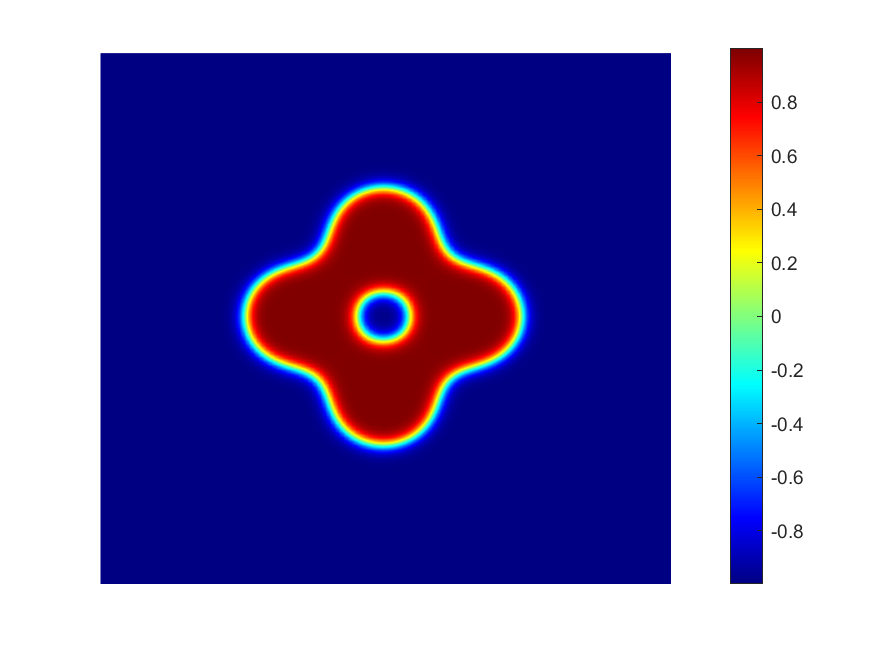}
			\includegraphics[width=1.43in]{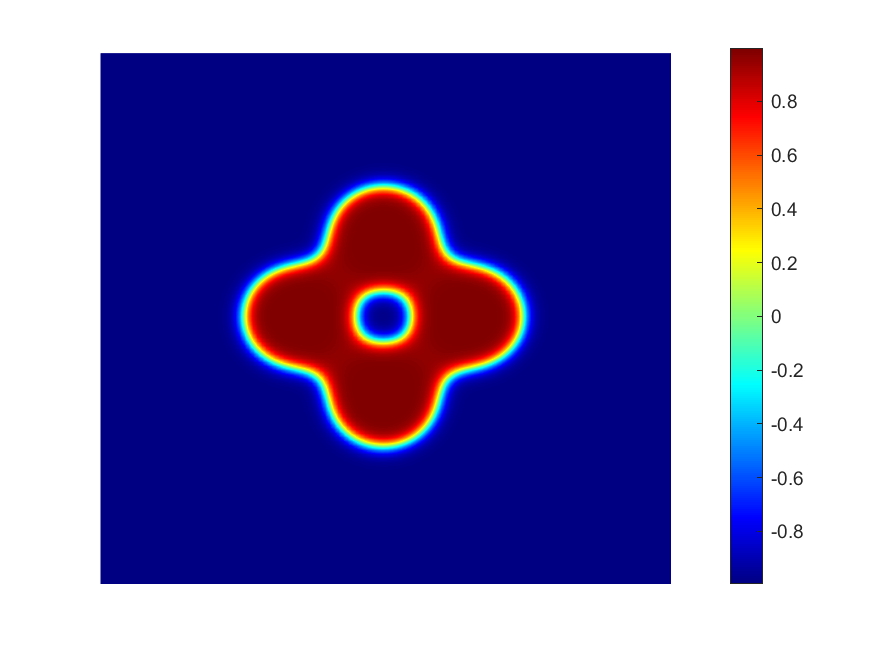}
			\includegraphics[width=1.43in]{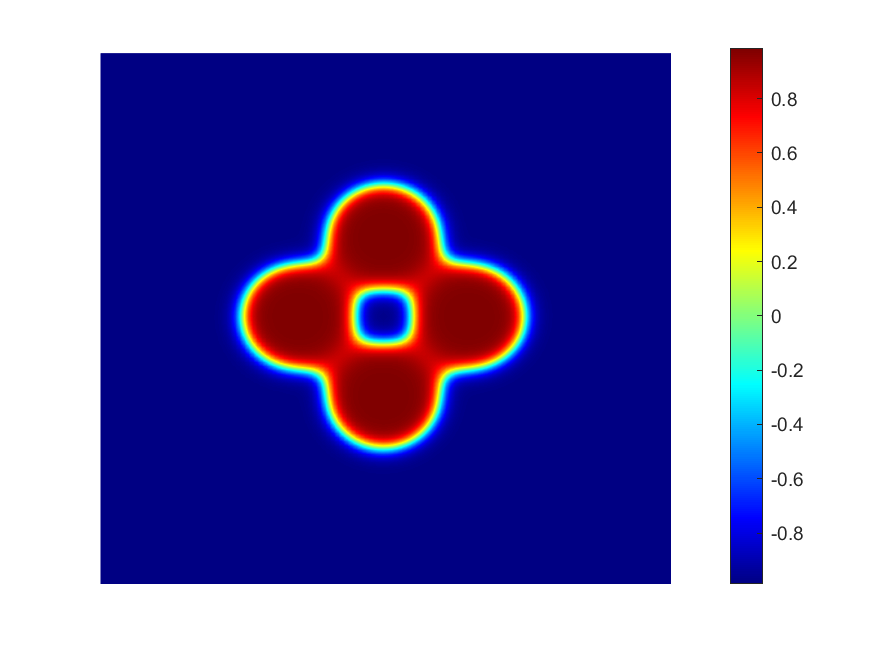}
		\end{minipage}%
	}%
	\subfigure[$t=50$]{
		\begin{minipage}[t]{0.24\linewidth}
			\centering
			\includegraphics[width=1.43in]{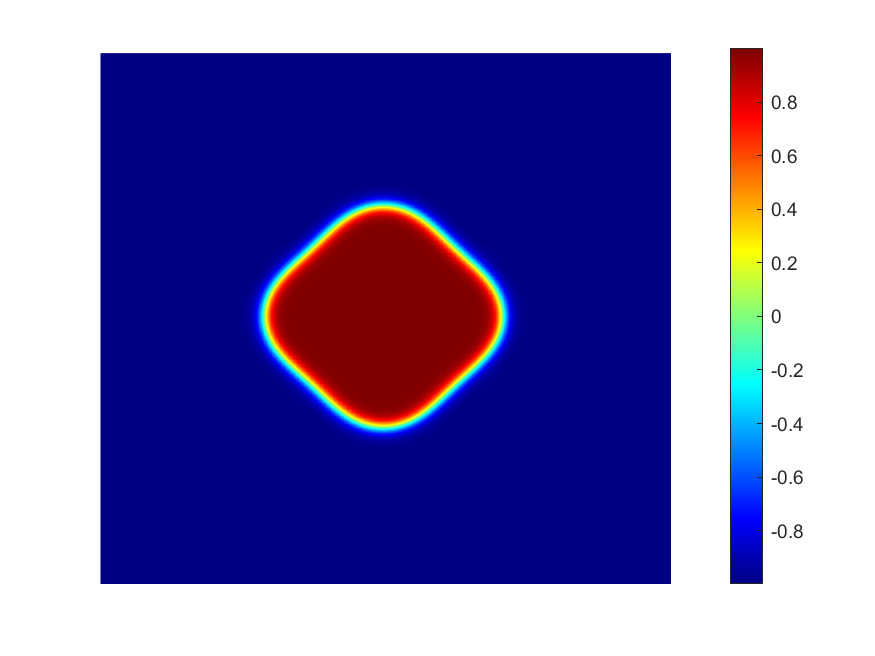}
			\includegraphics[width=1.43in]{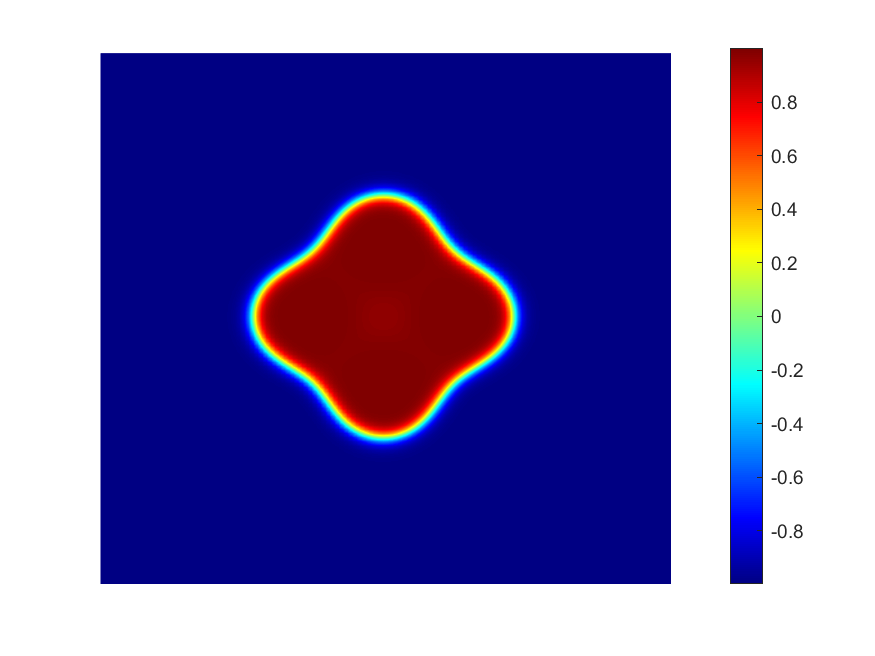}
			\includegraphics[width=1.43in]{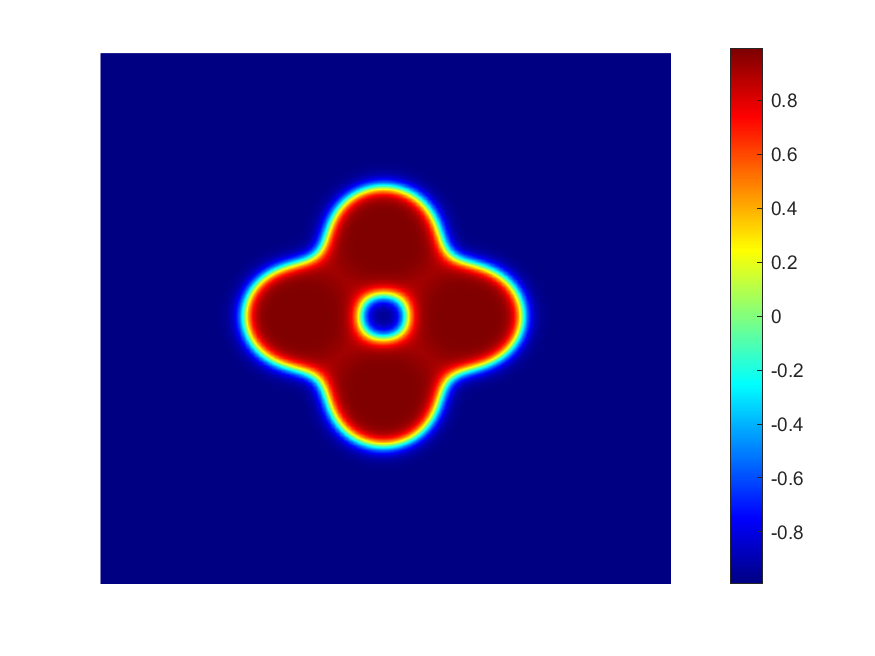}
		\end{minipage}%
	}%
	\subfigure[$t=100$]
	{
		\begin{minipage}[t]{0.24\linewidth}
			\centering
			\includegraphics[width=1.43in]{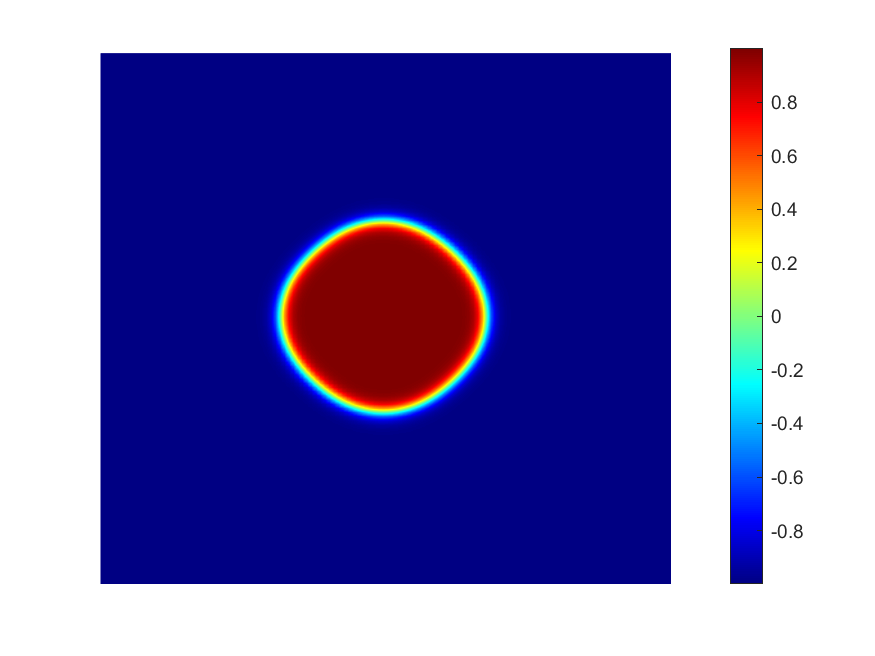}
			\includegraphics[width=1.43in]{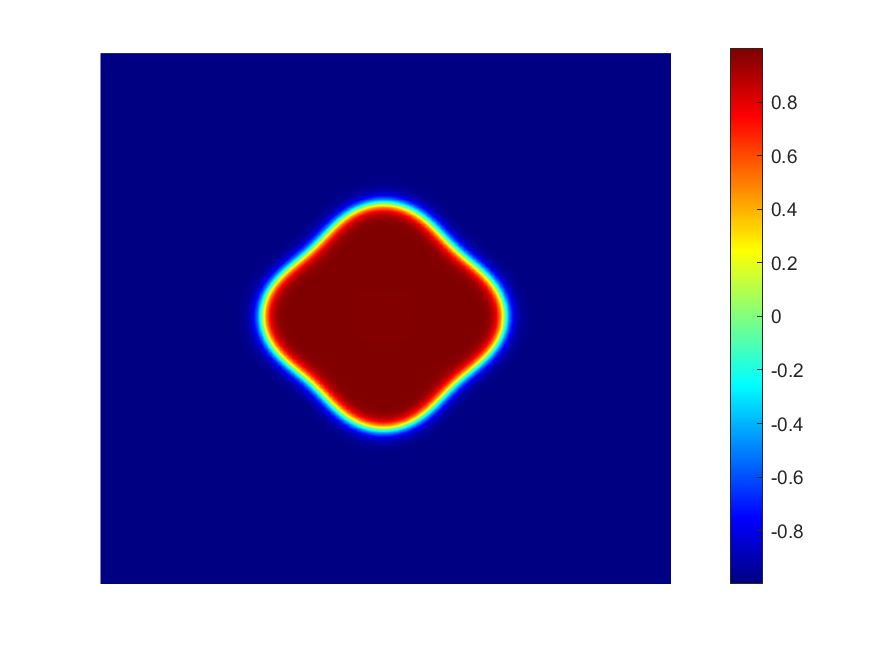}
			\includegraphics[width=1.43in]{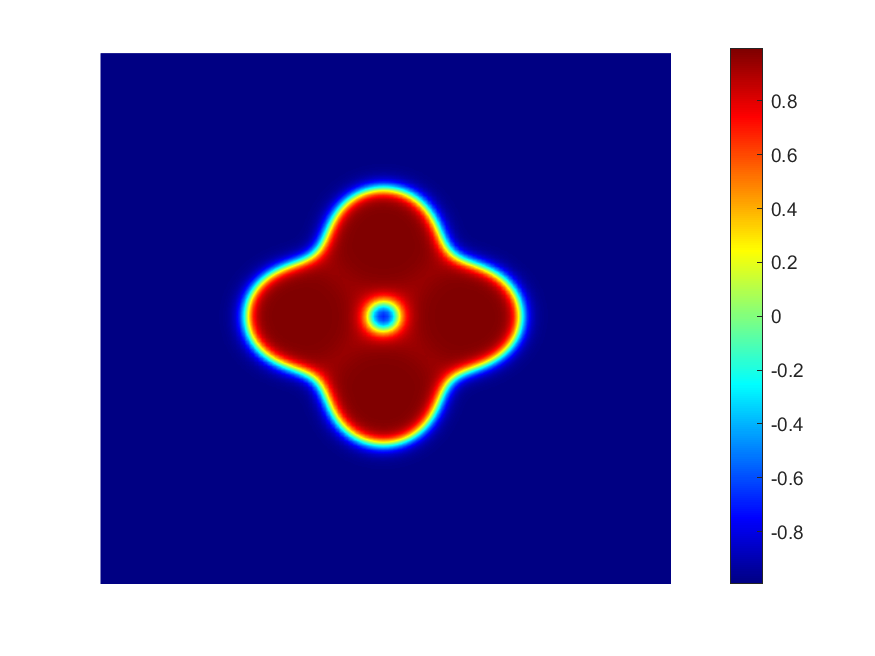}
		\end{minipage}%
	}%
	\setlength{\abovecaptionskip}{0.0cm} 
	\setlength{\belowcaptionskip}{0.0cm}
	\caption{The dynamic snapshots of the numerical solution $\phi$ computed with $\alpha=0.9,0.7,0.4$ (from top to bottom, respectively)}	\label{figEx4_1}
\end{figure}
In this example, simulations are performed using the proposed adaptive IMEX-Alikhanov scheme with parameters $ \tau_{\max} = 1 $, $ \tau_{\min} = 0.01 $, and $ \eta = 10^{7} $. The computational domain $\Omega = (0,1)^2$ is uniformly discretized into $M = 128$ grid points in each spatial direction. Fig. \ref{figEx4_1} illustrates the time evolution of the phase-field variable $\phi$ for different fractional orders $\alpha = 0.9, 0.7$ and $0.4$ at selected time instants. As observed, the initially separated four bubbles gradually coalesce into a single bubble. Moreover, the evolution dynamics is markedly influenced by the fractional order: larger values of $\alpha$ lead to faster evolution processes. This trend can be further confirmed by the maximum norm and energy curves presented in Fig. \ref{figEx4_2a}--\ref{figEx4_2b}. As $\alpha$ decreases, the energy decays more slowly, and it takes a much longer time for $ \phi $ to reach $\pm 1$, which typically indicates the formation of a pure-phase state. In addition, the adaptive time stepsizes are displayed in Fig. \ref{figEx4_2c}. It can be observed that the time stepsize is adjusted according to the energy variation: smaller time stepsizes are employed during periods of rapid energy decay to maintain accuracy, while larger stepsizes are adopted when the energy decreases smoothly. This adaptive strategy significantly enhances the computational efficiency of the long-term simulations.
\begin{figure}[!htbp]
	\vspace{-10pt}
	\centering
	\subfigure[maximum-norm of $\phi$]
	{
		\includegraphics[width=0.325\textwidth]{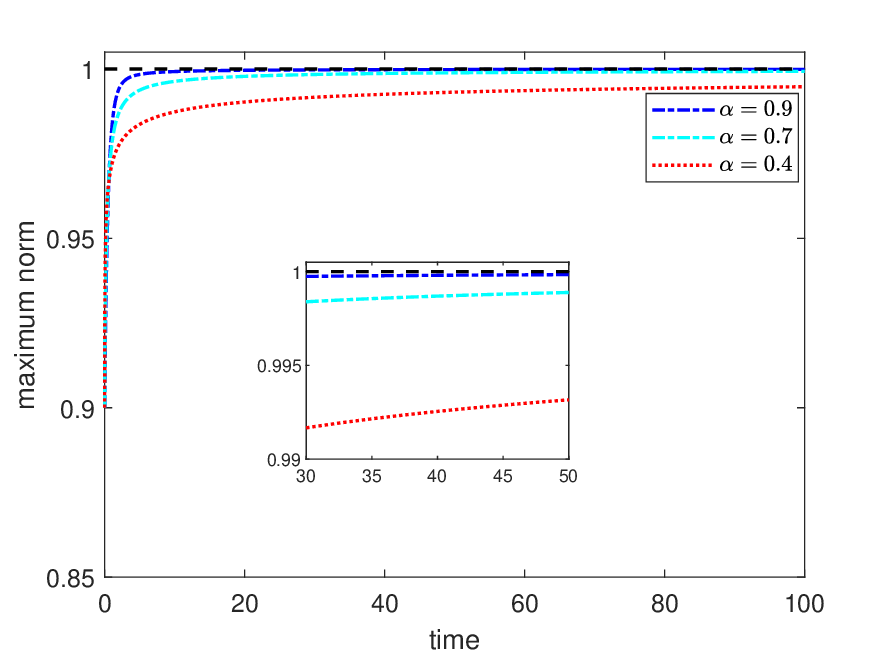}
		\label{figEx4_2a}
	}%
	\subfigure[energy]
	{
		\includegraphics[width=0.325\textwidth]{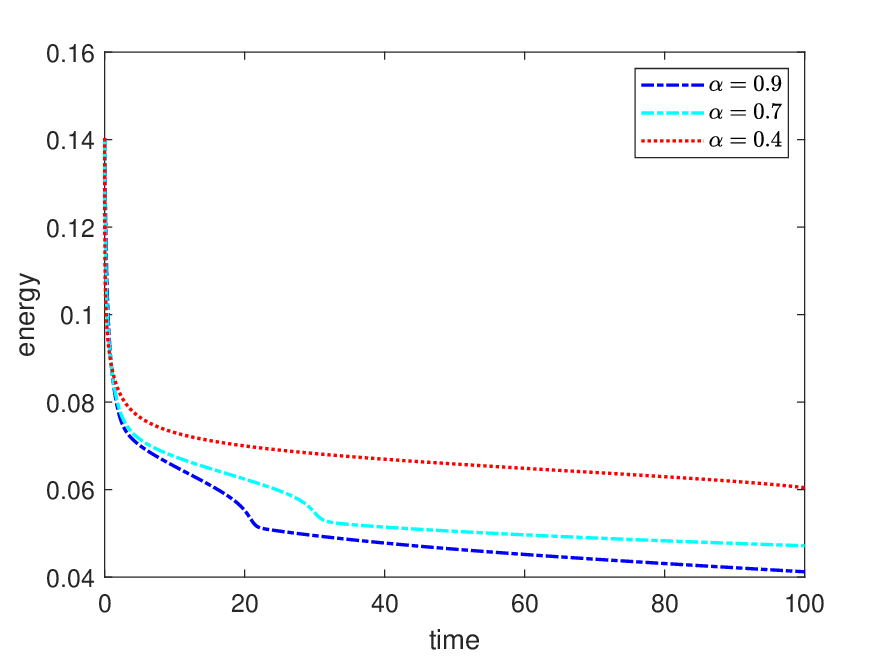}
		\label{figEx4_2b}
	}%
	\subfigure[time stepsizes]
	{
		\includegraphics[width=0.325\textwidth]{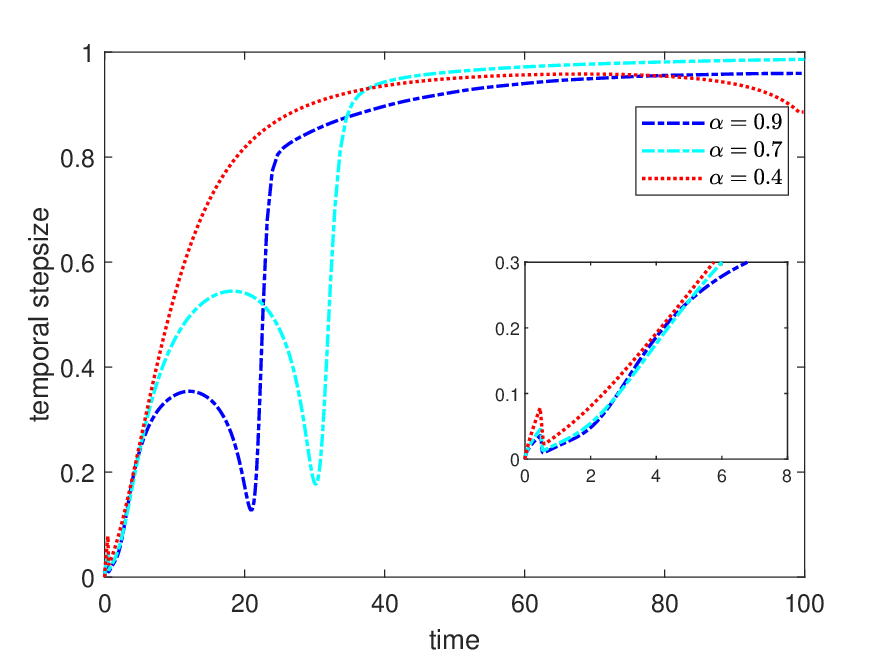}
		\label{figEx4_2c}
	}%
	\setlength{\abovecaptionskip}{0.0cm} 
	\setlength{\belowcaptionskip}{0.0cm}
	\caption{Time evolution of the maximum-norm (left), energy (middle), and time stepsizes (right)}
\end{figure}

\subsection{3D simulations}
In the last example, some three-dimensional simulations are reported for the tFAC model with $\alpha = 0.5$ and $\varepsilon = 0.01$. The phase-field variable is initialized with uniformly distributed random values in the interval $[-0.9,0.9]$. The computational domain is chosen as $\Omega = (-0.5,0.5)^3$. The proposed adaptive IMEX-Alikhanov scheme is employed with parameters $M = 100$, $ \tau_{\max} = 0.5 $, $ \tau_{\min} = 0.005 $, and $ \eta = 10^{7} $.

\begin{figure}[!t]
	\centering
	{
		\begin{minipage}[t]{0.24\linewidth}
			\centering
			\includegraphics[width=1.43in]{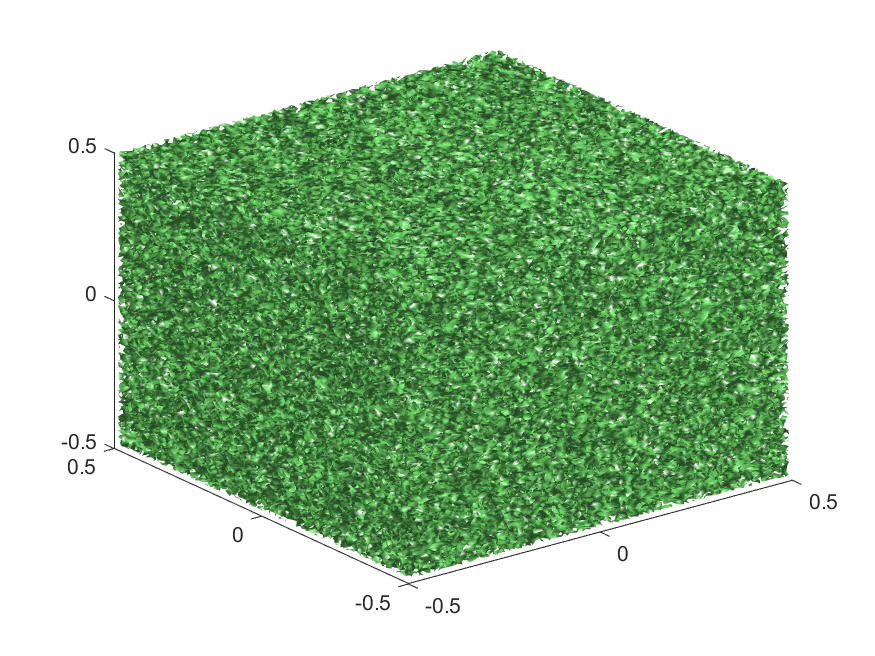}
			\includegraphics[width=1.43in]{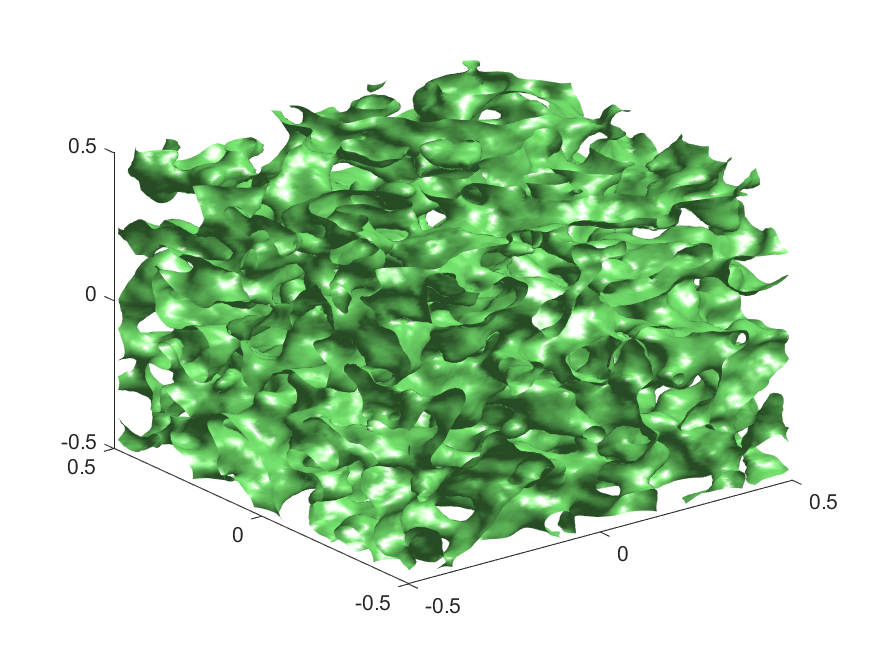}
			\includegraphics[width=1.43in]{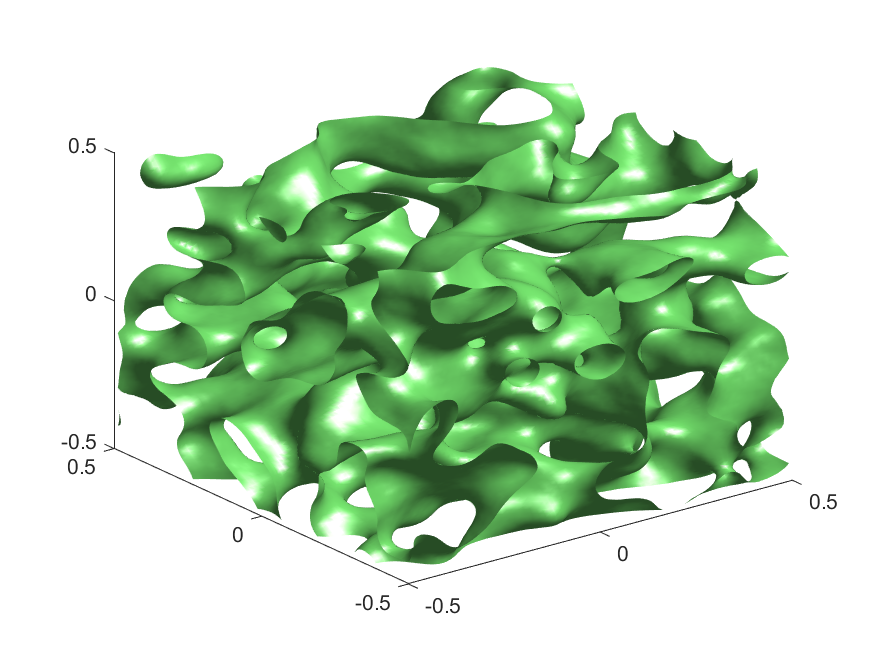}
			\includegraphics[width=1.43in]{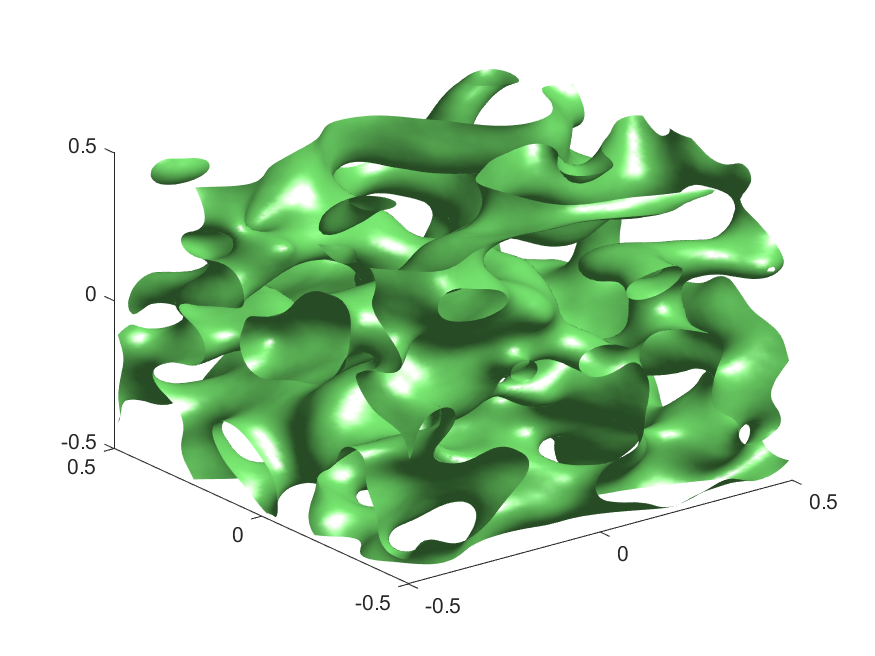}
			\includegraphics[width=1.43in]{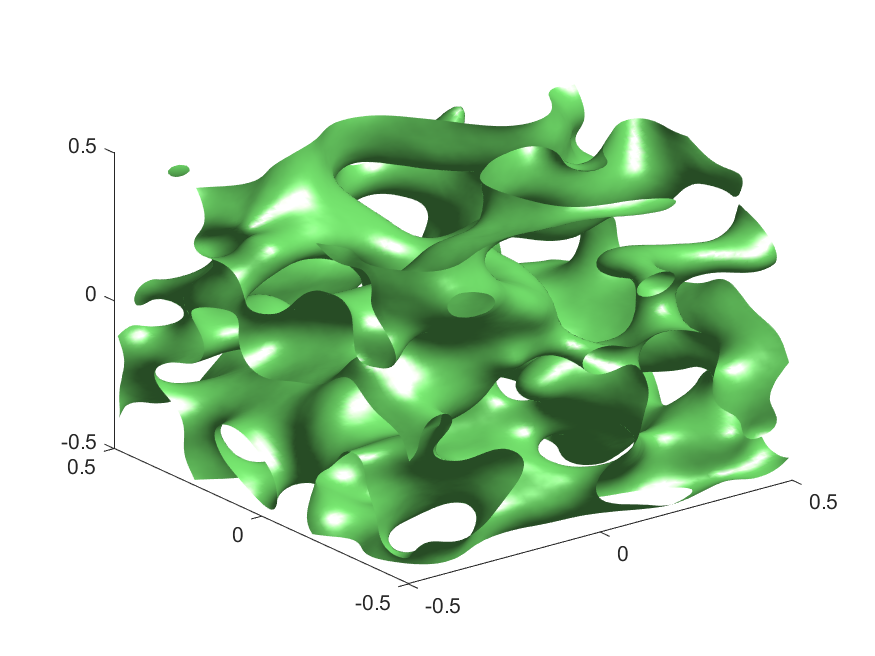}
		\end{minipage}%
	}%
	{
		\begin{minipage}[t]{0.24\linewidth}
			\centering
			\includegraphics[width=1.43in]{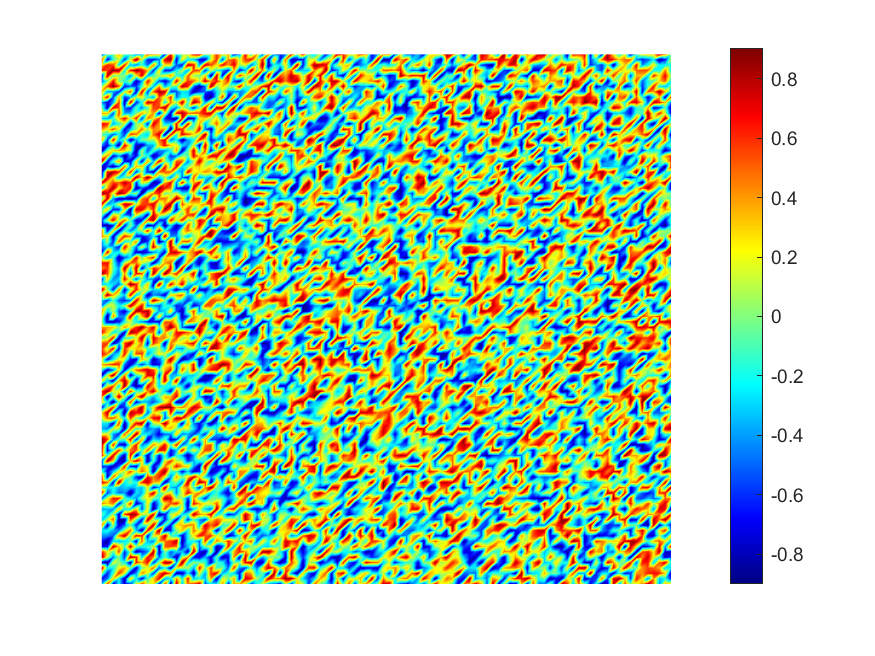}
			\includegraphics[width=1.43in]{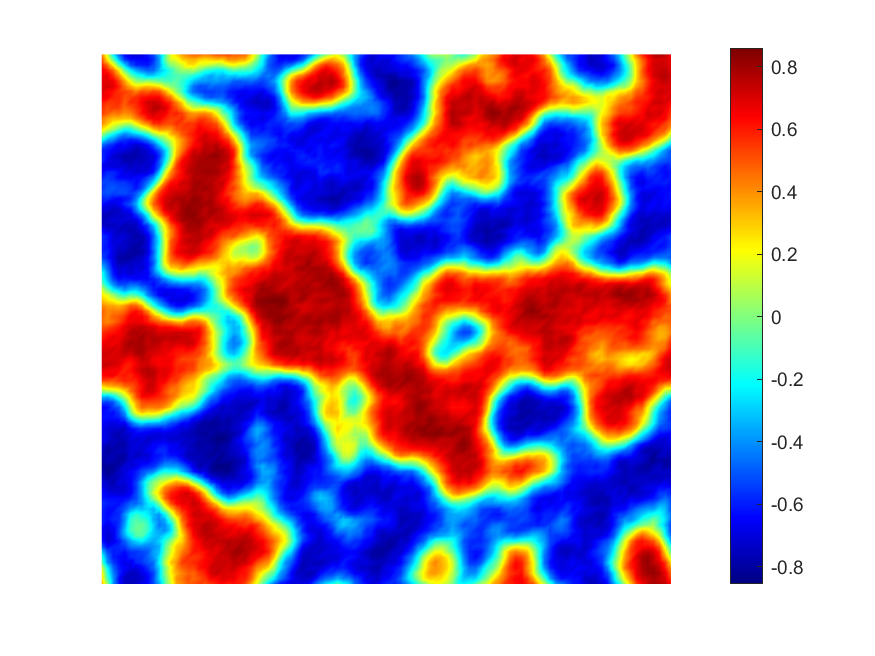}
			\includegraphics[width=1.43in]{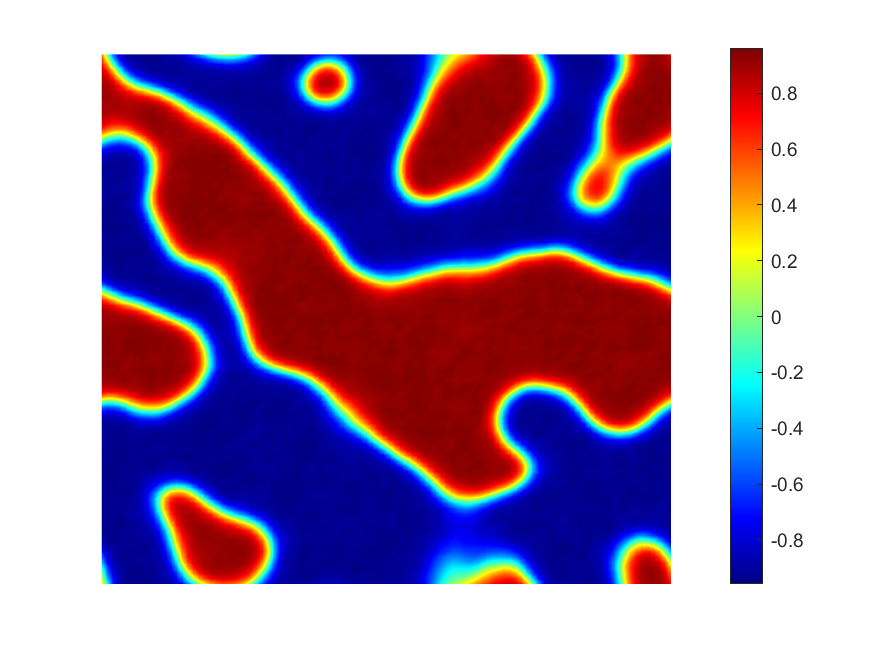}
			\includegraphics[width=1.43in]{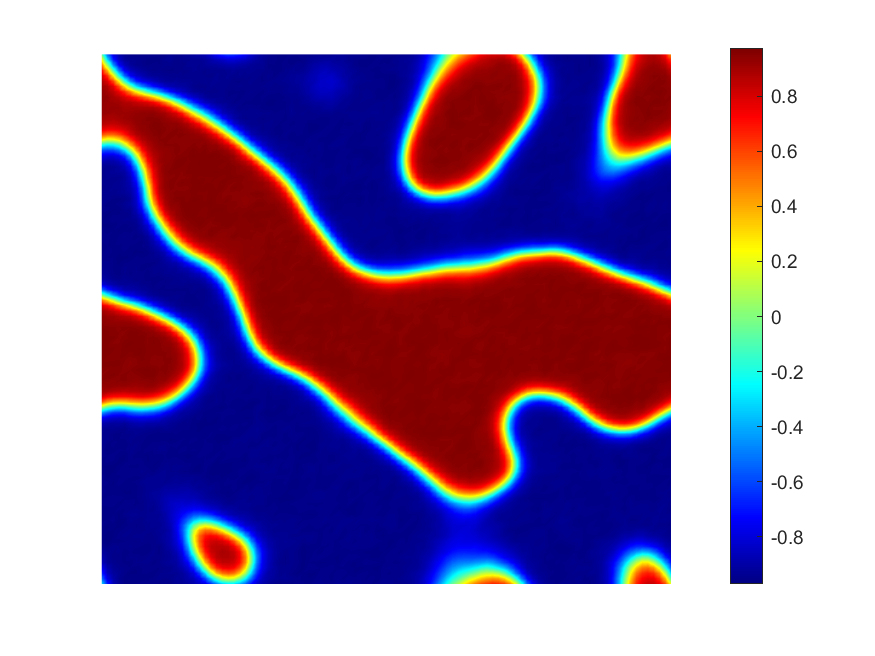}
			\includegraphics[width=1.43in]{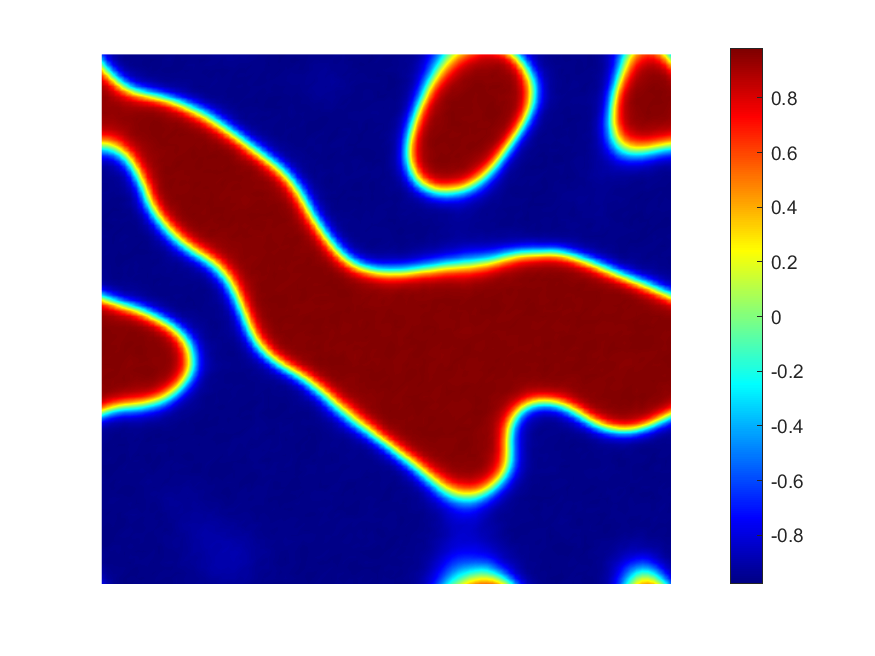}
		\end{minipage}%
	}%
	{
		\begin{minipage}[t]{0.24\linewidth}
			\centering
			\includegraphics[width=1.43in]{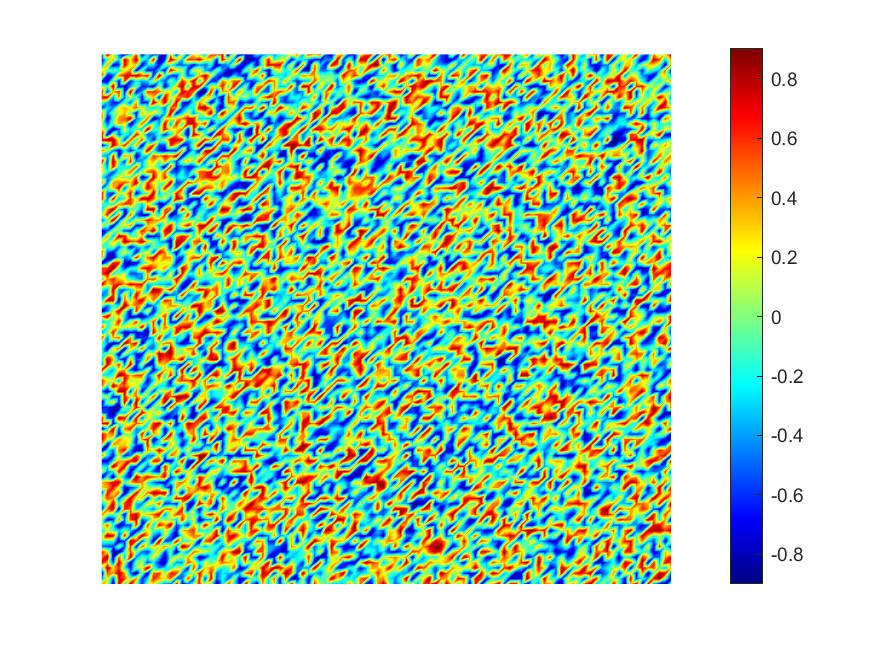}
			\includegraphics[width=1.43in]{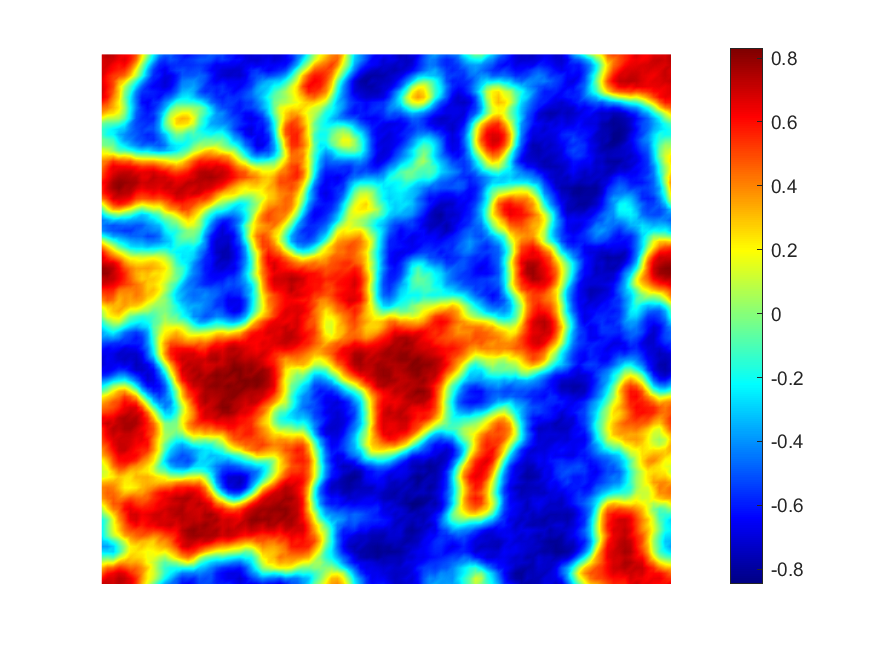}
			\includegraphics[width=1.43in]{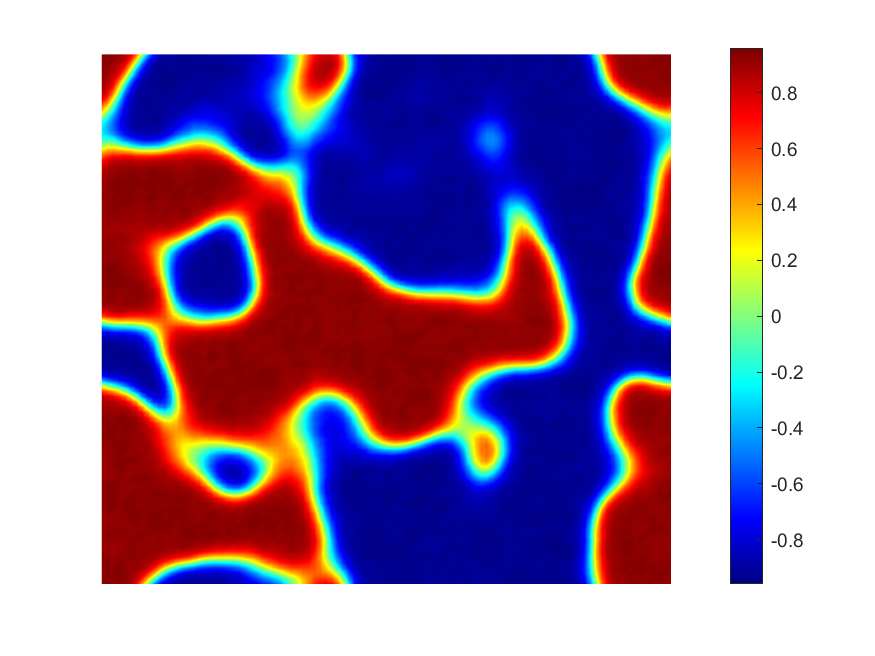}
			\includegraphics[width=1.43in]{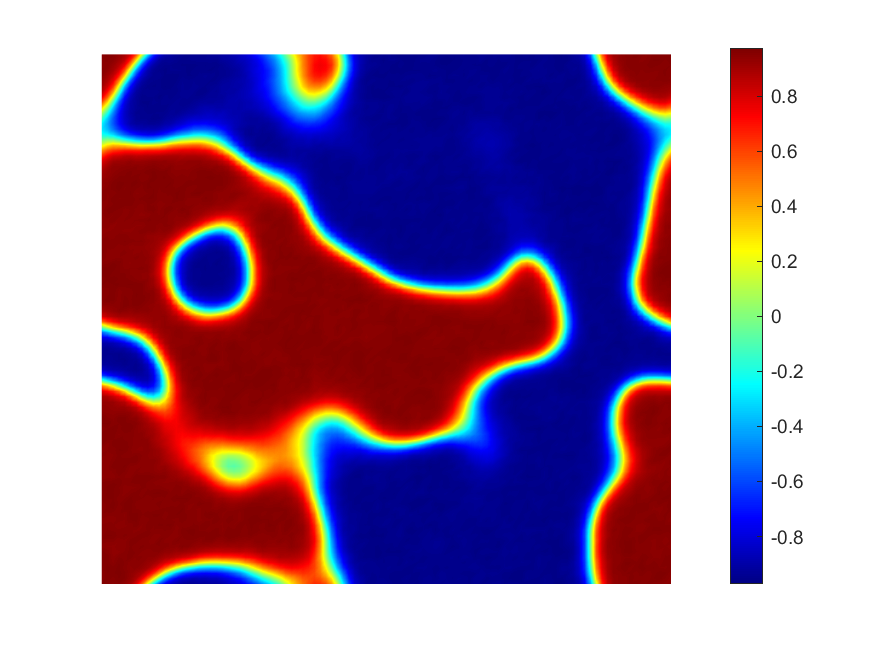}
			\includegraphics[width=1.43in]{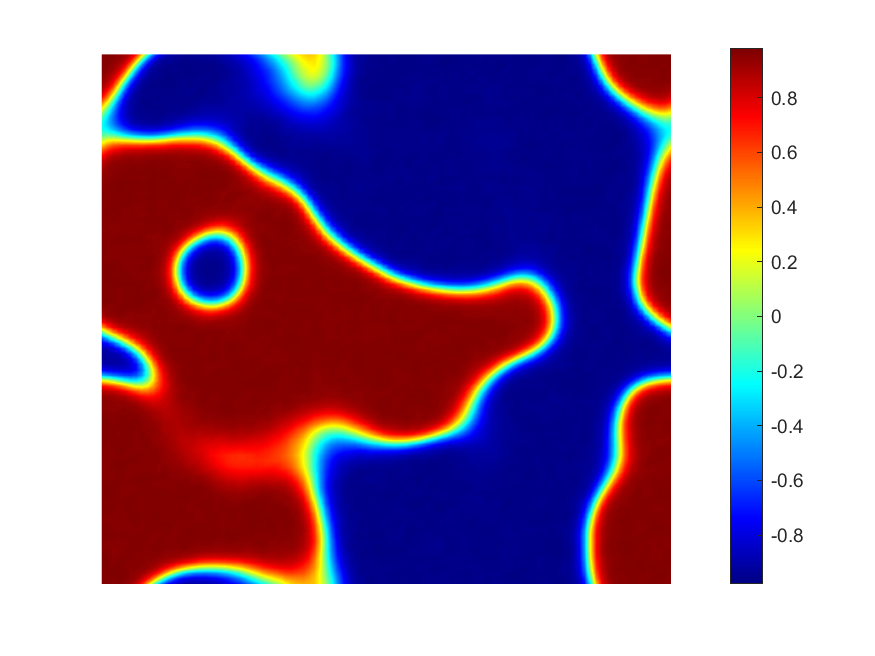}
		\end{minipage}%
	}%
	{
		\begin{minipage}[t]{0.24\linewidth}
			\centering
			\includegraphics[width=1.43in]{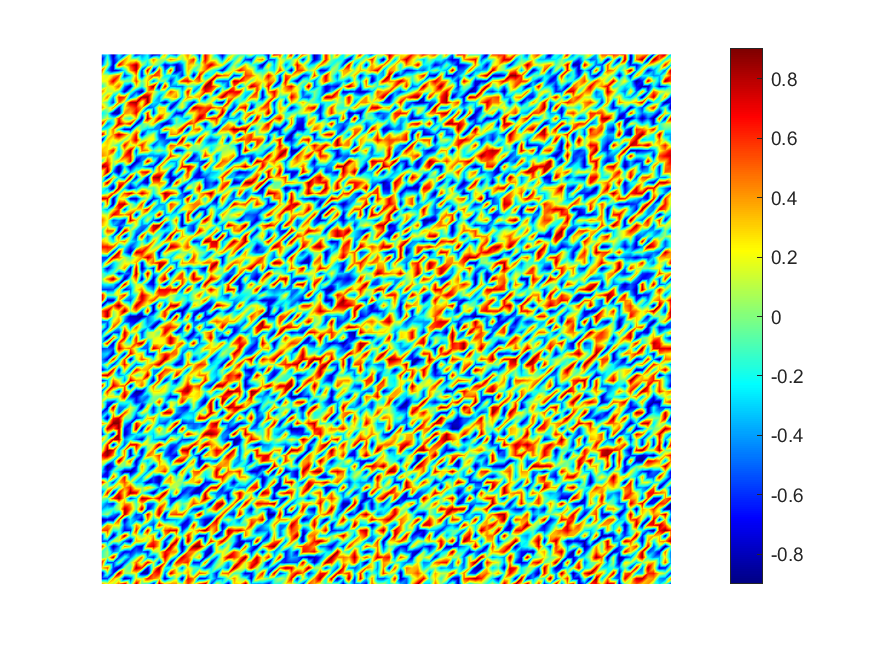}
			\includegraphics[width=1.43in]{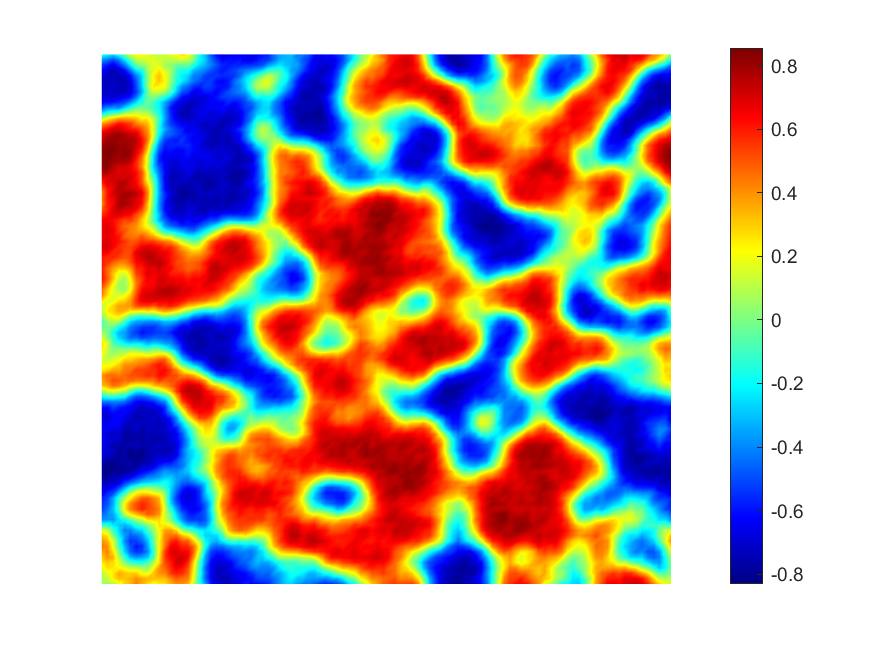}
			\includegraphics[width=1.43in]{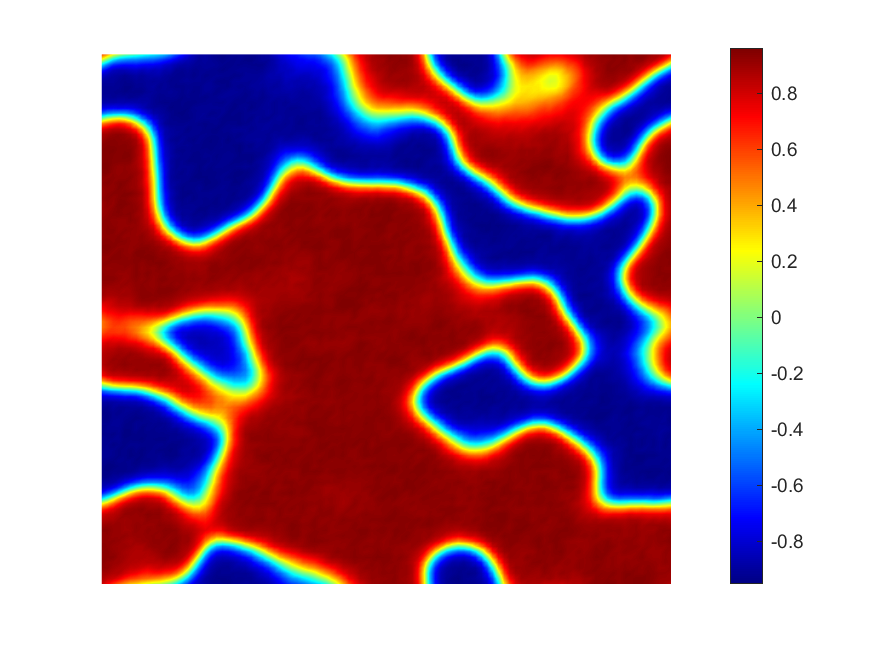}
			\includegraphics[width=1.43in]{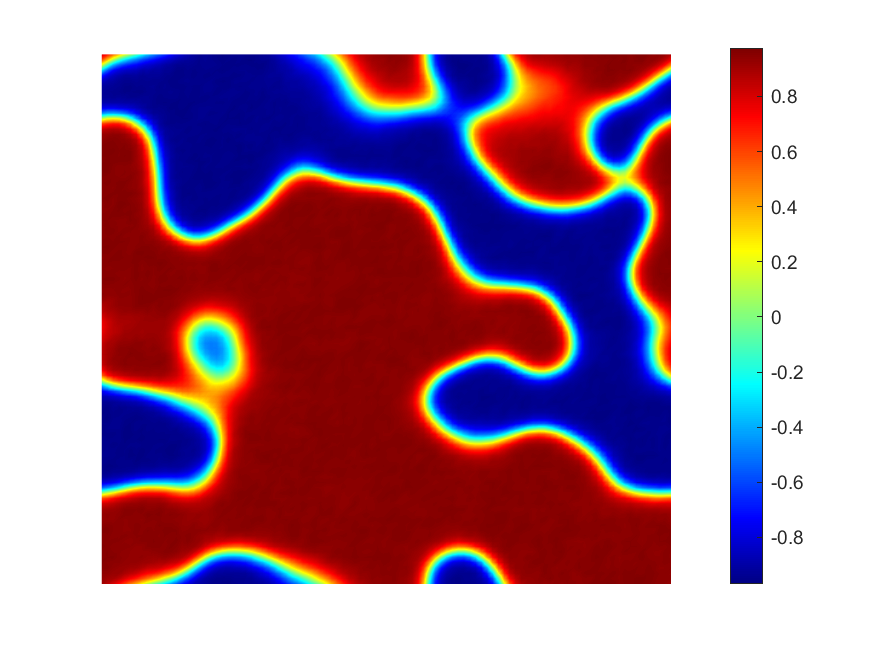}
			\includegraphics[width=1.43in]{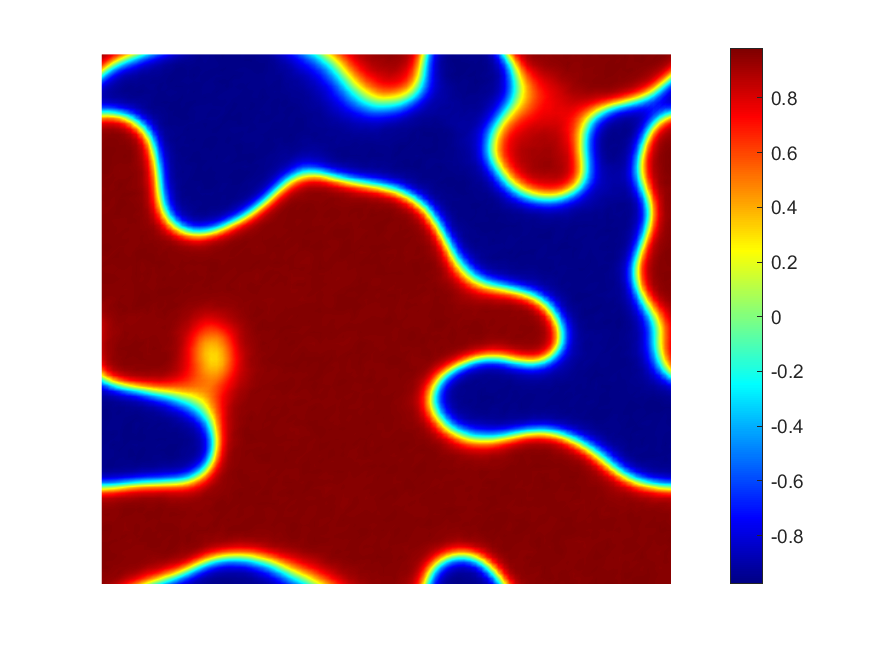}
		\end{minipage}%
	}%
	\setlength{\abovecaptionskip}{0.0cm} 
	\setlength{\belowcaptionskip}{0.0cm}
	\caption{The iso-surfaces of $\phi=0$ and the snapshots of the cross-sections, $x=0,y=0,z=0$ (from left to right) at times $t = 0,5.00,25.01,50.16$ and $80.00$ for the 3D tFAC model}	\label{figEx5_1}
\end{figure}

Fig. \ref{figEx5_1} shows the iso-surfaces formed by the zero level set (i.e., $\phi=0$) of the numerical solution, along with cross-sectional snapshots at $x=0$, $y=0$, and $z=0$ at different time instants. From these results, the ordering and grain coarsening process can be clearly observed during the evolution. The corresponding evolution of the maximum norm, energy, and time stepsizes is presented in Fig. \ref{figEx5_2}. It is again evident that the proposed scheme preserves both energy stability and the MBP well, while the time stepsize is adaptively adjusted in response to variations in the energy.

\begin{figure}[!htbp]
	\vspace{-10pt}
	\centering
	\subfigure[maximum-norm of $\phi$]
	{
		\includegraphics[width=0.325\textwidth]{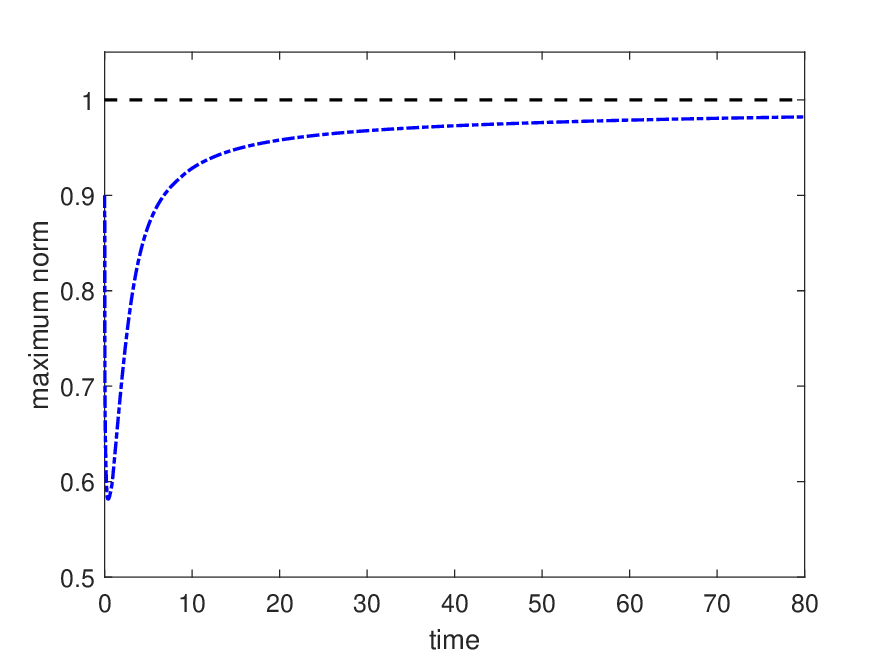}
		\label{figEx5_2a}
	}%
	\subfigure[energy]
	{
		\includegraphics[width=0.325\textwidth]{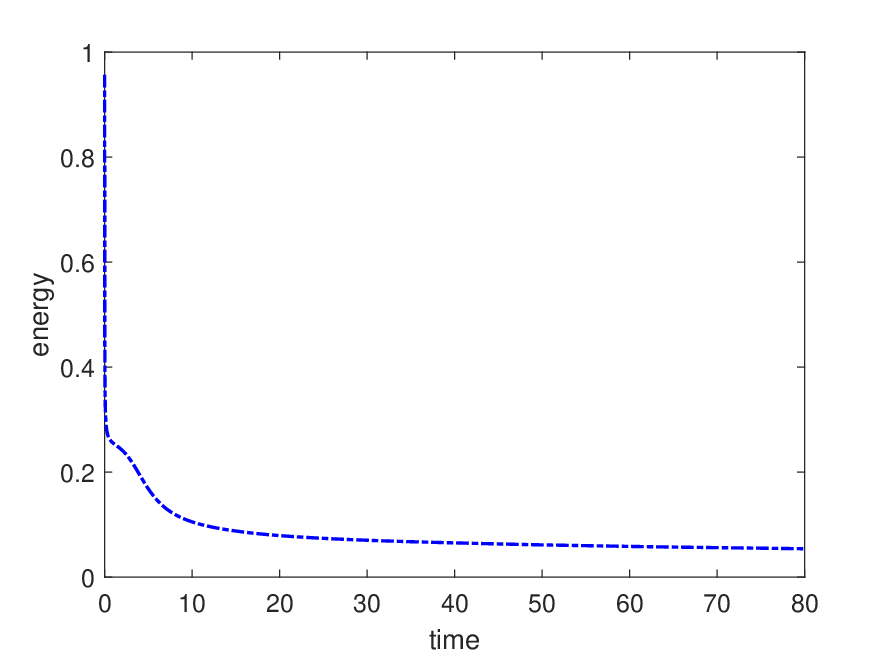}
		\label{figEx5_2b}
	}%
	\subfigure[time stepsizes]
	{
		\includegraphics[width=0.325\textwidth]{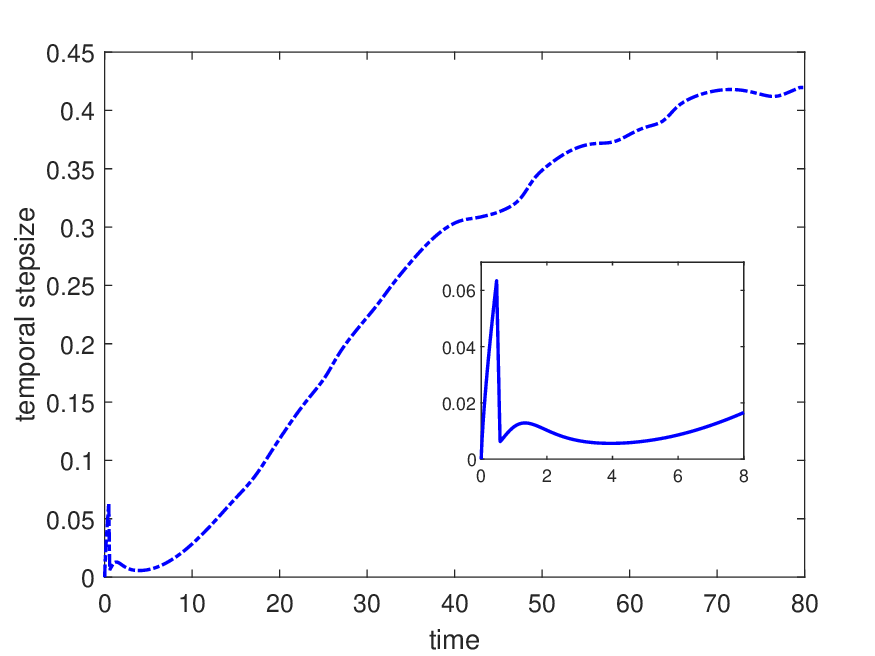}
		\label{figEx5_2c}
	}%
	\setlength{\abovecaptionskip}{0.0cm} 
	\setlength{\belowcaptionskip}{0.0cm}
	\caption{Time evolutions of the maximum-norm bound (left), energy (middle), and time stepsizes (right)}
	\label{figEx5_2}
\end{figure}

\section{ Concluding remarks }\label{Sec:Clu}
In this paper, we have developed a variable time-step, second-order linear IMEX-Alikhanov scheme for the tFAC model based on an appropriate stabilization technique. In particular, the nonlinear potential is treated explicitly via a second-order extrapolation with preprocessing to ensure that the proposed scheme preserves the discrete MBP. Moreover, rigorous maximum-norm error estimates and discrete energy stability analysis have also been established. As our analysis is built on variable time stepsize, the intrinsic initial singularity can be effectively resolved by employing nonuniform meshes such as the graded mesh, and furthermore, the long-term dynamical evolution governed by the tFAC model can be accurately and efficiently simulated using an adaptive time-stepping strategy. Extensive numerical experiments confirm the reliability and efficiency of the proposed approach. 

Finally, we briefly comment on the dependence of the time stepsize conditions (see \eqref{MBP:L21_tau}, \eqref{Condi:Cover}, and \eqref{Condi:energy}) arising in the theoretical analysis on the fractional order $\alpha$. Due to the presence of the exponent $1/\alpha$, these conditions can become extremely restrictive as $ \alpha \rightarrow 0^{+} $. Similar conditions also appear in existing numerical methods for linear and nonlinear time-fractional model problems (see \cite{JCP_Liao_2020,SINUM_Liao_2018,SISC_Liao_2021,JSC_Fu_2023,JSC_Huang_2022,JCP_Zhang_2025}). Consequently, removing this restrictive condition for small $\alpha$ in the theoretical analysis is a meaningful and interesting endeavor.

\section*{CRediT authorship contribution statement}
			\textbf{Bingyin Zhang}: Methodology, Formal analysis, Software, Writing-Original draft.
			\textbf{Ao Zhang}: Methodology, Software, Writing-Original draft.
			\textbf{Hongfei Fu}: Conceptualization, Supervision, Writing-Reviewing and Editing, Funding acquisition.

\section*{Declaration of competing interest} The authors declare that they have no competing interests.
			
\section*{Data availability}   Data will be made available on request.

\section*{Acknowledgements}
This work was supported in part by the National Natural Science Foundation of China (Nos. 11971482, 12131014), by the Natural Science Foundation of Shandong Province (No. ZR2024MA023), and by the Shandong Provincial Key Laboratory of Stochastic System Control and Scientific Computing.


\bibliographystyle{spmpsci}      
\bibliography{Ref_tAC}

@article{JSC_Chen_2019,
	author = {H Chen and M Stynes},
	title = {Error Analysis of a Second-order Method on Fitted Meshes
	for a Time-Fractional Diffusion Problem},
	journal = {J. Sci. Comput.},
	volume = {79},
	pages = {624-647},
	year = {2019},
}

@article{JSC_Huang_2022,
	author = {C Huang and M Stynes},
	title = {A Sharp $\alpha$-robust {$ L^{\infty}( H^{1} ) $} Error Bound for a Time-Fractional
{A}llen--{C}ahn Problem Discretised by the {A}likhanov {$L2$-$1_{\sigma}$} scheme and a standard {FEM}},
	journal = {J. Sci. Comput.},
	volume = {91},
	pages = {43},
	year = {2022},
}

@article{JCAM_Hu_2026,
	author = {D Hu and M Chen and H Jiang and H Huang},
	title = {Energy dissipation and maximum-bound principle of the variable-step {$L2$-$1_{\sigma}$} scheme for the time-fractional {A}llen--{C}ahn
	equation with general nonlinear potential},
	journal = {J. Comput. Appl. Math.},
	volume = {475},
	pages = {117054},
	year = {2026},
}

@article{JSC_Huang_2023,
	author = {G Zhang and C Huang and A.A. Alikhanov and B Yin},
	title = {A High-Order Discrete Energy Decay and Maximum-Principle Preserving Scheme for Time Fractional {Allen--Cahn} Equation},
	journal = {J. Sci. Comput.},
	volume = {96},
	pages = {39},
	year = {2023},
}

@article{ACM_Ji_2020,
	author = {B Ji and H Liao and L Zhang},
	title = {Simple maximum principle preserving time-stepping methods for time-fractional {Allen--Cahn} equation},
	journal = {Adv. Comput. Math.},
	volume = {46},
	pages = {37},
	year = {2020},
}

@article{SISC_Liao_2021,
	author = {H Liao and T Tang and T Zhou},
	title = {AN ENERGY STABLE AND MAXIMUM BOUND PRESERVING SCHEME WITH VARIABLE TIME STEPS FOR TIME FRACTIONAL
	{A}LLEN--{C}AHN EQUATION},
	journal = {SIAM J. Sci. Comput.},
	volume = {43},
	pages = {A3503--A3526},
	year = {2021},
}

@article{SISC_Ji_2020,
	author = {B Ji and H Liao and Y Gong and L Zhang},
	title = {ADAPTIVE SECOND-ORDER {C}RANK--{N}ICOLSON TIME-STEPPING SCHEMES FOR TIME-FRACTIONAL MOLECULAR BEAM EPITAXIAL GROWTH MODELS},
	journal = {SIAM J. Sci. Comput.},
	volume = {42},
	pages = {B738--B760},
	year = {2020},
}

@article{SISC_Hou_2021,
	author = {D Hou and C Xu},
	title = {HIGHLY EFFICIENT AND ENERGY DISSIPATIVE SCHEMES FOR THE TIME FRACTIONAL {A}LLEN--{C}AHN EQUATION},
	journal = {SIAM J. Sci. Comput.},
	volume = {43},
	pages = {A3305--A3327},
	year = {2021},
}

@article{SISC_Zhao_2024,
	author = {R Qi and X Zhao},
	title = {A UNIFIED DESIGN OF ENERGY STABLE SCHEMES WITH VARIABLE STEPS FOR FRACTIONAL GRADIENT FLOWS AND NONLINEAR INTEGRO-DIFFERENTIAL EQUATIONS},
	journal = {SIAM J. Sci. Comput.},
	volume = {46},
	pages = {A130--A155},
	year = {2024},
}

@article{JCP_Quan_2022,
	author = {C Quan and B Wang},
	title = {Energy stable {$L2$} schemes for time-fractional phase-field equations},
	journal = {J. Comput. Phys.},
	volume = {458},
	pages = {111085},
	year = {2022},
}

@article{JCP_Shen_2018,
	author = {J Shen and J Xu and J Yang},
	title = {The scalar auxiliary variable ({SAV}) approach for gradient flows},
	journal = {J. Comput. Phys.},
	volume = {353},
	pages = {407--416},
	year = {2018},
}

@article{JCP_Wang_2017,
	author = {Z Li and H Wang and D Yang},
	title = {A space-time fractional phase-field model with tunable sharpness and decay behavior and its efficient numerical simulation},
	journal = {J. Comput. Phys.},
	volume = {347},
	pages = {20--38},
	year = {2017},
}

@article{CPC_Wang_2019,
	author = {L. Chen and J. Zhang and J. Zhao and W. Cao and H. Wang and J. Zhang},
	title = {An accurate and efficient
	algorithm for the time-fractional molecular beam epitaxy model with slope selection},
	journal = {Comput. Phys. Commun.},
	volume = {245},
	pages = {106842},
	year = {2019},
}

@article{PRL_Chepizhko_2013,
	author = {O. Chepizhko and F. Peruani},
	title = {Diffusion, subdiffusion, and trapping of active particles in heterogeneous media},
	journal = {Phys. Rev. Lett.},
	volume = {111},
	pages = {160604},
	year = {2013},
}

@article{WRR_Schumer_2003,
	author = {R. Schumer and D. Benson and M.M. Meerschaert and B. Baeumer},
	title = {Fractal mobile/immobile solute transport},
	journal = {Water Resour. Res.},
	volume = {39},
	pages = {1296},
	year = {2003},
}

@article{JCP_Zhokn_2017,
	author = {A. Zhokh and P. Strizhak},
	title = {Non-{F}ickian diffusion of methanol in mesoporous media: geometrical restrictions or adsorption-induced?},
	journal = {J. Chem. Phys.},
	volume = {146},
	pages = {124704},
	year = {2017},
}

@article{MS_Sharma_2015,
	author = {A. Sharma and S. Namsani and J.K. Singh},
	title = {Molecular simulation of shale gas adsorption and diffusion in inorganic nanopores},
	journal = {Mol. Simul.},
	volume = {41},
	pages = {414--422},
	year = {2015},
}

@article{Acta_Allen_1979,
	author = {S Allen and J Cahn},
	title = {A microscopic theory for antiphase boundary motion and its application to antiphase domain coarsening},
	journal = {Acta Metall.},
	volume = {27},
	pages = {1085--1095},
	year = {1979},
}

@Article{JCP_2006_Yang,
	title =	 {Numerical simulations of jet pinching-off and drop formation using an energetic variational phase-field method},
	author =	 {Yang, X. and Feng, J.J. and Liu, C. and Shen, J.},
	journal =	 {J. Comput. Phys.},
	volume =	 218,
	year =	 2006,
	pages =	 {417--428},
}

@Article{CiCP_2012_Kim,
	title =	 {Phase-Field Models for Multi-Component Fluid Flows},
	author =	 {Kim, J.},
	journal =	 {Commun. Comput. Phys.},
	volume =	 12,
	year =	 2012,
	pages =	 {613--661},
}

@Article{SISC_2010_Shen,
	title =	 {A PHASE-FIELD MODEL AND ITS NUMERICAL APPROXIMATION FOR TWO-PHASE INCOMPRESSIBLE FLOWS WITH DIFFERENT DENSITIES AND VISCOSITIES},
	author =	 {Shen, J. and Yang, X.},
	journal =	 {SIAM J. Sci. Comput.},
	volume =	 32,
	year =	 2010,
	pages =	 {1159--1179},
}

@Article{NM_2003_Feng,
	title =	 {Numerical analysis of the {A}llen--{C}ahn equation and approximation for mean curvature flows},
	author =	 {Feng, X.B. and Prohl, A.},
	journal =	 {Numer. Math.},
	volume =	 94,
	year =	 2003,
	pages =	 {33--65},
}

@Article{JDG_1993_Ilmanen,
	title =	 {Convergence of the {A}llen--{C}ahn equation to Brakke's motion by mean curvature},
	author =	 {T. Ilmanen},
	journal =	 {J. Differential Geom.},
	volume =	 38,
	year =	 1993,
	pages =	 {417--461},
}

@Article{ANM_2004_Benes,
	title =	 {Geometrical image segmentation by the {A}llen--{C}ahn equation},
	author =	 {M. Benes and V. Chalupecky and K. Mikula},
	journal =	 {Appl. Numer. Math.},
	volume =	 51,
	year =	 2004,
	pages =	 {187--205},
}

@article{LiuQiaoZhang,
	author = {Liu, C. and Qiao, Z. and Zhang, Q.},
	title = {Two-Phase Segmentation for Intensity Inhomogeneous Images by the {A}llen--{C}ahn Local Binary Fitting Model},
	journal = {SIAM J. Sci. Comput.},
	volume = {44},
	number = {1},
	pages = {B177-B196},
	year = {2022}
}

@Article{JCP_Zhang_2025,
	title =	 {High-order nonuniform time-stepping and {MBP}-preserving linear schemes for the time-fractional {Allen--Cahn} equation},
	author =	 {B Zhang and H Wang and H Fu},
	journal =	 {J. Comput. Phys.},
	volume =	 551,
	year =	 2026,
	pages =	 {114694},
}

@Article{SINUM_Mustapha_2014,
	title =	 {A DISCONTINUOUS {P}ETROV-{G}ALERKIN METHOD FOR
	TIME-FRACTIONAL DIFFUSION EQUATIONS},
	author =	 {K Mustapha and B Abdaallah and K Furati},
	journal =	 {SIAM J. Numer. Anal.},
	volume =	 52,
	year =	 2014,
	pages =	 {2512--2529},
}

@Article{IMA_Jin_2016,
	title =	 {An analysis of the {$L1$} scheme for the subdiffusion equation with nonsmooth data},
	author =	 {B Jin and R  Lazarov and Z Zhou},
	journal =	 {IMA J. Numer. Anal.},
	volume =	 36,
	year =	 2016,
	pages =	 {197--221},
}

@Article{SINUM_Stynes_2017,
	title =	 {Error analysis of a finite difference method on graded meshes for a time-fractional diffusion equation},
	author =	 {M Stynes and E O'Riordan and J Gracia},
	journal =	 {SIAM J. Numer. Anal.},
	volume =	 55,
	year =	 2017,
	pages =	 {1057--1079},
}

@Article{SISC_Qiao_2011,
	title =	 {An adaptive time-stepping strategy for the molecular beam epitaxy models},
	author =	 {Z Qiao and Z Zhang and T Tang},
	journal =	 {SIAM J. Sci. Comput.},
	volume =	 33,
	year =	 2011,
	pages =	 {1395--1414},
}

@Article{IMA_Chen_2021,
	title =	 {Blow-up of error estimates in time-fractional initial-boundary value problems},
	author =	 {Hu Chen and Martin Stynes},
	journal =	 {IMA J. Numer. Anal.},
	volume =	 41,
	year =	 2021,
	pages =	 {974--997},
}

@article{JSC_Fu_2023,
	author = {H Fu and B Zhang and X Zheng},
	title = {A High-Order Two-Grid Difference Method for Nonlinear Time-Fractional Biharmonic Problems and Its Unconditional $\alpha$-Robust Error Estimates},
	journal = {J. Sci. Comput.},
	volume = {96},
	pages = {54},
	year = {2023},
}

@article{SINUM_Liao_2019,
	author = {H Liao and W Mclean and J Zhang},
	title = {A discrete {G}r\"onwall inequality with applications to numerical schemes for subdiffusion problems},
	journal = {SIAM J. Numer. Anal.},
	volume = {57},
	pages = {218--237},
	year = {2019},
}

@article{JSC_Liao_2024,
	author = {H Liao and X Zhu and H Sun},
	title = {Asymptotically Compatible Energy and Dissipation Law
	of the Nonuniform ${L}2$-$1_{\sigma}$ Scheme for Time Fractional
	{A}llen--{C}ahn Model},
	journal = {J. Sci. Comput.},
	volume = {99},
	pages = {46},
	year = {2024},
}

@article{JCP_Alikhanov_2015,
	author = {A.A Alikhanov},
	title = {A new difference scheme for the time fractional diffusion equation},
	journal = {J. Comput. Phys.},
	volume = {280},
	pages = {424--438},
	year = {2015},
}

@article{CiCP_Liao_2021,
	author = {H Liao and W Mclean and J Zhang},
	title = {A second-order scheme with nonuniform time steps for a linear reaction-subdiffusion problem},
	journal = { Commun. Comput. Phys.},
	volume = {30},
	pages = {567--601},
	year = {2021},
}

@article{JCP_Liao_2020,
	author = {H Liao and T Tang and T Zhou},
	title = {A second-order and nonuniform time-stepping
	maximum-principle preserving scheme for time-fractional
	{A}llen--{C}ahn equations},
	journal = { J. Comput. Phys.},
	volume = {414},
	pages = {109473},
	year = {2020},
}

@article{SISC_Tang_2019,
	author = {Tang, Tao and Yu, Haijun and Zhou, Tao},
	title = {On Energy Dissipation Theory and Numerical Stability for Time-Fractional Phase-Field Equations},
	journal = {SIAM J. Sci. Comput.},
	volume = {41},
	number = {6},
	pages = {A3757-A3778},
	year = {2019},
}

@article{JSC_Du_2020,
	author = {Du, Qiang and Yang, Jiang and Zhou, Zhi},
	title = {Time-Fractional {A}llen--{C}ahn Equations: Analysis and Numerical Methods},
	journal = {J. Sci. Comput.},
	volume = {85},
	pages = {42},
	year = {2020},
}

@Article{SINUM_Liao_2018,
	title =	 {Sharp error estimate of the nonuniform ${L}1$ formula for linear reaction-subdiffusion equations},
	author =	 {Honglin Liao and Dongfang Li and Jiwei Zhang},
	journal =	 {SIAM J. Numer. Anal.},
	volume =	 56,
	year =	 2018,
	pages =	 {1112--1133},
}

@Article{SIREV_Du_2021,
	title =	 {Maximum bound principles for a class of semilinear parabolic equations
	and exponential time-differencing schemes},
	author =	 {Qiang Du and Lili Ju and Xiao Li and Zhonghua Qiao},
	journal =	 {SIAM Rev.},
	volume =	 63,
	year =	 2021,
	pages =	 {317--359},
}

@Article{JCM_Tang_2016,
	title =	 {Implicit-explicit scheme for the {A}llen--{C}ahn equation preserves the maximum principle},
	author =	 {T Tang and J Yang},
	journal =	 {J. Comput. Math.},
	volume =	 34,
	year =	 2016,
	pages =	 {471--481},
}

@Article{JSC_ZFLX_2025,
	title =	 {Energy Dissipation Law and Maximum Bound Principle-Preserving Linear {BDF2} Schemes with Variable Steps for the {A}llen--{C}ahn Equation},
	author =	 {Zhang, B. and Fu, H. and Lan, R. and Xie, S.},
	journal =	 {J. Sci. Comput.},
	volume =	 105,
	year =	 2025,
	pages =	 {51},
}

\end{document}